\documentclass[10pt]{amsart}

\usepackage{amssymb,a4wide
}
\usepackage{comment}
\usepackage{amscd}
\usepackage{amsthm}
\usepackage[english]{babel}
\usepackage[T1]{fontenc}
\usepackage[utf8]{inputenc}
\usepackage{mathrsfs}
\usepackage{tensor}
\usepackage{amsmath}
\usepackage{amssymb}
\usepackage{graphics}
\usepackage{graphicx}
\usepackage{color}

\newcommand{\weg}[1]{}
\newcommand{\bq}{\begin{equation}}
\newcommand{\eq}{\end{equation}}
\renewcommand{\d}{\mathrm{d}}
\newcommand{\dd}{\mathrm{d}}
\newcommand{\gr}{\mathrm{grad}}
\newcommand{\D}{\partial}

\newcommand{\R}{\mathbb{R}}
\newcommand{\C}{\mathbb{C}}

\newcommand{\tr}{\mathrm{tr}\,}

\newcommand{\Id}{\mathrm{Id}}

\newcommand{\CC}{\mathcal{C}}

\newtheorem{thm}{Theorem}[section]
\newtheorem{lem}[thm]{Lemma}
\newtheorem{cor}[thm]{Corollary}
\newtheorem{prop}[thm]{Proposition}
\theoremstyle{definition}
\newtheorem{rem}{Remark}[section]

\newtheorem{ex}{Example}

\newtheoremstyle{named}
 {\topsep}   
  {\topsep}   
  {\normalfont}  
  {0pt}       
  {\bfseries} 
  {.}         
  {5pt plus 1pt minus 1pt} 
{\thmname{#1}\thmnumber{ #2}\thmnote{ #3}}
\theoremstyle{named}
\newtheorem{remthm}{Remark}

\date{\today}

\title[C-projectively equivalent pseudo-K\"ahler metrics]{Local normal forms for 
c-projectively equivalent  metrics and  proof of the Yano-Obata conjecture 
in arbitrary signature.  Proof of the projective Lichnerowicz conjecture  for Lorentzian metrics} 

\author[A.V. Bolsinov, V.S. Matveev, and S. Rosemann]{Alexey V. Bolsinov, Vladimir S. Matveev, and Stefan Rosemann}
\numberwithin{equation}{section}

\address{Department of Mathematical Sciences,
 Loughborough University,
 LE11 3TU, UK }
 \email{A.Bolsinov@lboro.ac.uk} 
\address{Institute of Mathematics, Friedrich-Schiller-Universit\"at Jena, Jena, Germany.}
\email{vladimir.matveev@uni-jena.de}
\address{Institute of Mathematics, Friedrich-Schiller-Universit\"at Jena, Jena, Germany.}
\email{stefan.rosemann@uni-jena.de}

\begin{document}

\begin{abstract}
Two K\"ahler metrics on a complex manifold are called c-projectively equivalent if their $J$-planar curves coincide. These curves are defined by the property that the acceleration is complex proportional to the velocity. 
We give an explicit local description of all pairs of c-projectively equivalent K\"ahler metrics of arbitrary signature and  use  this description  to prove  the classical Yano-Obata conjecture: 
we show that on a closed connected K\"ahler manifold of arbitrary signature, any c-projective vector field is an affine vector field unless the manifold is $\C P^n$  with (a multiple of) the  Fubini-Study metric. As a by-product, we prove the projective Lichnerowicz conjecture for metrics of Lorentzian signature: we show that on a closed connected Lorentzian  manifold, any projective vector field is an affine vector field. 
\end{abstract}

\maketitle

\tableofcontents

\section{Introduction} 

\subsection{Definitions and description of results}

Let $(M,g,J)$ be a  K\"ahler manifold of arbitrary signature of real dimension $2n\ge 4$. We denote by $\nabla$ the Levi-Civita connection of  $g$ and let $\omega=g(J\cdot,\cdot)$ denote the K\"ahler form. All objects under consideration are assumed to be sufficiently  smooth.

A regular curve $\gamma: \R\supseteq I \to M$ is called \emph{$J$-planar} if there exist functions $\alpha,\beta:I\rightarrow\mathbb{R}$ such that 
\begin{equation}
\label{eq:eq1} 
\nabla_{\dot \gamma(t)}\dot \gamma(t)=\alpha\dot \gamma(t)+\beta J(\dot \gamma(t))\mbox{ for all } t\in I,
\end{equation} 
where $\dot \gamma =\tfrac{\d}{\d t} \gamma$.

From the definition we see  immediately  that the property of $J$-planarity is independent of the parameterisation of the curve, and that geodesics are $J$-planar curves.   
We also see that  $J$-planar curves form a much bigger family than the family of geodesics;  at every point and in every direction there exist 
infinitely many geometrically different  $J$-planar curves.

Two metrics $g$ and $\hat g$ of arbitrary signature  that are K\"ahler w.r.t the same complex structure $J$ are \emph{c-projectively equivalent} if any $J$-planar curve of $g$ is a $J$-planar curve of $\hat g$.  Actually, the condition that the metrics are K\"ahler with respect to the same complex structure is not essential; it is an easy exercise to show that if any $J$-planar curve of a K\"ahler structure $(g, J)$ is a $\hat J$-planar  curve of another K\"ahler structure 
$(\hat g, \hat J)$, then $\hat J=\pm J$.

C-projective equivalence was  introduced (under the name ``h-projective equivalence'')  by  Otsuki and Tashiro in  \cite{Otsuki,Tashiro}. Their motivation was to generalise the notion of 
 projective equivalence  to the K\"ahler situation. Since the notion of projective equivalence plays an essential role in our paper let us recall it. 
Two metrics $g$ and $\hat g$ of arbitrary signature are \emph{projectively equivalent}, if each $g$-geodesic is, up to an appropriate  reparameterisation, 
  a  $\hat g$-geodesic. 
  
Otsuki and Tashiro have shown that projective equivalence is not interesting in the  K\"ahler situation, since only simple examples are possible, and suggested c-projective equivalence as an  interesting object of study instead. This suggestion appeared to be very fruitful and between the 1960s and the 1970s, the theory of c-projectively equivalent metrics and c-projective transformations  was  one of the main research topics  in Japanese and Soviet (mostly Odessa and Kazan) differential geometry schools. For a collection of results of these times, see for example the survey \cite{Mikes} with more than 150 references. Moreover, two classical  books  \cite{Sinjukov,Yanobook} contain chapters on c-projectively equivalent metrics and connections.

Relatively recently c-projective equivalence was re-introduced, under different names and  because of different motivation. 
In fact, c-projectively  equivalent metrics are essentially the same as 
 { hamiltonian $2$-forms}, defined  and investigated in Apostolov {\it et al.}  \cite{ApostolovI,ApostolovII,ApostolovIII,ApostolovIV}. Though the definition of hamiltonian 2-forms is visually  different from that of c-projectively equivalent metrics,  the  defining equation \cite[equation $(12)$]{ApostolovI} of a hamiltonian $2$-form is algebraically equivalent to 
  a reformulation (see \eqref{eq:main} below) of the condition ``$\hat g$ is c-projectively equivalent to $g$'' into the language of PDE. The motivation of Apostolov {\it et al.}  to study hamiltonian $2$-forms is different from that of Otsuki and Tashiro.  Roughly speaking, in \cite{ApostolovI,ApostolovII} Apostolov {\it et al.}  observe that many interesting problems in K\"ahler geometry lead to  hamiltonian $2$-forms and suggest studying them. The motivation is justified in \cite{ApostolovIII,ApostolovIV}, where the authors indeed construct interesting and useful examples of K\"ahler manifolds.  In dimension $\ge 6$, c-projectively equivalent metrics are also essentially the same as hermitian conformal Killing  (or twistor) $(1,1)$-forms  studied in \cite{Moroianu,Semmelmann0,Semmelmann1}, see  \cite[Appendix A]{ApostolovI} or \cite[\S 1.3]{MatRos}  for details. Finally, such metrics are closely related to the so-called K\"ahler-Liouville integrable systems of type $A$  introduced by Kiyohara {\it et al.} in \cite{Kiyohara2010}. 
 
Most recent  interest in c-projectively equivalent metrics is due to the observation that one can study them with the help of so-called parabolic geometry. We refer to \cite{CEMN} for details. Certain non-trivial  applications of the methods of  parabolic geometry in the theory of c-projective transformations are  in \cite{melnik}.
 
 Our paper contains three  main results. The first result is a local description (near a generic point) of c-projectively equivalent K\"ahler metrics of arbitrary signature. If $g$ is positive  definite, such a description follows from the local description of hamiltonian 2-forms due to  Apostolov {\it et al.} \cite{ApostolovI}.
 Although the precise statements  are slightly lengthy, we indeed provide an explicit description of the components of the metrics and  of the K\"ahler form $\omega= g(J\cdot, \cdot)$. 
 The parameters in this description are almost arbitrary numbers and functions of one variable  
and, in certain cases, almost arbitrary affinely equivalent K\"ahler metrics of smaller dimension (note that 
the description of affinely equivalent K\"ahler metrics was recently obtained by Boubel in \cite{Boubel}). 

 It is hard to overestimate the future  role of a local  description in the local and global theory of c-projectively  equivalent  metrics. Almost all known local results can easily
be proved using it. Roughly speaking, using the local description, one can reduce any problem
that can be stated using geometric PDEs (for example, any problem involving the curvature) to
the analysis of a system of ODEs.  As we mentioned above,  in the positive definite case, 
the description of c-projectively equivalent metrics in the language of hamiltonian 2-forms is due  
to Apostolov {\it et al.} \cite{ApostolovI}, and  with the help of such a 
 description they did a lot. In particular they  described possible  topologies  of closed manifolds admitting c-projectively equivalent metrics and  constructed new examples of Einstein and extremal metrics on closed manifolds, see \cite{ApostolovII,ApostolovIII,ApostolovIV}. We expect similar applications of our description  as well and, in particular, plan to study the topology of c-projectively equivalent closed K\"ahler manifolds of arbitrary signature in further papers.

A demonstration of the importance of the local description  is our  second main result, which is a proof of  the natural generalisation of the Yano-Obata conjecture
for K\"ahler manifolds of arbitrary signature.  A vector field on a K\"ahler manifold   is called \emph{c-projective} if its local flow  sends $J$-planar curves to $J$-planar curves, and \emph{affine} if its local flow preserves the Levi-Civita connection.

\begin{thm}[Yano-Obata conjecture]
\label{thm:yano_obata}
Let $(M,g,J)$ be a closed connected  K\"ahler manifold of arbitrary signature and 
of real dimension $2n\geq 4$ such that it admits a c-projective vector field that is not an affine vector field. Then the manifold is isometric to  $(\mathbb{C}P^n,c\cdot g_{FS},J_{\mathrm{standard}})$ for some non-zero  constant $c$, where $g_{FS}$ is the Fubini-Study metric.
\end{thm}

For positive definite metrics, Theorem~\ref{thm:yano_obata} was  
 proved in \cite{YanoObata},  where also a history  including  a list of previously proven special cases can be found.  The 4-dimensional version of Theorem~\ref{thm:yano_obata} was recently proved in \cite{BMMR}. 

We see that a closed K\"ahler manifold with a non-affine c-projective vector field has definite signature. 
This phenomenon is, of course, essentially global since locally  we can construct  counterexamples in any signature.
In dimension  4,  such examples are described in  \cite{BMMR}, and  in Proposition~\ref{prop:localclassgvFi} we explicitly construct  K\"ahler metrics  of any dimension and any signature admitting non-affine c-projective vector fields. Let us also mention   (see, e.g. \cite[Example 2]{YanoObata}) that  $(\mathbb{C}P^n,c\cdot g_{FS},J_{\mathrm{standard}})$
 admits many non-affine c-projective vector fields.

As a by-product of our proof of the Yano-Obata conjecture (we explain in the next section why it is a by-product), we prove the  possibly more popular {\it projective  Lichnerowicz} conjecture for metrics of Lorentzian signature. Recall that a vector field is \emph{projective} with respect to a (arbitrary,  not necessarily K\"ahler) metric $g$, if its local flow sends geodesics viewed as unparameterised curves to geodesics.

\begin{thm}[Projective  Lichnerowicz  conjecture for metrics of Lorentzian signature]
\label{thm:lichnerowicz} 
Let $(M,g)$ be a closed connected Lorentzian  manifold  
of  dimension $n\geq 2$.  Then any projective vector field on $M$ is an   affine vector field.
\end{thm}

For Riemannian metrics, the analogue of Theorem~\ref{thm:lichnerowicz} was proved in \cite{CMH} (dimension 2)  and \cite{Matveev2007} (dimension greater that 2 -- this paper also contains a historical overview  and  a  
  list of previously known special cases), see also \cite{zeghib}.  In Japanese mathematics, this statement, at least in the Riemannian setting,  is also known as \emph{ projective  Obata conjecture}
  and was published  many times as an  important  conjecture.   For 2-dimensional   Lorentzian manifolds, Theorem~\ref{thm:lichnerowicz}    was proved in \cite{japan}.

We would like to  emphasize here that our proofs of the Yano-Obata and Lichnerowicz conjectures are not generalisations of the proofs from \cite{CMH,Matveev2007,japan,YanoObata,zeghib}, and are based on a different circle of ideas. In general, it is    difficult
to extend global statements about Riemannian metrics to the pseudo-Riemannian setting, since many ``global'' methods  require   definiteness of the metrics. This is also the case in our situation;  the main ingredients of the proofs of \cite{Matveev2007,CMH,YanoObata,zeghib} are the global ordering of the eigenvalues of the endomorphisms  $A$  and $L$ (given by \eqref{eq:defA} and  \eqref{eq:defL} -- these endomorphisms   play an important role in our paper), and an  investigation of the behaviour of the sectional  and holomorphic sectional   curvature along the integral curves   of the projective and c-projective  vector fields (in \cite{Matveev2007,  YanoObata}  it is  shown
 that these curvatures are either constant or unbounded). Both ingredients do not exist in the  case of indefinite signature; examples show that  the eigenvalues of $A$ (resp. $L$) are not globally ordered anymore, and  the unboundedness of 
the sectional curvature does not lead to  any  contradiction, since for metrics of indefinite signature the sectional or holomorphic sectional curvature is either constant or unbounded. Moreover, as  follows from our  calculations  in \S\ref{ssec:l23}, in the indefinite case, 
all  curvature invariants along  the integral curves   of the projective and c-projective  vector fields  can be bounded.  Other proofs of special cases of the Yano-Obata and Lichnerowicz conjecture (see e.g. \cite{Yamauchi1,Yano1981}) 
are based on  the Bochner technique, which  also requires that the metric is  definite.

\subsection{Main idea} \label{sec:idea} 

The local description of c-projectively equivalent metrics will be given in 
Theorem~\ref{thm:localclassification} in \S\ref{sec:lc} (which does not require this  paragraph so a 
hurried reader can directly   go there).  The goal  of this section is to explain the
 main idea of our solution. We hope that this allows the reader to see the geometry behind the formulas 
and also  may be used in many other problems related to c-projectively equivalent metrics. 

Experts always expected that  projectively equivalent metrics must have a close relation with 
c-projectively equivalent metrics. The expectation  is based on the following informal observation: 
most mathematicians that studied c-projectively  equivalent  metrics and c-projective vector fields   
studied  projectively equivalent metrics and projective vector fields before. 
It appears that many ideas and many results in the theory of projectively equivalent metrics
 have their counterparts in the c-projective setting, though most of the proofs in the c-projective 
setting are longer  and are more involved than their projective analogues.\footnote{
This analogy between c-projective and projective geometry fails  on the level of affine connections
 (note that the definition of c-projective equivalence makes also sense for affine connections which 
are not necessarily Levi-Civita connections): 
though both affine projective and affine c-projective geometries are parabolic geometries,
 there are essential differences between these theories if only connections are involved, see e.g. \cite{CEMN}. }

 We suggest an explanation why the theories are closely related, which is simultaneously  
 the main idea of our description.    The following observation, which we
 formalise (and give a self-contained proof) 
 in  \S\ref{sec:killing}, is crucial:  c-projectively, but not affinely  equivalent metrics 
 $g$ and $\hat g$ allow us to construct  vector fields $K_1,\dots,K_\ell$
 which preserve the complex structure and which are Killing 
with respect to both metrics. 
For hamiltonian 2-forms (at least for a positive definite metric), 
the existence of these Killing vector fields was shown by Apostolov {\it et al.} \cite{ApostolovI}, and in the framework of 
K\"ahler-Liouville manifolds (under certain non-degeneracy assumptions) their existence was 
observed by Kiyohara and Topalov \cite{Kiyohara2010}. 

We consider the local action of these vector fields and the local quotient $Q$ of $M$ 
with respect to this action (it will be shown that such a quotient is  well-defined near a generic point). 
Let us denote the quotient metrics by $g_{Q}$ and $\hat g_{Q}$. 
Notice  that $Q$ is not a K\"ahler  quotient and the metrics $g_{Q}$ and $\hat g_{Q}$  are in general not K\"ahler.  

\vspace{1ex} 
{\bf   Main Observation.}  The following statements hold: 

\begin{enumerate}

\item  $g_{Q}$ and $\hat g_{Q}$ are projectively equivalent;

  \item  the metrics $g$ and $\hat g$ can be reconstructed from $g_{Q}$ and $\hat g_{Q}$ 
   in a relatively straightforward way. 

\end{enumerate}

Recently, projectively equivalent metrics have been explicitly locally described in \cite{BM}. 
We obtain our description of c-projectively equivalent metrics by
taking the formulas from \cite{BM} for the quotient metrics $g_{Q}$ and $\hat g_{Q}$ and then ``reconstructing''
$g$ and $\hat g$.

However, not  every pair of projectively equivalent metrics $g_{Q} $, $\hat g_{Q}$
 as considered in \cite{BM} can be obtained    from a pair $g$, $\hat g$ of c-projectively 
equivalent metrics: we will describe the conditions that $g_{Q}$ and $\hat g_{Q}$ have to satisfy 
in order to arise as quotients from c-projectively equivalent metrics. 
These additional conditions actually simplify the formulas for the metrics $g_{Q} $, $\hat g_{Q}$  as 
compared to the formulas from  \cite{BM} for the general case. 
Moreover, we show, assuming these conditions are satisfied, 
  how to effectively reconstruct the initial metrics $g$ and $\hat g$. 
	This yields our description of c-projectively equivalent K\"ahler metrics.

The relation between projectively and c-projectively  equivalent metrics plays also an 
important role in the proof of the Yano-Obata conjecture. We will see that under the additional assumption 
that the degree of mobility is $2$ (which means that the ``space of c-projectively equivalent metrics'' 
is two-dimensional -- the formal definition is in \S\ref{conv} where it is also explained why   it is 
 the most   non-trivial case  in the proof of the Yano-Obata conjecture), 
a c-projective vector field on the initial manifold reduces to a projective vector field on the quotient. 
 
We expect further applications of this observation  which suggests, in the metric setting, 
 an almost algorithmic way to produce results in c-projective geometry from results in projective geometry 
and the latter is much better  developed.

 Unfortunately, this almost algorithmic way does not automatically work in the other (c-projective $\to $  projective) direction. The reason is that the quotient metrics $g_{Q}$ and $\hat g_{Q}$, as already noticed, satisfy certain additional conditions.  The most important of them is as follows:  for the metrics $h=g_Q$ and $\hat h= \hat g_Q$  the endomorphism $L$ given by   \eqref{eq:defL} below  has no Jordan blocks with nonconstant eigevalues.  
 For general projectively equivalent metrics,  $L$  may have non-trivial Jordan blocks with nonconstant eigevalues. This is   the only reason  why we can not  modify the proof of the Yano-Obata conjecture to obtain the proof of  the   projective Lichnerowicz conjectures {\it for metrics  of all signatures}.   For the metrics {\it of Lorentzian signature}, at most  one non-trivial Jordan block may occur and after some additional  work in \S~\ref{sec:Jblock} we exclude this case in the proof of the projective  Lichnerowicz conjecture. The rest  of the proof of the projective  Lichnerowicz conjecture  is a straightforward modification (actually, a simplification) of the   proof of the Yano-Obata conjecture and when proving the {projective} Lichnerowicz conjecture in the 
 ``no-Jordan-blocks'' case (Theorem \ref{specialcase2}), we confine ourselves with a series of remarks  explaining necessary amendments.

\subsection{Local description of c-projectively equivalent metrics} 
\label{sec:lc}

Let $(M,g,J)$ be a K\"ahler manifold of real dimension $2n\geq 4$ and let $\nabla$ and $\omega=g(J\cdot,\cdot)$ denote the Levi-Civita 
connection and K\"ahler form respectively. 
We do not require that $g$ or any other K\"ahler metric that appears has positive signature. 

Instead of the pair $(g,\hat g)$ of c-projectively equivalent metrics it is appropriate 
to consider the pair $(g,A)$, where $A:TM\rightarrow TM$ is a  hermitian (i.e. $g$-selfadjoint and $J$-commuting) 
endomorphism  constructed from $g$ and $\hat{g}$ by 
\begin{align}
A=A(g,\hat{g})=\left(\frac{\mathrm{det}\,\hat{g}}{\mathrm{det}\,g}\right)^{\frac{1}{2(n+1)}}\hat{g}^{-1}g .\label{eq:defA}
\end{align}
In this formula, we view $g,\hat g:TM\rightarrow T^*M$ as bundle isomorphisms. In tensor notation (with summation convention in force), 
$$
A^i_j=\left(\frac{\mathrm{det}\,\hat{g}}{\mathrm{det}\,g}\right)^{\frac{1}{2(n+1)}}\hat{g}^{ik}g_{kj},
$$
where $\hat g^{ij}$ denotes the inverse to $\hat g_{ij}$, i.e. $\hat g^{ik}\hat g_{kj}=\delta^i_j$.

Clearly, one can reconstruct $\hat g$ from the pair $(g,A)$ and obtains
\begin{align}
\hat{g}=(\det A)^{-\frac{1}{2}}g(A^{-1}\cdot,\cdot).\label{eq:defhatg}
\end{align}
The endomorphism  $A$, introduced in  \cite{DomMik},  plays an important role in the theory of 
c-projectively equivalent metrics. One of the reasons for this  is that the condition that $g$ and $\hat{g}$ 
are c-projectively equivalent amounts to the fact that the tensor $A$ satisfies the \emph{linear} partial differential equation
\begin{align}
\label{eq:main}
\nabla_X A=X^\flat \otimes \Lambda+\Lambda^\flat\otimes X+(JX)^\flat \otimes J\Lambda+(J\Lambda)^\flat\otimes JX,
\end{align}
for all $X\in TM$, where $\Lambda=\frac{1}{4}\gr(\tr A)$ and $X^\flat=g(X,\cdot)$.  We say that $g$ and $A$ are {\it compatible in the c-projective sense} 
or just {\it c-compatible} if $A$ is a hermitian endomorphism solving \eqref{eq:main}.
In particular, any  hermitian endomorphism $A$ with nowhere vanishing determinant  and c-compatible with $g$ gives us  a 
c-projectively equivalent metric $\hat g$  by \eqref{eq:defhatg}, this metric is automatically K\"ahler with respect to $J$. 

Another reason for working with $A$ instead of $\hat g$ is that in our local description, 
the formulas for $(g,A)$ are much simpler than those for $(g, \hat g)$.

We describe locally all K\"ahler structures $(g,J,\omega)$ admitting solutions $A$ to \eqref{eq:main}
in  Theorem~\ref{thm:localclassification} below. Since the description  is relatively complicated, 
we first consider two special cases corresponding to the ``weakest'' (Theorem~\ref{ex:constant}) and 
``strongest''  (Theorem~\ref{thm:nonconstant}) case of c-projective equivalence. 

Note that  any parallel hermitian endomorphism $A$ (i.e.,  satisfying $\nabla A=0$), in particular the identity 
$\Id:TM\rightarrow TM$, is a solution to \eqref{eq:main}. Such solutions  correspond to  K\"ahler metrics $\hat g$ 
which are affinely equivalent to $g$, i.e., which have the same Levi-Civita connection as $g$.

\begin{thm}
[Well-known special case of Theorem~\ref{thm:localclassification}]
\label{ex:constant}
Let $(M,g,J)$ be a K\"ahler manifold of arbitrary signature and  $A:TM\rightarrow TM$ a parallel hermitian endomorphism. 
Then locally $(M,g,J)$ is a direct product of K\"ahler manifolds $(M_\gamma,g_\gamma,J_\gamma)$, $\gamma=1,\dots,N$, and $A$ decomposes as 
$A=A_1+\dots+A_N$, where $A_\gamma:TM_\gamma\rightarrow TM_\gamma$ is a parallel hermitian endomorphism on $(M_\gamma,g_\gamma,J_\gamma)$ 
having either a single real eigenvalue $c_\gamma$ or a pair of complex-conjugate eigenvalues $c_\gamma,\bar c_\gamma$.
\end{thm}

The above theorem  is just the de Rham--Wu decomposition \cite{deRham,Wu} of the K\"ahler manifold into components 
corresponding to the parallel distributions given by the generalised eigenspaces of $A$. 
This  is not a complete description  of pairs $((g,J), A)$, where $(g,J)$ is K\"ahler and $A$ is a parallel hermitian endomorphism: 
what is left is an explicit description of the blocks $(g_\gamma,J_\gamma)$ and $A_\gamma$. 
In the positive definite case, the description of these blocks is trivial since in this case  $A_\gamma$ is a constant multiple of 
$\mathrm{Id}:TM_\gamma\rightarrow TM_\gamma$. If the signature of $g_\gamma$ is arbitrary, 
the local description of $(g_\gamma, J_\gamma)$ and $A_\gamma$ has recently been obtained by C. Boubel in \cite{Boubel}.  
Boubel's  description of  $(g_\gamma, J_\gamma,A_\gamma)$ is quite complicated,  we will not repeat it here and     
refer to \cite{Boubel} for  more details. 

\begin{rem} 
Let us reformulate the statement from Theorem~\ref{ex:constant} in matrix notation: 
we can find local coordinates such that the matrices  of $g$, $J$ and $A$ in these coordinates 
are block-diagonal with the same structure of blocks:
\begin{equation}
\label{matrices:constant} 
g= 
\begin{pmatrix} 
g_1&& \\ 
&\ddots & \\ 
&&g_N
\end{pmatrix}\, , \quad
J= 
\begin{pmatrix} 
J_1&& \\ 
&\ddots & \\ &&J_N
\end{pmatrix}
\, , \quad 
A= 
\begin{pmatrix} 
A_1&& \\ 
&\ddots & \\ &&A_N
\end{pmatrix}.   
\end{equation} 
In all matrices, the components of each block only depend on the corresponding coordinates and 
for each $\gamma=1,\dots,N$ the endomorphism $A_\gamma$ is hermitian and parallel w.r.t. the K\"ahler structure $(g_\gamma, J_\gamma)$. 
\end{rem}

The ``main idea'' and ``main observation'' described in \S\ref{sec:idea} become vacuous in the setting of  Theorem~\ref{ex:constant}: 
the number of ``canonical''  Killing vector fields $K_1,\dots,K_\ell$ is zero, hence, the quotient of the manifold is 
the manifold itself.  The ``main observation'' remains, of course, formally true but in this case projective equivalence 
is affine equivalence. 

A special feature  of the situation described in Theorem~\ref{ex:constant} is that the eigenvalues of $A$ are constant, 
and may have high multiplicities.  Let us now consider  the ``strongest'' special case of c-projective equivalence: 
 all  eigenvalues of $A$ are non-constant (when considered as functions on $M$), and  their  multiplicity is minimal possible.

Consider two projectively equivalent pseudo-Riemannian metrics $h$ and $\hat h$ (i.e., metrics having the same unparametrised geodesics) 
and define the endomorphism  $L$  by
\begin{align}
\label{eq:defL}
L=L(h,\hat{h})=\left|\frac{\mathrm{det}\,\hat{h}}{\mathrm{det}\,h}\right|^{\frac{1}{n+1}}\hat{h}^{-1}h.
\end{align}

It is well known that $L$ satisfies the equation
\begin{equation} 
\label{eq:main:proj} 
\nabla_X L=X^\flat \otimes \Lambda+\Lambda^\flat\otimes X,  \quad \mbox{for all $X\in TM$},
\end{equation} 
where $\Lambda=\frac{1}{2}\gr(\tr L)$, $X^\flat=h(X,\cdot)$  and $\nabla$ denotes the Levi-Civita connection of $h$.
Moreover, if $L$ is $h$-selfadjoint and non-degenerate, then  \eqref{eq:main:proj}  is equivalent to the projective equivalence of $h$ and  $\hat h$ given by 
\begin{equation}
\label{eq:defhath}
\hat h=|\mathrm{det}\,L|^{-1}h(L^{-1}\cdot,\cdot)
\end{equation}
see \cite{Sinjukov} and e.g. \cite{BMgluing}.   
To emphasize both the difference and similarity with  c-compatibility introduced above, we will say that $h$ and an  
$h$-selfadjoint endomorphism $L$ satisfying \eqref{eq:main:proj}  are  {\it compatible in the projective sense} or just {\it compatible}.

\begin{ex} 
\label{ex:modelB}
Assume that on a certain domain $U\subset \R^\ell$ we have a compatible pair $h$ and $L$ for which the following conditions hold:

\begin{itemize}

\item[A1.] The eigenvalues $\rho_1, \dots, \rho_\ell$ of $L$ are all distinct at each point of $U$ 
(complex conjugate pairs $\rho$, $\bar\rho$ with $\mathrm{Im}\,\rho\ne 0$ are allowed too), which allows us to view them as  smooth functions on $U$;  

\item[A2.] $\d\rho_i\ne 0$ at each point of $U$, $i=1,\dots,\ell$.  

\end{itemize} 

We now explain how, starting from such a compatible  pair 
\begin{equation}
\label{eq:handL}
h=\sum_{i,j=1}^\ell B_{ij}(x)\d x_i\d x_j \quad \mbox{and} \quad L = \sum_{i,j=1}^\ell L^i_j  \d x_j\otimes\frac{\D}{\partial x_i},
\end{equation}
one can naturally construct a c-compatible pair, i.e., a K\"ahler structure $(g,  J, \omega)$ and a 
hermitian endomorphism $A$ satisfying \eqref{eq:main}.  By $\mu_1,\dots,\mu_\ell$ we denote the elementary 
symmetric polynomials in $\rho_1,\dots, \rho_\ell$ (i.e., $(\tau +\rho_1)...(\tau+\rho_\ell)= \tau^\ell + \mu_1 \tau^{\ell-1}+...+\mu^\ell$.)  Notice that under the assumption that $\rho_i$ are 
all distinct and $\d\rho_i\ne 0$, both systems of functions $\rho_i$'s and $\mu_i$'s can be considered 
as local coordinates on $U$.

Consider a domain $V\subset \mathbb{R}^\ell$ with local coordinates $t_1,\dots, t_\ell$ and define $g$, $\omega$ on $V\times U$ in the following way
\begin{equation}
\label{eq:ex1gandomega}
\begin{aligned}
g&=\sum_{\alpha,\beta=1}^\ell H_{\alpha\beta} (x) \d t_\alpha \d t_\beta + \sum_{i,j=1}^\ell B_{ij}(x)\d x_i \d x_j, \\
\omega&= \sum_{\alpha=1}^\ell \d\mu_\alpha\wedge \d t_\alpha,
\end{aligned}
\end{equation}
where $H_{\alpha\beta} = \sum_{ij} B^{ij} (x)\dfrac{\partial \mu_\alpha}{\partial x_i} 
\dfrac{\partial \mu_\beta}{\partial x_j}$ and $B^{ij}$ are the components of the matrix inverse to $B_{ij}$, i.e., $\sum_{k} B_{ik}B^{kj}=\delta^j_i$. 
We also set
\begin{equation}
\label{eq:ex1A}
A=\sum_{\alpha,\beta=1}^\ell  M^{\beta}_\alpha (x) \d t_\beta \otimes \frac{\partial}{ \partial {t_\alpha}} + \sum_{i,j=1}^\ell L^i_j (x) \d x_j \otimes \frac{\partial} {\partial {x_i}},
\end{equation}
where 
$M^{\beta}_\alpha = \sum_{i,j}L^i_j \dfrac{\partial \mu_\beta}{\partial x_i} \dfrac{\partial x_j}{\partial \mu_\alpha}$.

Equivalently, in matrix form w.r.t. the coordinates $t_1,\dots,t_\ell,x_1,\dots,x_\ell$, the above expressions take the form
\begin{equation}
\label{eq:gomegaAmatrix}
g = \begin{pmatrix}  
P h^{-1}P^\top & 0 \\ 0 & h
\end{pmatrix}, \quad 
\omega = 
\begin{pmatrix}
0 & -P \\ 
P^\top & 0
\end{pmatrix}, \quad
A = 
\begin{pmatrix}
(P L P^{-1})^\top & 0 \\
0 &  L
\end{pmatrix}
\end{equation}
where $P=\left(  \dfrac{\partial \mu_\alpha}{\partial x_i}  \right)$ is the Jacobi matrix of the system of 
functions $\mu_1,\dots, \mu_\ell$  (w.r.t. the local coordinates $x_1,\dots, x_\ell$).  
\end{ex}

The following theorem, which describes c-projectively equivalent metrics under the assumption that $A(g, \hat g)$ 
has the maximal number of non-constant eigenvalues,  shows that in this case the relation between projective equivalence 
and c-projective equivalence is rather straightforward.

\begin{thm}
\label{thm:nonconstant}
Let  $(h, L)$ be a compatible pair on $U$ satisfying  {\rm A1} and {\rm A2}. Then the above formulas \eqref{eq:ex1gandomega} and \eqref{eq:ex1A} (or equivalently \eqref{eq:gomegaAmatrix} in matrix form) define a K\"ahler structure $(g,\omega)$ and a hermitian endomorphism $A$ which are c-compatible, i.e., satisfy \eqref{eq:main}.   Conversely, if a K\"ahler structure $(g,\omega)$ and a hermitian endomorphism $A$ are c-compatible and the eigenvalues of $A$ (as a complex endomorphism) satisfy {\rm A1} and {\rm A2} in a neighbourhood of some point, then locally, in a neighbourhood of this point, $g, \omega$ and $A$ can be written in the form \eqref{eq:ex1gandomega} and \eqref{eq:ex1A}, where $h=\sum_{i,j} B_{ij}(x)\d x_i\d x_j$ and $L = \sum_{i,j} L^i_j  \d x_j\otimes\partial_{x_i}$ are compatible.
\end{thm}

\begin{ex}
The simplest example of the situation described in Theorem \ref{thm:nonconstant} is obtained by starting with a 
2-dimensional compatible pair $(h,L)$ such that $L$ has two real non-constant eigenvalues $\rho,\sigma$ 
satisfying {\rm A1} and {\rm A2}. The description of such a pair is due to Dini \cite{Dini}, see also \cite{BM}: locally, we find coordinates $x,y$ 
such that  $\rho=\rho(x)$ and $\sigma=\sigma(y)$ and
$$
h=(\rho-\sigma)(\d x^2\pm\d y^2),\quad L=\rho\d x\otimes \frac{\D}{\D x}+\sigma\d y\otimes \frac{\D}{\D y}.
$$
Applying Theorem \ref{thm:nonconstant} to these formulas, we obtain the 
formulas for the K\"ahler structure $(g,\omega)$ and the c-compatible endomorphism $A$. These formulas can be found in  \cite[(3.1) and (3.2)]{BMMR}.
\end{ex}

We see  that in the situation of Theorem~\ref{thm:nonconstant}, the entries  of $g$, $\omega$ and $A$ do not  
depend on the coordinates $t_1,\dots,t_\ell$. This  implies that $\tfrac{\partial}{\partial t_1},\dots,\tfrac{\partial}{\partial t_\ell}$ 
are $J$-preserving  Killing vector fields, and they are   precisely the Killing vector fields $K_1,\dots,K_\ell$   
which we mentioned in \S\ref{sec:idea}.  The quotient with respect to the local action of these vector fields is  $n$-dimensional with local coordinates 
$x_1,\dots, x_\ell$,  and the metric $g$ descends to the metric $g_{Q}=h$ on the quotient. 
As claimed in the ``main observation'' of \S\ref{sec:idea},  $g_{Q}$ admits a projectively equivalent metric $\hat g_{Q}$ defined by the endomorphism $L$ which also can be treated as the quotient $L=A_Q$ of the hermitian endomorphism $A$.  

In the next example and theorem, we present the most general local expression which a K\"ahler structure $(g,\omega)$ together with 
a solution $A$ to equation \eqref{eq:main} can take. The construction below combines the previous two cases from  Theorem~\ref{ex:constant} and Example~\ref{ex:modelB}.

\begin{ex}
\label{ex:now2}
We start with two ingredients:

\begin{itemize}

\item  a compatible pair $h$ and $L$ defined on a domain $U\subset \mathbb R^\ell$ and satisfying the conditions A1 and A2 as in Example~\ref{ex:modelB}, see \eqref{eq:handL};

\item a K\"ahler structure $(g_{\mathrm{c}}, \omega_{\mathrm{c}})$ defined on some domain $S$ with a parallel hermitian endomorphism $A_{\mathrm{c}}$  (notice that the eigenvalues of $A_{\mathrm{c}}$ are constant).

\end{itemize}  

In addition, we assume that the eigenvalues of $L$ at each point $p\in U$ are all different from those of $A_{\mathrm{c}}$.

Consider the direct product $V\times U\times S$, where $V\subset \R^\ell$ is a certain domain of the same dimension $\ell$ as $U$, 
and denote local coordinates on $V$, $U$ and $S$ by $(t_1,\dots, t_\ell)$, $(x_1,\dots, x_\ell)$ and $(y_1,\dots, y_{2k})$ respectively. On this product 
$V\times U\times S$, we now 
define a pseudo-Riemannian metric $g$ and a 2-form $\omega$:
\begin{equation}
\label{eq:gomegainv}
\begin{aligned}
g&=\sum_{\alpha,\beta=1}^\ell H_{\alpha\beta} (x) \theta_\alpha \theta_\beta + \sum_{i,j=1}^\ell B_{ij}(x)\d x_i \d x_j + 
 g_{\mathrm{c}} \bigl(\chi_L (A_{\mathrm{c}})\cdot,\cdot\bigr), \\
\omega &= \sum_{\alpha=1}^\ell \d\mu_\alpha\wedge \theta_\alpha + 
\omega_{\mathrm{c}} \bigl( \chi_L (A_{\mathrm{c}})\cdot,\cdot\bigr),
\end{aligned}
\end{equation}
where $\chi_L(t)=\det(t{\cdot} \mathrm{Id} - L)$ is the characteristic polynomial of $L$,  
$\theta_i = \d t_i + \alpha_i$ and the 1-forms $\alpha_i$ on $S$ are chosen in such a way that
$\d\alpha_i = (-1)^i \omega_{\mathrm{c}}(A_{\mathrm{c}}^{\ell-i}\cdot ,  \cdot)$ 
(which is possible since $\omega_{\mathrm{c}}(A_{\mathrm{c}}^{\ell-i}\cdot ,  \cdot)$ is a parallel 2-form on $S$). 
The other ingredients,  $H_{\alpha\beta}$ and $\mu_i$, are defined as above in Example~\ref{ex:modelB}  
and in addition we set  $\mu_0= 1$. 

Further we define the endomorphism
\begin{equation}
\label{eq:Ainv}
A=\sum_{\alpha,\beta=1}^\ell  M^{\beta}_\alpha (x) \, \theta_\beta  \otimes  \frac{\partial}{\partial {t_\alpha}}  + \sum_{i,j=1}^\ell L^i_j (x) \,\d x_j  \otimes \frac{\partial}{\partial {x_i}} 
+  \sum_{p,q=1}^{2k}    (A_{\mathrm{c}})^q_p  \,  \d y_p  \otimes \left( \frac{\partial}{\partial {y_q}}    - \sum_{i=1}^\ell\alpha_{iq} \frac{\partial}{\partial {t_i}}   \right)
\end{equation}
where 
 $M^{\beta}_\alpha = \sum_{i,j}L^i_j \dfrac{\partial \mu_\beta}{\partial x_i} \dfrac{\partial x_j}{\partial \mu_\alpha}$ and $\alpha_{iq}$ resp. $(A_{\mathrm{c}})^q_p$ denote the components
of $\alpha_i$ resp. $A_{\mathrm{c}}$ w.r.t. the coordinates $y_1,\dots,y_{2k}$, i.e., $\alpha_i = \sum_{q}  \alpha_{iq} \d y_q$ and
$A_{\mathrm{c}} = \sum_{p,q}    (A_{\mathrm{c}})^q_p  \,  \d y_p\otimes  \partial_{y_q}$. 

Equivalently, in matrix form (w.r.t. the  basis $\theta_1, \dots, \theta_\ell, \d x_1,\dots, \d x_\ell,  \d y_1,\dots, \d y_{2k}$), the above formulas take the form:
$$
g = 
\begin{pmatrix}  
P h^{-1}P^\top & 0 & 0 \\ 0 & h & 0  \\
0 & 0 &  g_{\mathrm{c}} {\cdot} \chi_L (A_{\mathrm{c}})
\end{pmatrix}, 
\quad 
\omega = 
\begin{pmatrix}
0 & -P  & 0 \\ P^\top & 0 & 0 \\ 0 & 0 &  \omega_{\mathrm{c}} {\cdot} \chi_L (A_{\mathrm{c}})
\end{pmatrix}, \quad
A = 
\begin{pmatrix}
(P L P^{-1})^\top & 0 & 0 \\0 &  L  &  0 \\ 0 & 0 & A_{\mathrm{c}}
\end{pmatrix}
$$
where $P=\left(  \dfrac{\partial \mu_\alpha}{\partial x_i}  \right)$ is the Jacobi matrix 
of the system of functions $\mu_1,\dots, \mu_\ell$  (w.r.t. the local coordinates $x_1,\dots, x_\ell$). 
\end{ex}

\begin{rem}
Each of the 1-forms $\alpha_i$ on $S$ is determined by $(\omega_{\mathrm {c}}, A_{\mathrm {c}})$  up to adding the differential
of a function. However, replacing $\theta_i$ by the 1-forms $\tilde \theta_i=\theta_i+\d f_i$ in the formulas of Example \ref{ex:now2}
for functions $f_i$ on $S$, it is easy to construct a local transformation $f:M\rightarrow M$ identifying the formulas in 
Example \ref{ex:now2} written down w.r.t. $\theta_i$ and $\tilde\theta_i$ respectively, 
see also the discussion after Proposition \ref{prop:locform} below.
\end{rem}

\vspace{1ex} 

We will call a point $p\in M$ \emph{regular} with respect to a solution $A$ of \eqref{eq:main}, 
if in a neighbourhood of this point  the number of different eigenvalues of $A$ is constant 
(which implies that the eigenvalues are smooth functions in some neighbourhood of $p$), and 
for each eigenvalue $\rho$ either $\d\rho\ne 0$, or $\rho$ is constant in a neighbourhood of $p$.   
Clearly, the  set $M^0$ of regular points  is open and dense in $M$.
Further (see Lemma~\ref{lem:properties}~\eqref{l3} below) we will see that the number of 
non-constant eigenvalues of $A$ is the same near every regular point. 
The following theorem generalises Theorems~\ref{ex:constant} and~\ref{thm:nonconstant}:

\begin{thm}
\label{thm:invform} 
The metric $g$ and 2-form $\omega$ defined    by  \eqref{eq:gomegainv} are a K\"ahler structure and $A$ defined by \eqref{eq:Ainv} is a hermitian solution of \eqref{eq:main}. 

Conversely, let $(M,g,\omega)$ be a  K\"ahler manifold of arbitrary signature and $A$ be a hermitian solution of 
\eqref{eq:main}.  Then in a neighbourhood of a regular  point,  the K\"ahler structure $(g, \omega)$ and the endomorphism $A$ can be written in the form  
\eqref{eq:gomegainv} and \eqref{eq:Ainv} from Example {\rm \ref{ex:now2}}. 
\end{thm}

\begin{ex}
The simplest example of the situation described in Theorem \ref{thm:invform} is obtained by starting with a 
1-dimensional compatible pair $h=\d x^2$, $L=\rho \d x\otimes \D_x$ for a function $\rho=\rho(x)$ satisfying $\d \rho\neq 0$
and a 2-dimensional K\"ahler structure $(g_{\mathrm {c}},\omega_{\mathrm {c}})$ with parallel hermitian endomorphism $A_{\mathrm {c}}=c\cdot\Id$ for a constant $c$.
Applying Theorem \ref{thm:nonconstant} to these formulas, we obtain the 
formulas for the K\"ahler structure $(g,\omega)$ and the c-compatible endomorphism $A$ given by \cite[formulas (3.5) and (3.6)]{BMMR}
(up to a slight change of notation).
\end{ex}

Theorem~\ref{thm:invform} gives us a description of a c-compatible pair $(g,  \omega)$ and $A$ (at a generic point) 
 provided we know a  description of compatible pairs $(h,L)$ and also of 
 K\"ahler structures $(g_{\mathrm{c}}, \omega_{\mathrm{c}})$ admitting a parallel hermitian endomorphism $A_{\mathrm{c}}$. 
As we already mentioned above, the latter have been described in \cite{Boubel}.  The local normal forms for compatible pairs $(h,L)$  
have been obtained in \cite{BM} and this combined with Theorem~\ref{thm:invform} implies the local normal forms 
for a c-compatible pair $(g,\omega)$ and $A$, see Example~\ref{ex:main} and Theorem~\ref{thm:localclassification} below. 
We also refer to \cite[Theorem 3.1]{BMMR} for the formulas in the 4-dimensional case.

\begin{ex}[Main example]
\label{ex:main}
Let $2n\geq 4$ and consider an open subset $W$ of $\R^{2n}$ of the form $W=V\times U\times  S_1\times\dots\times S_N$ 
for open subsets $V,U\subseteq \R^\ell$ and $S_\gamma\subseteq \R^{2m_\gamma}$. Let $t_1,\dots,t_\ell$ denote the coordinates on $V$ 
and let the coordinates on $U$ be separated into $r$ complex coordinates $z_1,\dots,z_{r}$ and $q=\ell-2r$ 
real coordinates $x_{r+1},\dots,x_{r+q}$.

Suppose the following data is given on these open subsets:

\begin{itemize}

\item K\"ahler structures $(g_\gamma,J_\gamma,\omega_\gamma)$ on $S_\gamma$ for $\gamma=1,\dots,N$.

\item For each $\gamma=1,\dots,N$, a parallel hermitian endomorphism $A_\gamma:TS_\gamma\rightarrow TS_\gamma$ for $(g_\gamma,J_\gamma)$ 
having a pair of complex conjugate eigenvalues $c_\gamma,\bar{c}_\gamma\in \C\setminus \R$ for $\gamma=1,\dots,R$ 
and a single real eigenvalue $c_\gamma\in\R$ for $\gamma=R+1,\dots,N$ such that the algebraic multiplicity 
of $c_\gamma$ equals $m_\gamma/2$ for $\gamma=1,\dots,R$ and $m_\gamma$ for $\gamma=R+1,\dots,N $. 

\item Holomorphic functions $\rho_j(z_j)$ of $z_j$ for $1\leq j\leq r$ and smooth functions $\rho_j(x_j)$ 
for $r+1\leq j\leq r+q$.

\end{itemize}
Moreover, we choose 1-forms $\alpha_1,\dots,\alpha_\ell$ on $S=S_1\times\dots\times S_N$ which satisfy
\begin{align}
\d\alpha_i=(-1)^i \sum_{\gamma=1}^N \omega_\gamma(A_\gamma^{\ell-i}\cdot ,\cdot).\label{eq:theta}
\end{align}
We introduce some notation to be used throughout the paper. The function $\Delta_i$ for $1\leq i\leq r+q$ is given by
$\Delta_i=\prod_{\rho\in E_{\mathrm{nc}}\setminus\{\rho_i\}}(\rho_i-\rho)$, where 
$E_{\mathrm{nc}}=\{\rho_1,\bar{\rho}_1,\dots,\rho_r,\bar{\rho}_r,\rho_{r+1},\dots,\rho_{r+q}\}$.
The $1$-forms $\theta_1,\dots,\theta_\ell$ on $W$ are defined by $\theta_i=\d t_i+\alpha_i$.
The function $\mu_i$ denotes the $i$th elementary symmetric polynomial in the $\ell$ variables $E_{\mathrm{nc}}$, 
$\mu_{i}(\hat{\rho_s})$ denotes the $i$th elementary symmetric polynomial in the $\ell-1$ variables 
$E_{\mathrm{nc}}\setminus\{\rho_s\}$ and the notation ``$c.c$'' refers to the conjugate complex of the preceding term.

Suppose that at every point of $W$ the values of the functions $\rho_1,\bar\rho_1,\dots,\rho_{r+q}$ 
are mutually different and different from the constants $c_1,\bar c_1,\dots ,c_N$ and their differentials are non-zero. 
Then $(g,\omega,J)$ given by the formulas
\begin{align}
\begin{array}{c}
\displaystyle
g=-\frac{1}{4}\sum_{i=1}^r \left(\Delta_i \d z_i^2+c.c.\right)+\sum_{i=r+1}^{r+q}\varepsilon_i\Delta_i \d x_i^2+\sum_{i=0}^{\ell}(-1)^{i}\mu_i \sum_{\gamma=1}^N g_\gamma(A_\gamma^{\ell-i} \cdot , \cdot)\vspace{1mm}\\
\displaystyle+\sum_{i,j=1}^\ell\left[-4\sum_{s=1}^r\left(\frac{\mu_{i-1}(\hat{\rho_s})\mu_{j-1}(\hat{\rho_s})}{\Delta_s}\left(\frac{\partial \rho_s}{\partial z_s}\right)^2+c.c.\right)+\sum_{s=r+1}^{r+q}\varepsilon_s\frac{\mu_{i-1}(\hat{\rho}_s)\mu_{j-1}(\hat{\rho}_s)}{\Delta_s}\left(\frac{\partial \rho_s}{\partial x_s}\right)^2\right]\theta_i\theta_j,\vspace{1mm}\\
\displaystyle\omega=\sum_{i=1}^\ell \d \mu_i \wedge \theta_i + \sum_{i=0}^{\ell}(-1)^i\mu_i\sum_{\gamma=1}^N \omega_\gamma(A_\gamma^{\ell-i}\cdot , \cdot),
  \vspace{1mm}\\
\displaystyle \d z_i\circ J=4\frac{1}{\Delta_i}\frac{\partial \rho_i}{\partial z_i}\sum_{j=1}^\ell\mu_{j-1}(\hat{\rho_i})\theta_j,\quad 
\d x_i\circ J=-\frac{\varepsilon_i}{\Delta_i}\frac{\partial \rho_i}{\partial x_i}\sum_{j=1}^\ell\mu_{j-1}(\hat{\rho}_i)\theta_j,\vspace{1mm}\\
\displaystyle \theta_i\circ J=\frac{(-1)^{i}}{4}\sum_{j=1}^r \rho_{j}^{\ell-i}\left(\frac{\partial \rho_j}{\partial z_j}\right)^{-1}\d z_j+c.c+(-1)^{i-1}\sum_{j=r+1}^{r+q}\varepsilon_j\rho_{j}^{\ell-i}\left(\frac{\partial \rho_j}{\partial x_j}\right)^{-1}\d x_j
\end{array}
\label{eq:localclassificationgJw}
\end{align}
is K\"ahler, where $\varepsilon_i=\pm 1$ depending on the signature of $g$. Moreover, writing 
$\alpha_i = \sum_{q}  \alpha_{iq} \d y_q$ and $A_{\gamma} = \sum_{p,q}    (A_{\gamma})^q_p  \,  \d y_p\otimes  \partial_{y_q}$  w.r.t. local coordinates $y_1,\dots,y_{2k}$ on $S=\prod_{\gamma} S_\gamma$, we have that the endomorphism $A$ given by
\begin{align}
\begin{array}{c}
\displaystyle A=\sum_{i,j=1}^\ell\left(\mu_i\delta_{1j}-\delta_{i(j-1)}\right)\theta_i\otimes \frac{\D }{\D t_j} 
+\sum_{s=1}^r(\rho_s\d z_s \otimes \frac{\D}{\D z_s}+c.c.)+\sum_{s=r+1}^{r+q} \rho_s \d x_s\otimes \frac{\D}{\D x_s}\vspace{1mm}\\
\displaystyle
+\sum_{\gamma=1}^N \sum_{p,q=1}^{2k}    (A_{\gamma})^q_p  \,  \d y_p  \otimes \left( \frac{\partial}{\partial {y_q}}    - \sum_{i=1}^\ell\alpha_{iq} \frac{\partial}{\partial {t_i}}   \right)
\end{array}
\label{eq:localclassificationA}
\end{align}
is a hermitian solution to \eqref{eq:main}.
\end{ex}

\vspace{1ex} 

Example~\ref{ex:main} is an explicit  construction of a K\"ahler metric \eqref{eq:localclassificationgJw} and 
a solution \eqref{eq:localclassificationA} of \eqref{eq:main}. This fact can be verified by a straightforward, 
though non-trivial computation. Another proof will be given in Sections~\ref{sec:realisation} and~\ref{explicitform}.   
The next theorem shows that in a neighbourhood of a generic point,  a  K\"ahler metric $g$ (of any signature) and a solution 
$A$ of \eqref{eq:main} are, in a certain coordinate system, as in Example~\ref{ex:main}.

\begin{thm}[Local description of c-projectively equivalent metrics]
\label{thm:localclassification}
Suppose  $(M,g,J)$ is  a  K\"ahler manifold of arbitrary signature and  $A$ is  a hermitian solution of 
\eqref{eq:main}.  Assume  that in a small neighbourhood $W\subseteq M^0$ of a regular point, $A$ has 

\begin{itemize}

\item $\ell=2r+q$ non-constant eigenvalues on $W$ which separate into $r$ pairs of complex-conjugate 
eigenvalues $\rho_1,\bar\rho_1,\dots,\rho_{r},\bar\rho_{r}:W\rightarrow \C$ and $q$ real eigenvalues $\rho_{r+1},\dots,\rho_{r+q}:W\rightarrow \R$,

\item $N+R$ constant eigenvalues which separate into $R$ pairs of complex conjugate eigenvalues 
$c_1,\bar c_1,\dots,c_R,\bar c_R$ and $N-R$ real eigenvalues $c_{R+1},\dots,c_{N}$.

\end{itemize}
Then the K\"ahler structure $(g,J,\omega)$ and $A$ are given on $W$  by the formulas \eqref{eq:localclassificationgJw} 
and \eqref{eq:localclassificationA} from Example~\ref{ex:main}.
\end{thm}

\begin{rem}  
As stated above, the corresponding local description of a positive definite K\"ahler structure $(g,J,\omega)$ 
admitting a hermitian solution $A$ of \eqref{eq:main} has been obtained in \cite{ApostolovI} in the language of 
hamiltonian $2$-forms. To obtain a better understanding of the relation between the formulas in \cite{ApostolovI} and the
 formulas  \eqref{eq:localclassificationgJw} and \eqref{eq:localclassificationA} above, note that if 
$g$ is positive definite, all the eigenvalues of the hermitian endomorphism $A$, both constant and non-constant,  
are real. Thus, for a positive definite metric $g$, the pairs of conjugate complex non-constant eigenvalues
 $\rho_1,\bar\rho_1,\dots,\rho_r,\bar\rho_r$ disappear from the formulas \eqref{eq:localclassificationgJw} 
and \eqref{eq:localclassificationA} and the endomorphisms $A_\gamma:TS_\gamma\rightarrow TS_\gamma$ are just a constant multiple
 of the identity $\mathrm{Id}:TS_\gamma\rightarrow TS_\gamma$, that is, $A_\gamma=c_\gamma \cdot \mathrm{Id}$. Moreover, 
in \cite{ApostolovI}, the non-constant eigenvalues $\rho_1,\dots,\rho_\ell$ of $A$ have been used as
 part of the coordinate system instead of $x_1,\dots,x_\ell$ used in \eqref{eq:localclassificationgJw}
 and \eqref{eq:localclassificationA}. The change of coordinates $x_i\longmapsto \rho_i$ for $i=1,\dots,\ell$ yields
$$
\d x_i^2=\frac{\d \rho_i^2}{\Theta_i(\rho_i)},
$$
where $\Theta_i(\rho_i)=\left(\frac{\D \rho_i}{\D x_i}(x_i(\rho_i))\right)^2$. The functions $F_i(t)$ appearing 
in the local classification of \cite{ApostolovI} relate to the $\Theta_i$'s by
$$
F_i(t)=\Theta_i(t)\prod_{\gamma=1}^N(t-c_\gamma)^{m_\gamma}.
$$
The formula for the hamiltonian 2-form $\phi$ (corresponding to $A$) is obtained by inserting \eqref{eq:localclassificationgJw} and \eqref{eq:localclassificationA}
into $\phi=\omega(A\cdot,\cdot)$. Up to changing notations (for instance $\rho_i\longmapsto \xi_i$, $\mu_i\longmapsto \sigma_i$), 
this relates \eqref{eq:localclassificationgJw} and \eqref{eq:localclassificationA} in the case when $g$
is positive definite to the description obtained in \cite{ApostolovI}.
\end{rem}

\begin{rem}
As mentioned above, Theorem~\ref{thm:localclassification} yields an ``almost'' 
explicit description of a K\"ahler metric $g$ admitting a c-projectively equivalent metric. 
What is not described explicitly are the K\"ahler structures $(g_\gamma,\omega_\gamma)$ that admit parallel
 hermitian endomorphisms $A_\gamma$. The formulas for such a triple $(g_\gamma,\omega_\gamma,A_\gamma)$ in local
 coordinates can be found in \cite{Boubel}.
\end{rem}

\subsection{Structure of the paper}

In \S\ref{sec:killing}, we recall that the existence of a c-projectively equivalent
 metric $\hat{g}$ for a K\"ahler metric $g$ implies the existence of a family of 
independent commuting hamiltonian (w.r.t. the K\"ahler form $\omega$) Killing vector
 fields $K_1,\dots,K_\ell$. These vector fields are also Killing w.r.t. $\hat{g}$.

We can form the quotient of $M$ w.r.t. the local $\R^\ell$-action induced by these vector 
fields and obtain a bundle structure $M\rightarrow Q$ with fibers being the leaves of the 
foliation generated by the vector fields $K_i$. Since $g,\hat{g}$ are invariant w.r.t. 
the action of the vector fields $K_i$ and the orthogonal complements to the fibers w.r.t.   $g$ 
and $\hat g$ coincide, they descend to metrics $g_Q,\hat{g}_Q$ on the quotient. This 
reduction will be explained in detail in \S\ref{sec:reduction}.  As we already mentioned, in \S\ref{sec:idea}, 
the crucial observation  is that the metrics
 on the quotient are \emph{projectively equivalent}. 
We prove this property in  \S\ref{sec:cprorpro}.  More precisely,  as explained in Example~\ref{ex:modelB}, 
instead of $\hat{g}_Q$ we consider the endomorphism $A_Q$ obtained from  $g_Q$ and $\hat{g}_Q$ by \eqref{eq:defL} 
and check the compatibility condition for the pair $g_Q$, $A_Q$.

The local classification of 
pseudo-Riemannian projectively equivalent metrics, or equivalently,  compatible pairs $g_Q$ and $A_Q$ has been derived in \cite{BM}. We apply 
these results in \S\ref{sec:normalformsquotient}  to obtain the normal forms for 
$g_Q$, $A_Q$ on the quotient. These normal forms are, in fact, simpler than the generic ones 
for projectively equivalent metrics: the tensor $A=A(g,\hat{g})$ from Theorem~\ref{thm:localclassification} 
has non-constant eigenvalues of (complex) algebraic multiplicity equal to one such that the corresponding 
tensor $A_Q$ on the quotient has no non-trivial Jordan blocks corresponding to the non-constant eigenvalues. 
This makes the formulas from \cite{BM} much easier.

The requirement that $g$ is K\"ahler and $A$ is hermitian  implies
 that they are completely determined by the reduced objects $g_Q$ and $A_Q$ on the quotient. In \S\ref{sec:normalformsquotient}
 we derive the formulas  for $(g,\omega)$ and  
 $A$ which are in essence equivalent to  \eqref{eq:gomegainv}   and \eqref{eq:Ainv}  (Proposition~\ref{prop:locform}). 
The next step is to show that there are no further restrictions on $g$, $\omega$ and $A$ so that the formulas 
from Proposition~\ref{prop:locform} give us a desired local description, see \S\ref{sec:realisation}.

Finally in \S\ref{explicitform}, we complete the proof of Theorem~\ref{thm:localclassification} 
by deriving the explicit formulas \eqref{eq:localclassificationgJw} and \eqref{eq:localclassificationA} from Example~\ref{ex:main}.

The second part of the article contains the proof of Theorems~\ref{thm:yano_obata} and~\ref{thm:lichnerowicz}.  
As we already pointed out,  there is a close relationship between c-projective and projective equivalence.  
This makes the proofs of these theorems rather similar. In Section~\ref{sec:yano_obata} we focus on the proof 
of the Yano-Obata conjecture (Theorem~\ref{thm:yano_obata})  and explain in a series of remarks how this proof 
can be adapted for the Lichnerowicz conjecture (Theorem~\ref{thm:lichnerowicz}).  This is done  under one 
additional algebraic condition: the endomorphism $A$ compatible with the metric $g$ and induced by the 
projective vector field $v$ has no Jordan blocks with non-constant eigenvalues\footnote{According to the 
local description given by Theorem~\ref{thm:localclassification}, in the c-projective setting such blocks do not occur.}.   
The latter case when $A$ admits a ``non-constant'' Jordan block is treated in Section~\ref{sec:Jblock}.

\section{Canonical Killing vector fields for  c-projectively equivalent metrics}
\label{sec:killing}

Let  $(M,g,J)$ be a connected K\"ahler manifold of real dimension $2n\geq 4$. Since by definition any $A$ which is c-compatible with $(g,J)$ commutes with $J$, 
we can consider $A$ as an endomorphism of the $n$-dimensional  complex vector space $T_p M$ (with complex multiplication
given by $(a+ib)X=aX+bJX$). The determinant of $A$ considered as complex endomorphism will be denoted by $\mathrm{det}_{\C}\,A$. It  is a smooth function on $M$ and since $A$ is hermitian,
it is real valued. Up to a sign $\mathrm{det}_{\C}\,A$ equals $\sqrt{\mathrm{det} \, A},$ though the latter is always positive and smoothness may fail at the points where it vanishes.

\vspace{1ex}  
 
Recall that a vector field is called a \emph{ Killing vector field} w.r.t. the metric $g$, if its local flow 
preserves $g$. Similarly, a vector field is called \emph{holomorphic} if its local flow preserves the complex structure $J$. 

Since the local flow  of a  holomorphic Killing vector field $K$ preserves the symplectic  form $\omega= g( J\cdot, \cdot)$, the vector field $K$ 
is hamiltonian in a neighbourhood of any point or, more generally, on every simply connected open subset. Recall that a vector field $K$ is called hamiltonian if there 
exists a function $f$ such that 
$$
i_K \omega = -\d f\mbox{ or, equivalently, }K=J\gr f.
$$
Such a function $f$ is called a hamiltonian for $K$ and it is only unique up to adding a constant.
Conversely, since every hamiltonian vector field preserves $\omega$, it is Killing if and only if it is holomorphic. 
Recall also that holomorphic vector fields are characterised by the property that their covariant differential is complex-linear 
(when considered as endomorphism of $T_p M$) and therefore a hamiltonian vector field $K=J\gr f$ is holomorphic if and only if 
the hessian $\nabla^2 f$ is hermitian.

\begin{lem} \label{lem:killing} 
For any  $A$ which is c-compatible with $(g, J)$ the function $ \mathrm{det}_{\C}\,A$ is a  hamiltonian for a Killing vector field. 
\end{lem}

We do not pretend that Lemma~\ref{lem:killing} is new: for positive definite metrics it is equivalent to 
\cite[Proposition 3]{ApostolovI} and this proof can be generalised to all signatures. 
We give a different and shorter proof, which is based on the same observation as the proof given  
in  \cite{CEMN}  but does not require introducing c-projectively invariant objects.

\begin{proof}  
Since the statement is local, w.l.o.g.  we may  assume    that  $\mathrm{det}\,A\ne 0$,   
otherwise we can locally replace $A$ by $A+\textrm{const}  \cdot \Id$.   Then, as explained  in \S\ref{sec:lc}, the metric $\hat g$ given by \eqref{eq:defhatg} 
is c-projectively equivalent to $g$.  We denote by $\nabla$ and $ \hat \nabla$ the Levi-Civita connections of $g$ and $\hat g$. 
It is well known (see for example the survey \cite{Mikes}), and follows directly from the definition of c-projective equivalence, 
that the connections $\nabla $ and $\hat\nabla$ are related by 
the equation 
\begin{equation}
\label{M1}
 \hat\nabla_X Y - \nabla_X Y=  \Phi(X)Y+\Phi(Y)X-\Phi(JX)JY-\Phi(JY)JX, 
 \end{equation}
 where  $\Phi$  is an exact $1$-form equal to the differential of the function 
 \begin{equation}
 \label{M2}
 \phi= \tfrac{1}{4(n+1)}\ln \left(\frac{\det  \hat g}{\det g}\right). 
 \end{equation}
Combining \eqref{eq:defhatg} and \eqref{M2}, we see that  $$
\exp(-2\phi)=  \left|\mathrm{det}_{\C}\,A\right|.$$

Now, it follows from straightforward calculations using \eqref{M1} (see e.g. \cite{Mikes}),  that 
the Ricci tensors $\mathrm{Ric}$ and $\widehat{\mathrm{Ric}}$ of the metrics $g$ and $\hat g$ are related by
$$
\widehat{\mathrm{Ric}} - \mathrm{Ric}= -2(n+1)(\nabla\Phi- \Phi^2 + (\Phi\circ J)^2).
$$ 
Note  that $\nabla\Phi$ is a symmetric $(0,2)$-tensor. For a K\"ahler metric, the Ricci tensor is hermitian 
w.r.t. the complex structure.  Then the above equation implies that $\nabla\Phi- \Phi^2 + (\Phi\circ J)^2$ is hermitian. 
Hence, 
\begin{eqnarray*}
\nabla^2\left|\mathrm{det}_{\C}\,A\right|
& = &\nabla^2\exp(-2\phi) = 2\exp(-2\phi) (- \nabla \Phi + 2 \Phi^2) \\ & =& 
2\exp(-2\phi) (- (\nabla \Phi - \Phi^2+(\Phi\circ J)^2)+\Phi^2+(\Phi\circ J)^2) 
\end{eqnarray*}
is hermitian as well. This implies that $J\gr\left|\mathrm{det}_{\C}\,A\right|$ is Killing. 
\end{proof}

\vspace{1ex} 

For  each $A$ which is c-compatible with $(g,J)$,  Lemma~\ref{lem:killing} gives us  a hamiltonian  Killing vector field with the hamiltonian 
function $\det_\mathbb{C}A$. If $A$ is not parallel, this function is non-constant and therefore the Killing vector field is non-trivial.

Since equation \eqref{eq:main} is linear in $ A $ and admits $\Id:TM\rightarrow TM$ as a solution, we actually have a whole family 
$A(t) = t \cdot \Id-A $ of endomorphisms c-compatible with $(g,J)$. 
For any fixed  $t$,  the function ${\det_\mathbb{C}A(t)}$  is a hamiltonian for a  Killing vector field which we denote by $K(t)$.  We will call  these   
vector fields, and also all their linear combinations with constant coefficients,  \emph{canonical Killing vector fields 
corresponding to the solution $A$ of \eqref{eq:main}} (or to the c-projectively equivalent metric $\hat g$), or simply \emph{canonical Killing vector fields}.

\begin{lem} 
\label{lem:properties} 
The following statements hold for any endomorphism $A$ which is c-compatible with $(g,J)$:
 
 \begin{enumerate} 
 
 \item  Suppose for a smooth function $\rho$ on an open subset $U\subseteq M$ and for any point 
$p\in U$ the number $\rho(p)$ is an eigenvalue of $A$ at $p$ of algebraic  multiplicity $\ge 4$. 
Then this function $\rho$ is a constant on $U$.   {Moreover,   for any point of the manifold 
the constant  $\rho$ is an eigenvalue of $A$.} \label{l1} 
   
 \item The vectors  $\gr\, \rho$  and $J\gr\, \rho$ are  eigenvectors 
		of $A$ with eigenvalue $\rho$ at the points where the eigenvalue  $\rho$ is a smooth function.  \label{l4}
  
 \item  At a generic  point, the number of linearly independent canonical Killing vector fields coincides with the number of 
non-constant eigenvalues of $A$. \label{l2} 
      
\item At each regular point the number of eigenvalues $\rho$ with $\d\rho\ne 0$ is the same. \label{l3}

\item At regular points, the  restriction of  $g$ to the distribution 
spanned by the canonical Killing vector fields is non-degenerate. \label{l5} 

\item The canonical Killing vector fields $K(t)$, and also the vector fields $JK(t)$ commute: for any  $t_1, t_2\in \R $  
  we have \label{l6}
 $$   [K(t_1),K(t_2)]=[K(t_1),JK(t_2)]=[JK(t_1),JK(t_2)]=0 \quad\mbox{and}\quad \omega(K(t_1), K(t_2))=0.$$

 \item The local flow of every canonical  Killing vector field preserves $A$.  \label{l7} 
 
\item  For any two canonical Killing vector  fields $K(t_1), K(t_2)$ 
the vector $J\nabla_{ K(t_1)}K(t_2) $ at any point is contained in the span of the vector fields $K(t)$, $t\in \R$.  \label{l8}
  
 \end{enumerate}
\end{lem}

Most statements of the  lemma can be found in \cite{CEMN}. The proofs in the present paper are  
different from those in \cite{CEMN}, are shorter and do  not require introducing c-projectively 
invariant objects. For $g$ positive definite, the statements of the lemma have been obtained in the 
language of hamiltonian 2-forms in \cite{ApostolovI}. It is  not possible (that is, we did not 
find an easy way to do it) to directly generalise the proof from  \cite{CEMN} to metrics of all signatures. 

\begin{proof} 
  Let  $\lambda_1(x),\dots,\lambda_k(x)$ be the eigenvalues of $A$ at a  point $x\in M$. In the  
	proof of the \ref{l1}st  statement  of Lemma~\ref{lem:properties}  
	we  will  work in a neighbourhood of a generic point, which implies that we may assume w.l.o.g. that   
	the algebraic multiplicities of the eigenvalues are $2m_1,\dots,2m_k$, they do not change in this neighbourhood and 
 all  $\lambda_i$ are smooth, possibly complex-valued functions.   Now, evidently 
$f(t)= \det_\mathbb{C}(t \cdot \Id- A)= (t -\lambda_1)^{m_1}\dots(t- \lambda_k)^{m_k}$. 
     Note that the formula for $f(t)$  makes sense also   if $t\in \mathbb{C}\setminus \mathbb{R}$. 
 Indeed, because of linearity of the Killing equation, for a hamiltonian function $f(t)$ with   
$t\in \mathbb{C}\setminus \mathbb{R}$  the   hamiltonian vector field, which is now  complex-valued, 
is still a holomorphic Killing vector field in the sense that its real and imaginary parts are 
holomorphic Killing vector fields. To see this, note that 
 $f(t)$ is a polynomial in  $t$ so all of its coefficients are hamiltonians for holomorphic 
Killing vector fields. Thus,  for complex-valued $t$ the real and imaginary parts of $f(t)$
are still  linear combinations of the coefficients. 
 
  Consider now the family $\d f(t)$ of  differentials  of hamiltonians of the canonical  Killing vector fields. 
	It is   given by 
 \begin{equation}
 \label{kill}
\begin{aligned} 
m_1(t -\lambda_1)^{m_1-1}  ( t-\lambda_2)^{m_2}\dots(t-\lambda_k )^{m_k} \d\lambda_1
&+ m_2(t -\lambda_1)^{m_1}  (t -\lambda_2)^{m_2-1}\dots(t-\lambda_k )^{m_k} \d\lambda_2 +
\\ 
 \dots &+
  m_k( t-\lambda_1)^{m_1}  ( t-\lambda_2)^{m_2}\dots(t-\lambda_k )^{m_k-1}\d\lambda_k. \end{aligned} 
 \end{equation} 
 
 Suppose now that a (possibly complex-valued) eigenvalue  $\lambda_i$    has  algebraic multiplicity $2 m_i\ge 4$. 
W.l.o.g. we may think that $i=1$. We take an arbitrary  point $p$,   set $\tilde \lambda= \lambda_1(p)$ and    
consider the hamiltonian  $f(t)$ with  $t=\tilde \lambda$. Since $m_1\ge 2$, we see  that 
   $\d f(\tilde \lambda)=0$ at $p$. Then the components of the matrix of the hessian $\nabla^2 f(\tilde \lambda) $ at $p$ in any coordinate 
	system $x_i$ are simply given by the components $\D_i\D_j f(\tilde \lambda)$ of the usual hessian  at $p$ and, hence,    
	$$
   \left.\nabla^2 f(\tilde \lambda)(p)=m_1(m_1-1)(\tilde \lambda -\lambda_1)^{m_1-2}  (\tilde \lambda-\lambda_2 )^{m_2}
	\dots(\tilde \lambda-\lambda_k )^{m_k} \d\lambda_1^2\right..
   $$
 We see that  if $\lambda_1$ is actually real-valued, the hessian $\nabla^2 f(\tilde \lambda) $  at the point  $p$  
vanishes  or  has rank $1$. But it cannot have rank $1$ because it is hermitian. 
Thus, $\nabla^2 f(\tilde \lambda) $ has to vanish.  
   
   Suppose now $\lambda_1= \alpha + i \beta$,  where $\alpha$ and $\beta$ are real-valued functions. 
	Then, 
	$$ \d\lambda_1^2= 
   \d\alpha^2 -  \d\beta^2+ 2i \d\alpha\d\beta.
   $$ 
  If $\d\alpha$ and  $\d\beta$ are linearly dependent,   $\d\alpha^2 -  \d\beta^2$ and  
	$\d\alpha\d\beta $ have  rank $1$ or $0$. Since     rank $1$ is  
	impossible they  vanish. The case when  $\d\alpha$ and  $\d\beta$ are linearly independent cannot occur 
	because in this case the bilinear forms $\d\alpha^2 -  \d\beta^2$ and  $\d\alpha\d\beta$
	have signature $(1,1,2n-2)$, which contradicts that they are hermitian. Finally, $\nabla^2 f(\tilde \lambda) =0$ at $p$.

   It is well  known that the first jet (i.e, the   vector field and its first covariant derivative)  
	of a Killing  vector field at a point determines the Killing vector field on the whole manifold. As we 
	just proved,  the first jet of the  Killing vector field corresponding to the hamiltonian $ f(\tilde \lambda) $
	vanishes at $p$. Then it vanishes on the whole manifold which implies that the function $ f(\tilde \lambda) $  
	is a constant. It is clearly zero at the point $p$ so it is identically zero and     $\tilde \lambda$ is an 
	eigenvalue at every point of the manifold.     The \ref{l1}st statement of Lemma~\ref{lem:properties} is proved.
  
    Let us now  prove the  \ref{l4}nd  statement.  Denote by $X$ a vector   field of   eigenvectors 
		corresponding to a non-constant eigenvalue  $\rho$ (viewed as a function on the manifold). 
   First observe that for  any vector   $Y$   we have  
\begin{align}
\begin{array}{c}
(A-\rho\cdot\Id)\nabla_Y X=\d\rho(Y)X-g(X,Y)\Lambda-g(X,\Lambda)Y
-g(X,JY)J\Lambda+g(X,J\Lambda)JY.
\end{array}
\label{eq:covderiv}
\end{align}
To obtain  \eqref{eq:covderiv},  take the covariant derivative  in the direction of $Y$ of the equation 
$(A-\rho\cdot\Id) X=0$, substitute \eqref{eq:main} and rearrange the terms.  
 
Taking $Y$ orthogonal to $X$ and to $JX$,  we see that the right hand side of \eqref{eq:covderiv} is a linear 
combination of the linearly independent vectors $X$, $Y$ and $JY$. Since the left hand side $(A-\rho\cdot\Id)\nabla_Y X$ 
is orthogonal to  the kernel of  $(A-\rho\cdot\Id)$, the coefficient of $X$, which is $\d\rho(Y)$,  is zero. 
Thus, the function $\rho$ is constant in any direction orthogonal to $X$ and to $JX$. By the \ref{l1}st statement of Lemma~\ref{lem:properties}, 
the algebraic multiplicity of $\rho$ is $2$ and it follows that
$\gr\,\rho$ and $J\gr\, \rho$ are eigenvectors of $A$ corresponding to the eigenvalue $\rho$.
   
   To prove the \ref{l2}rd statement, consider the non-constant eigenvalues of $A$ and denote them  by  $\rho_1,\dots,\rho_\ell$. 
	We will work near a generic point so we may assume that $\rho_1,\dots,\rho_\ell$ are smooth functions with non-zero differentials.
	Observe that   for any $t$   the function   $( t-\lambda_i)^{m_i}$ is a constant if the eigenvalue $\lambda_i$ is a constant and, in view of  the proved \ref{l1}st statement,
   if $m_i\ge 2$. Then each $f(t)$ is proportional  with  a constant coefficient to $(t-\rho_1 )\dots(t-\rho_\ell)$. 
	The function  $\tilde f(t)= (t-\rho_1 )\dots(t-\rho_\ell)$ is a polynomial of degree $\ell$ with  leading coefficient equal to $1$  
	and has at most $\ell$ non-constant coefficients. Thus, the number of linearly independent canonical Killing vector fields is at most $\ell$. 
   
   Since the gradients $\gr\, \rho_i$ belong to different eigenspaces, they are linearly independent and in view of 
	\eqref{kill}, the differentials of  $\tilde f(t_1)$  and $\tilde f(t_2)$ are linearly independent for $t_1\ne t_2$ so the number of linearly 
	independent canonical Killing vector fields is precisely $\ell$.

To prove the \ref{l3}th  statement, recall that  a Killing vector field which vanishes on 
an open set vanishes everywhere.  Then, by the \ref{l2}rd statement of  the lemma,  the number of  non-constant eigenvalues   
is the same on every open subset. Then   the number of  constant eigenvalues is also the same on every open subset   and the claim follows. 
   
In order to proof the \ref{l5}th statement, observe that the distribution spanned by the canonical Killing vector fields, at regular points, 
 coincides with the distribution   spanned by the hamiltonian vector fields generated by the non-constant eigenvalues.  
By the \ref{l4}nd  statement, such hamiltonian vector fields  have non-zero length at regular points and  are mutually orthogonal, and the claim follows.  
   
Let us prove the \ref{l6}th statement. By the \ref{l4}nd  statement, we have 
$$
\omega(K(t_1),K(t_2))=g(JK(t_1),K(t_2))=0
$$ 
for any real numbers $t_1,t_2$.
By definition of a Poisson bracket, this equation is equivalent to say that the hamiltonian functions $f(t_1),f(t_2)$ corresponding to $K(t_1),K(t_2)$ Poisson commute,
$\{f(t_1),f(t_2)\}=0$. On the other hand, recall that $[K(t_1),K(t_2)]$ is the hamiltonian vector field corresponding to the hamiltonian $\{f(t_1),f(t_2)\}$.
We obtain $[K(t_1),K(t_2)]=0$. The remaining equations follow from the fact that the vector field $K(t)$ is holomorphic and therefore also $JK(t)$ is holomorphic.

To prove the \ref{l7}th statement, assume w. l. o. g. that $A$ is non-degenerate. Then we can consider the metric $\hat g$ from \eqref{eq:defhatg}
c-projectively equivalent to $g$. It is sufficient to show that the canonical Killing vector fields for $g$ are also canonical Killing vector  fields for 
$\hat g$. W.l.o.g.  we may work in a neighbourhood of a regular point. Let 
$\rho_1,\dots,\rho_\ell$ denote the non-constant eigenvalues of $A $. If we swap the metrics $g$ and $\hat g$ in the definition \eqref{eq:defA},
   the tensor constructed by the pair of metrics $\hat g$, $g$ is clearly the inverse of the initial  $A$, therefore its non-constant eigenvalues are 
   $\tfrac{1}{\rho_1},\dots,\tfrac{1}{\rho_\ell}$. 
   
 We will show that the canonical  Killing vector field $K(t)$ for $g$, whose hamiltonian  is $ \det_\mathbb{C} ( t\cdot \Id-A )$,  
  is proportional with a non-zero constant coefficient  to the  canonical  Killing vector field $\hat K\left(\tfrac{1}{t}\right) $  for $\hat g$,
	whose hamiltonian is  $ {\det_\mathbb{C} \left( \tfrac{1}{t}\cdot \Id- A^{-1}\right)}$. 
 
 Since the multiplicity of the non-constant eigenvalues of $A$ is two, 
 up to multiplication by  a constant, for any $t$, 
 the differential of  ${\det_\mathbb{C} (t \cdot\Id -A )}$ coincides  with  the differential 
 of $(t-\rho_1 )\dots(t-\rho_\ell)$ which is 
$$
    (t -\rho_2)\dots(t-\rho_\ell)\d\rho_1   + ( t-\rho_1) (t -\rho_3 )\dots(t-\rho_\ell) \d\rho_2  + \dots +  (t -\rho_1)\dots(t-\rho_{\ell-1})\d \rho_\ell . 
 $$ 
Similarly,   for any $t\ne 0$, 
 the differential of  $\det_\mathbb{C}\left(\tfrac{1}{t} \cdot\Id - A^{-1} \right)$ is proportional 
  with a constant coefficient  to the differential  
 of 
 $(\tfrac{1}{t}- \tfrac{1}{\rho_1})\dots(\tfrac{1}{t}-\tfrac{1}{\rho_\ell})$ 
which is, up to multiplication by a non-zero  constant, given by  
$$
   \frac{1}{\det_\mathbb{C}(A)} \left((\rho_2 -t)\dots(\rho_\ell-t)\tfrac{1}{\rho_1}\d\rho_1   
	+ (\rho_1 -t) (\rho_3 - t)\dots(\rho_\ell-t) \tfrac{1}{\rho_2}\d\rho_2  + \dots +  (\rho_1 -t)\dots(\rho_{\ell-1}-t)\tfrac{1}{\rho_\ell}\d \rho_\ell\right) . 
 $$
 Now, the canonical vector fields $K(t)$ and $\hat  K \left(\tfrac{1}{t}\right)$ are related to the differentials 
of $\det_\mathbb{C} ( t \cdot\Id-A )$ and $\det_\mathbb{C}\left(\tfrac{1}{t}\cdot \Id - A^{-1} \right)$ by 
 $$
 K(t)= J\gr_g \mathrm{det}_{\mathbb{C}} ( t \cdot\Id-A ) \textrm{ and }   \hat  K \left(\tfrac{1}{t}\right) 
= J\gr_{\hat g}\mathrm{det}_{\mathbb{C}}\left(\tfrac{1}{t} \cdot\Id - A^{-1} \right). 
 $$
Combining this with  \eqref{eq:defhatg} and the \ref{l4}nd statement, we conclude that $ K(t)$ 
is proportional to $\hat  K \left(\tfrac{1}{t}\right)$  with a constant factor.

Let us now prove the \ref{l8}th statement.  It is sufficient to prove it on the dense and open subset $M^0$ of regular points. 
As usual, by $\rho_1,\dots,\rho_\ell$ we denote  the non-constant eigenvalues of $A$.
From the definition, it follows that the integrable distribution $\mathcal V$ spanned by the canonical Killing vector fields 
$K(t)$, $t\in \R$, coincides with the distribution spanned by the vector fields  $J\gr \rho_i$, $i=1,\dots,\ell$.  
Consider the distribution $\mathcal F=\mathcal V\oplus J\mathcal V$. It is spanned by the family of vector fields $K(t),JK(t)$, $t\in \R$,
is integrable by the \ref{l6}th statement and coincides with the span of $\gr \rho_i,J\gr \rho_i$, $i=1,\dots,\ell$. 
From formula \eqref{eq:covderiv} combined with  the  \ref{l4}nd statement, it follows immediately  that the distribution $\mathcal F$ is 
totally geodesic. By the \ref{l5}th statement, the restriction $g|_{\mathcal L}$ of $g$ to an integral leaf $\mathcal L\subseteq M^0$ of $\mathcal F$ 
is nondegenerate. Then it follows that the integral leafs $\tilde {\mathcal L}\subseteq \mathcal L$ of the integrable 
subdistribution $J\mathcal V\subset \mathcal F$ are totally geodesic since they are orthogonal in $(\mathcal L,g|_{\mathcal L})$ to a distribution 
spanned by Killing vector fields. This implies that $\nabla_{JK(t_1)}JK(t_2)$ is tangent to $J\mathcal V$, or equivalently (since $J$ is parallel 
and $K(t)$ is holomorphic), that $J\nabla_{K(t_1)}K(t_1)$ is tangent to
$\mathcal V$ as we claimed. This completes the proof of Lemma~\ref{lem:properties}.
\end{proof}

Let $\mu_i$ denote the $i$th elementary symmetric
polynomial in $\rho_1,\dots,\rho_\ell$,  i.e., in the non-constant eigenvalues of $A$ c-compatible with $(g,J)$. 
Note that although the $\rho_i$ may fail to be smooth at certain points, the $\mu_i$
are globally defined smooth functions on $M$: clearly we have $\mathrm{det}_{\C}(t\cdot\Id-A)=P(t)\sum_{i=0}^\ell (-1)^i\mu_i t^{\ell-i}$, where
$P(t)$ is a polynomial of degree $n-\ell$ with constant coefficients and we put $\mu_0=1$. 
In what follows we will mainly work with a special set of canonical Killing vector fields 
$K_1,\dots,K_\ell$ corresponding to $A$, where $K_i$ is defined to be the hamiltonian vector field with $\mu_i$ as a hamiltonian function, i.e.,
\begin{equation}
\label{eq:defKi}
K_i=J\gr \mu_i.
\end{equation}
These Killing vector fields have been considered in \cite{ApostolovI}. By Lemma~\ref{lem:properties}, the span of these vector fields 
at each point coincides with the span of the vector fields $K(t)$, $t\in \R$, and therefore, they share all the properties that have 
been proven for the vector fields $K(t)$ in Lemma~\ref{lem:properties}. For instance, $\omega(K_i,K_j)=0$, hence, 
$[K_i,K_j]=[JK_i,K_j]=[JK_i,JK_j]=0$ and $K_1\wedge\dots\wedge K_\ell\neq 0$ at each point of $M^0$.

\section{Reduction to the real projective setting}
\label{sec:reduction}

We recall the description of a K\"ahler manifold with a local isometric hamiltonian $\R^\ell$-action
 in \S\ref{sec:generalreduction}. In the setting of c-projectively equivalent K\"ahler metrics, 
this action is given by the commuting Killing vector fields $K_1,\dots,K_\ell$ from \eqref{eq:defKi} 
induced by a hermitian solution $A$ of  \eqref{eq:main}. As stated in \S\ref{sec:idea}, the quotient of the
 K\"ahler manifold $(M,g,J)$ w.r.t. to this action yields a manifold $(Q,g_Q)$ and $g_Q$ admits 
a projectively equivalent metric. This will be described in detail in \S\ref{sec:cprorpro}.

\subsection{The K\"ahler quotient w.r.t. a local isometric hamiltonian $\R^\ell$-action}
\label{sec:generalreduction}

Recall from \cite[\S 3.1]{ApostolovI} that a \emph{local isometric hamiltonian $\R^\ell$-action} 
on a K\"ahler manifold $(M,g,J)$ is given by holomorphic Killing vector fields $K_1,\dots,K_\ell$ 
satisfying 
\begin{align}
\omega(K_i,K_j)=0\nonumber
\end{align}
and $K_1\wedge \dots\wedge K_\ell\neq 0$ on a dense and open subset $M^0\subset M$ called the 
set of regular points. 

Note that in \cite{ApostolovI}, the name ``$\ell$-torus action'' was used instead of ``$\R^\ell$-action''.  
The point is that the metrics in \cite{ApostolovI,ApostolovII} are positive definite so that under the additional 
assumption of compactness, the isometric $\R^\ell$-action described above generates a
 commutative subgroup of the compact group of isometries, its closure being a torus.

The conditions above imply that the vector fields $K_1,\dots,K_\ell$ commute 
with each other and that they are locally hamiltonian, i.e. we have $K_i=J\gr\,\mu_i$ 
for certain local functions $\mu_i$, $i=1,\dots,\ell$. 

By Lemma~\ref{lem:properties}, the canonical Killing vector fields \eqref{eq:defKi} 
coming from a solution $A$  of \eqref{eq:main} generate a local isometric hamiltonian 
$\R^\ell$-action, where $\ell$ is the number of non-constant eigenvalues of $A$ at a regular point. 
The notion of regular points as introduced above coincides with the notion of regular points introduced in  
\S\ref{sec:lc}. However, for the time being, we will forget about c-projective geometry and will first 
restrict to the general setting of a local isometric hamiltonian $\R^\ell$-action. 

Since we are dealing with metrics
 of arbitrary signature, we have to take care of the non-degeneracy of orbits of an $\R^\ell$-action.
A local isometric hamiltonian $\R^\ell$-action given by Killing vector fields $K_1,\dots,K_\ell$
 is called \emph{non-degenerate} if the restriction of the metric $g$ to the (regular) distribution 
$$\mathcal{V}=\mathrm{span}\{K_1,\dots,K_\ell\}$$ on the regular set $M^0$ is non-degenerate.

As shown in Lemma~\ref{lem:properties}~(\ref{l5}), the $\R^\ell$-action coming from the canonical Killing vector fields
corresponding to a solution $A$ of \eqref{eq:main} is non-degenerate in the above sense.

Given a local isometric non-degenerate hamiltonian $\R^\ell$-action, we will now reduce the 
setting by considering the quotient of $M$ w.r.t. the action of the Killing vector 
fields. The procedure of this reduction and the local description of K\"ahler metrics admitting 
a local isometric hamiltonian $\R^\ell$-action can be found in \cite[\S 3.1]{ApostolovI} 
and \cite{PedersenPoon}. For the sake of completeness, we will recall these results. 
The only difference in our case is that the metric $g$ is allowed to have arbitrary 
signature but assuming non-degeneracy, there is actually no difference to the 
procedure described in \cite{ApostolovI}.

Consider a non-degenerate local isometric hamiltonian $\R^\ell$-action on $(M,g,J)$ by
 holomorphic Killing vector fields $K_1,\dots,K_\ell$. Let us restrict our attention to 
the regular set $M^0$ and let $G$ denote the commutative (pseudo-)group generated by the
 local flows of $K_1,\dots,K_\ell$. Consider the (local) quotient $Q=M/G$ 
of $M$ w.r.t. the $G$-action and the (local) fiber bundle  
$$
\pi:M\rightarrow Q=M/G.
$$
The vertical distribution of this bundle coincides with the distribution $\mathcal{V}$ 
 and we define a ($G$-invariant) horizontal distribution $\mathcal{Q}=\mathcal{V}^\perp$.
 Let $\theta=(\theta_1,\dots,\theta_\ell):TM\rightarrow \R^\ell$ be the corresponding connection
 $1$-form on $M$, where the components $\theta_i$ have been chosen to be dual to the
 Killing vector fields $K_i$, that is, the 1-forms $\theta_i$ are defined by
$$
\theta_i(K_j)=\delta_{ij}\mbox{ and }\theta_i(\mathcal Q)=0.
$$

As above, the local generators for the vector fields $K_i$ will be denoted by $\mu_i$
 (so that $K_i=J\gr\,\mu_i$) and we can gather these functions into a (locally defined)
 moment map $\mu=(\mu_1,\dots,\mu_\ell):M\rightarrow (\R^\ell)^*$ for the hamiltonian action of $G$. 
 Lemma~\ref{lem:properties}~\eqref{l6}   implies that the moment map $\mu$ is $G$-invariant,  thus
 it descends to a mapping $\mu:Q\rightarrow (\R^\ell)^*$ on the quotient. 
 The level sets 
$S_\mu$ in $Q$ of this mapping are the K\"ahler quotients of $(M,g,J)$ w.r.t. the isometric 
hamiltonian action of $G$. We refer the reader to \cite[\S 3]{Hitchin} for some background 
on symplectic reduction and K\"ahler quotients.

On the other hand, we can also take the (local) quotient of $M$ w.r.t. the action of the 
commutative (pseudo-)group $G^{\C}$ generated by the local flows of the commuting vector 
fields $K_1,\dots,K_\ell,JK_1,\dots,JK_\ell$. The result is a manifold $S=M/G^{\C}$ and since 
the tangent spaces of the fibers of the bundle $M\rightarrow S$ are $J$-invariant and the
 action of $G^{\C}$ is by holomorphic transformations, $S$ inherits a canonical complex 
structure $J_S$. As a complex manifold, $S$ can canonically be identified with the K\"ahler 
quotients $S_\mu$. In view of this, $S$ carries a family of K\"ahler structures $(g_\mu,\omega_\mu)$ 
which are compatible with the complex structure $J_S$. The quotient $Q$ may locally be written 
in the form $Q=S\times U$, where the open subset $U\subseteq (\R^\ell)^*$ parametrises the
 level sets of $\mu$. In this picture, the subset $U$ can be viewed as the parameter space
 for the family of compatible K\"ahler structures $(g_\mu,\omega_\mu)$ on $S$. 

Since the forms $\theta_i\circ J$ and $\d\mu_i$ span the same subspace of $T^* M$,
 we can define a point-wise non-degenerate matrix of functions $G_{ij}$ and its inverse with
 components $H_{ij}$ by  
\begin{align}
\theta_i\circ J=\sum_{j=1}^\ell G_{ij}\d\mu_j\mbox{ and }\d\mu_i\circ J=-\sum_{j=1}^\ell H_{ij}\theta_j.\label{eq:formulaforJ}
\end{align}
Note that it follows from \eqref{eq:formulaforJ} that 
$$H_{ij}=g(K_i,K_j),$$
in particular, $H_{ij}$ and $G_{ij}$ are symmetric in $i,j$.

The K\"ahler structure can now be written in the form
\begin{align}
\begin{array}{c}
g=g_\mu+\sum_{i,j=1}^\ell H_{ij}\theta_i\theta_j+\sum_{i,j=1}^\ell G_{ij}\d\mu_i \d\mu_j,\vspace{1mm}\\
\omega=\omega_\mu+\sum_{i=1}^\ell\d\mu_i\wedge \theta_i.
\end{array}
\label{eq:formulaforgw}
\end{align}

In our case,  the $\R^\ell$-action induced by a solution  $A$ of \eqref{eq:main} satisfies
 one additional property called {\it rigidity} that essentially simplifies the above local 
formulas for $g$ and $\omega$.

Recall from \cite[\S 3.2]{ApostolovI} that a local hamiltonian $\R^\ell$-action on a K\"ahler 
manifold $(M,g,J)$ given by holomorphic Killing vector fields $K_1,\dots,K_\ell$ is called \emph{rigid}
  if the leaves of the distribution 
$$\mathcal{F}=\mathcal{V}\oplus J\mathcal{V}$$
are totally geodesic (where $\mathcal{V}$ is the vertical distribution of $M\rightarrow Q$).

There are a number of equivalent conditions to this rigidity property, see  
\cite[Proposition 8]{ApostolovI}. We recall this result for convenience of the reader. Note that, although it has been proven 
in \cite{ApostolovI} for positive signature, the proof still works in arbitrary signature assuming non-degeneracy of the $\R^\ell$-action.

\begin{prop}
\cite{ApostolovI}
\label{prop:rigidity}
Consider a local isometric hamiltonian $\R^\ell$-action given by holomorphic Killing vector fields $K_1,\dots,K_\ell$. 
The following assumptions are equivalent:

\begin{enumerate}

\item The action is rigid.

\item  The functions $H_{ij}=g(K_i,K_j)$ are constant on the level surfaces of the 
moment map $\mu:M\rightarrow (\R^\ell)^*$. 

\item $\nabla_{K_i} K_j \in \mathcal{F}$ for all $i,j=1,\dots,\ell$. 

\item The K\"ahler quotient forms $\omega_\mu$ depend affinely on the components $\mu_i$ 
of the moment map $\mu:Q\rightarrow (\R^\ell)^*$
and their linear part pulls back to the curvature of $(\theta_1,\dots,\theta_\ell)$. 

\end{enumerate}
\end{prop}

\begin{rem}
Condition (3) of the proposition can be replaced by ``$\nabla_{K_i} K_j\in J\mathcal V$ for all $i,j=1,\dots,\ell$''.
Indeed, if this holds, $\mathcal F$ is obviously totally geodesic, hence, the action is rigid. The converse direction
follows from the same line of arguments that has been used in the proof of Lemma~\ref{lem:properties}~(8). 
We see that ``$J\mathcal V$ is totally geodesic'' is another condition equivalent to rigidity of the action.
\end{rem}

Proposition~\ref{prop:rigidity}  gives rise to some simplifications in \eqref{eq:formulaforgw} and
 we come to the following local description:

\begin{prop}
\label{prop:locrigid}
Let $(M, g,J,\omega)$ be a K\"ahler $2n$-manifold together with 
a rigid  non-degenerate (local) isometric hamiltonian $\R^\ell$-action
generated by Hamiltonian Killing vector fields $K_i=J\gr\,\mu_i$, $i=1,\dots,\ell$.  
Then locally $M$ can be presented as direct product
$$
V(t_1,\dots,t_\ell) \times U(\mu_1,\dots,\mu_\ell) \times S(y_1,\dots, y_{2k}),
$$  
and $g$, $\omega$ and $J$ take the following form:
\begin{equation}
\label{eq1}
g= \sum_{i,j=1}^\ell H_{ij} (\mu) \theta_i\theta_j  +  \sum_{i,j=1}^\ell G_{ij} (\mu) \d\mu_i \d\mu_j  +  
\sum_{i=1}^\ell \mu_i g_i + g_0,  \quad  \sum_{j=1}^\ell H_{ij}G_{jk}=\delta_{ik}
\end{equation}
\begin{equation}
\label{eq3}
\omega = \sum_{i=1}^\ell d\mu_i \wedge \theta_i +  \sum_{i=1}^\ell \mu_i \omega_i + \omega_0
\end{equation}
and 
\begin{equation}
\label{eq4}
\theta_i\circ J=\sum_{j=1}^\ell G_{ij}\d\mu_j,\,\,\,\d\mu_i\circ J=-\sum_{j=1}^\ell H_{ij}\theta_j,\,\,\,
\d y_i\circ J=\d y_i\circ J_S.
\end{equation}
where the ingredients in these formulas are as follows:

\begin{enumerate}

\item  $\theta_i= \d t_i + \alpha_i$ with $\d\alpha_i=\omega_i$;

\item  $\bigl(g_\mu{=}\sum_i\mu_i g_i+g_0, \  \omega_\mu{=}\sum_i\mu_i\omega_i+\omega_0, \ J_S\bigr)$, 
is a K\"ahler structure on $S$ for any $\mu\in U$  
(compatible with the same complex structure $J_S$ independent of $\mu$);

\item  $\partial_{\mu_i} G_{jk} = \partial_{\mu_j} G_{ik}$.
\end{enumerate}

Conversely,  if on $M=V\times U \times S$ we consider $g$, $\omega$, $J$ as above, 
then $(g, \omega)$ is a K\"ahler structure on $M$ and the generators $\mu_1,\dots, \mu_\ell$ 
define a rigid non-degenerate (local) isometric hamiltonian $\R^\ell$-action. In particular,  
the vector fields $K_i=J\gr\, \mu_i$ are holomorphic Killing vector fields.

\end{prop}

\subsection{Reduction from the c-projective to the projective setting}
\label{sec:cprorpro}

We continue to use  the notation introduced in the preceding section but assume that 
the nondegenerate local isometric hamiltonian $\R^\ell$-action is given by the Killing vector fields $K_1,\dots,K_\ell$ 
from \eqref{eq:defKi}  that come from a certain solution $A$ of \eqref{eq:main}. Let $g_Q$ denote the metric 
on the quotient $Q=M/G$ obtained from $g$.  In the notation of Proposition~\ref{prop:locrigid}, 
$Q$ can locally be identified with $U\times S$ and $g_Q$ can be obtained from \eqref{eq1} by removing the first term, i.e.
\begin{equation}
\label{eq:formforg_Q}
g_Q =  \sum_{i,j=1}^\ell G_{ij} (\mu) \d\mu_i \d\mu_j  +  
\sum_{i=1}^\ell \mu_i g_i + g_0. 
\end{equation}

Recall that the vertical distribution $\mathcal{V}=\mathrm{span}\{K_1,\dots,K_\ell\}$ coincides 
with the span of the vector fields $J\gr\,\rho_1,\dots,J\gr\,\rho_\ell$,
 where $\rho_1,\dots,\rho_\ell$ are the non-constant eigenvalues of $A$. Since by Lemma~\ref{lem:properties}, 
the vector fields  $J\gr\,\rho_i$ take values in the eigenspaces of $A$, 
 the distribution $\mathcal{V}$  is $A$-invariant and consequently, $A$ preserves also $\mathcal{Q}=\mathcal{V}^\perp$.
 On the other hand, according to Lemma~\ref{lem:properties}~(\ref{l7}),  $A$ is preserved by 
the Killing vector fields $K_i$ and it follows that $A$ 
descends to a $g_{Q}$-selfadjoint endomorphism $A_Q:TQ\rightarrow TQ$.

Recall the O'Neill formula \cite{ONeill} for a Riemannian submersion relating the Levi-Civita 
connection $\nabla^Q$ of the quotient metric $g_Q$ to the Levi-Civita connection $\nabla$ of $g$ by
\begin{align}
\nabla^Q_X Y=\mathrm{pr_{\mathcal{Q}}}(\nabla_{X} Y),\label{eq:ONeill}
\end{align}
where $\mathrm{pr_{\mathcal{Q}}}:TM\rightarrow \mathcal{Q}$ is the projection onto the horizontal 
distribution $\mathcal{Q}$ and we adopted the convention to denote vector fields 
on $Q$ and their horizontal lifts to $\mathcal{Q}$ by the same symbol. Note that the vector field
 $\Lambda$ in \eqref{eq:main} is tangent to the horizontal distribution $\mathcal{Q}$ and it is 
invariant w.r.t. the action of the $K_i$'s. Thus, $\Lambda$ is (the horizontal lift of) a 
vector field on $Q$.

\begin{lem}
\label{lem:redtorealproj}
The endomorphism $A_Q:TQ\rightarrow TQ$ obtained from $A$ by reduction satisfies the equation 
\begin{align}
\nabla^Q_X A_Q=X^\flat \otimes \Lambda+\Lambda^\flat\otimes X\label{eq:mainreal0}
\end{align}
for all $X\in TQ$, where $X^\flat=g_Q(X,\cdot)$.  In other words, $g_Q$ and $A_Q$ are 
compatible in the projective sense.
\end{lem}

\begin{proof}
From the O'Neill formula \eqref{eq:ONeill},  the definition of $A_Q$ and  commutativity 
of $A$ with $\mathrm{pr}_{\mathcal{Q}}$, it follows that
$$(\nabla^Q_X A_Q)Y=\nabla^Q_X(A_Q Y)-A_Q(\nabla^Q_X Y)=\mathrm{pr}_{\mathcal{Q}}( (\nabla_{X}A)Y).$$
Inserting \eqref{eq:main} into this equation and using the fact that $J\Lambda$ is tangent to 
$\mathcal{V}=\mathcal{Q}^\perp$, we obtain 
$$(\nabla^Q_X A_Q)Y=g(X,Y)\Lambda+g(\Lambda,Y)X=g_Q(X,Y)\Lambda+g_Q(\Lambda,Y)X$$
as we claimed.
\end{proof}

The local description of compatible pairs $(g_Q, A_Q)$ has been recently
 obtained in \cite{BM} (see also \cite{BMgluing}).  These results, after some adaptation,
 will lead us to the local description of pairs $(g, A)$.

\section{Local description of c-projectively equivalent metrics}
\label{sec:locclass}

In this section we prove Theorems~\ref{thm:invform}   and~\ref{thm:localclassification}.

\subsection{Local description of the quotients of c-projectively equivalent metrics and lifting}
\label{sec:normalformsquotient}

We have shown above that by taking the quotient of $M$ w.r.t. the action of the Killing 
vector fields $K_1,\dots,K_\ell$, the local description of a K\"ahler manifold 
$(M,g,J)$ of arbitrary signature admitting a hermitian solution $A$ of \eqref{eq:main}  is reduced to the classification of  pseudo-Riemannian
 manifolds $(Q,g_Q)$ admitting a $g_Q$-selfadjoint solution $A_Q$ to \eqref{eq:mainreal0}. In other words, a description of 
c-compatible pairs $g, A$ is reduced to a similar problem for compatible pairs $g_Q$, $A_Q$ on the quotient $Q=M/G$ which 
has been solved  in \cite{BM} and we apply this result to our situation. 


Before deriving the local description for the pair $(g_Q, A_Q)$ in our specific situation, we briefly 
recall the ``splitting and gluing constructions'' from \cite{BMgluing} appropriately reformulated for our purposes, 
we refer to \cite[\S 1.2]{BM} for a more detailed summary.

Let $(Q,g_Q)$ be a pseudo-Riemannian manifold  and $A_Q:TQ\rightarrow TQ$ be a $g_Q$-selfadjoint endomorphism compatible 
with $g_Q$ in the projective sense, i.e., satisfying \eqref{eq:mainreal0}.

In a neighbourhood of a generic point, the eigenvalues of $A_Q$ are smooth (possibly complex valued functions).   
Some of them,  say $c_1, \dots, c_n$, are constant.   Then the characteristic polynomial 
$\chi(t)=\mathrm{det}(t\cdot\Id-A_Q)$ of $A_Q$ can be written as
$$
\chi(t) = \chi_{\mathrm{nc}}(t) \cdot \chi_{\mathrm{c}}(t)
$$
where the roots of $\chi_{\mathrm{c}}$  are the constant eigenvalues of $A_Q$  
(with multiplicities), whereas the roots of $\chi_{\mathrm{nc}}$ are the non-constant eigenvalues.  
Assume that these polynomials $\chi_{\mathrm{nc}}(t)$ and $\chi_{\mathrm{c}}(t)$ are relatively 
prime, i.e.,  the non-constant eigenvalues cannot take the values  $c_1, \dots, c_n$.  In other words,
 we divide the spectrum of $A_Q(p)$, $p\in Q$ into the ``constant'' and ``non-constant'' parts,
 and assume that these parts are disjoint for any $p\in Q$.

\begin{prop}
\label{prop:split1}\cite{BMgluing}
Locally $Q$ can be presented as $U(x_1,\dots, x_\ell)\times S(y_1,\dots,y_{s})$
so that the  endomorphism $A_Q$ and the metric $g_Q$ take the following block-diagonal form
\begin{equation}
\label{eq:glspecial}
A_Q(x,y) = 
\begin{pmatrix} 
L (x) & 0\\0 &  A_{\mathrm{c}}(y) 
\end{pmatrix}
\quad\mbox{and}\quad
g_Q (x,y)= 
\begin{pmatrix}   
h (x) & 0 \\ 0&  g_{\mathrm{c}}(y) \cdot \chi_{\mathrm{nc}}\bigl(A_{\mathrm{c}}(y)\bigr)   
\end{pmatrix},
\end{equation}
where $L$ and $h$ are compatible (that is, satisfy 
\eqref{eq:main:proj}) on $U$,  and $A_{\mathrm{c}}$ is parallel on $S$  w.r.t. $g_{\mathrm{c}}$.

Conversely, $A_Q(x,y)$ and $g_Q(x,y)$  defined by \eqref{eq:glspecial}  are compatible in the projective sense, 
i.e., satisfy \eqref{eq:mainreal0} on $U\times S$,  whenever $h$ and $L$ are compatible, $A_{\mathrm{c}}$ is 
parallel w.r.t. $g_{\mathrm{c}}$ and the spectra of $L$ and $A_{\mathrm{c}}$ are disjoint.

\end{prop}

Notice that the formula for the metric $g_Q$ can be equivalently rewritten as follows
\begin{equation}
\label{splitforg}
g_Q = \sum_{i,j=1}^\ell B_{ij}(x) \d x_i\d x_j  + \sum_{i=1}^\ell \mu_i (x) g_i + g_0,  
  \quad  g_i = (-1)^i g_{\mathrm{c}}(A_\mathrm{c}^{\ell-i}\cdot, \cdot)
\end{equation}
which completely agrees with the formula \eqref{eq:formforg_Q} for the reduced metric $g_Q$.  
Here the first term corresponds to 
the metric $h$ and the remaining terms represent the other block, i.e., the 
metric $g_{\mathrm{c}}(y) \cdot \chi_{\mathrm{nc}}\bigl(A_{\mathrm{c}}(y)\bigr)$ which can 
be understood as a family  $g_\mu=\sum \mu_i g_i + g_0$ of metrics on $S$ parametrised by the 
coefficients $\mu_1,\dots, \mu_\ell$ of the characteristic polynomial $\chi_{\mathrm{nc}}=\chi_L$ 
of the ``non-constant'' block $L$.
Notice that the splitting of $Q$ into the direct product $U\times S$ in both cases is 
determined by the decomposition of $T_pQ$ into two  $A_Q$-invariant subspaces 
corresponding to the partition of the spectrum of $A_Q$  into two parts, ``constant'' and ``non-constant''.  
 Also notice that in the both cases $\mu_i$ are the same:  these are the elementary symmetric 
polynomials of non-constant eigenvalues of $L$ (or, which is the same, of $A$).

Formula \eqref{splitforg} describes, however, a more general situation than \eqref{eq1}. 
 In particular,  in Proposition~\ref{prop:split1},  the non-constant eigenvalues may have 
arbitrary multiplicities  and the ``constant''  block  $(S, g_{\mathrm{c}}, A_{\mathrm{c}})$
  carries no K\"ahler structure.  Thus, some additional specific properties of $g_Q$ and $A_Q$ should be taken into account.
		In particular,  we need  local formulas for the metric which simultaneously satisfies
		\eqref{splitforg} and \eqref{eq1}.

 As we know from  Lemma~\ref{lem:properties}~(1),  the multiplicities of the non-constant
 eigenvalues $\rho_1,\dots, \rho_\ell$ of $A_Q$ equal one and moreover  $\d\rho_i  \ne 0$ on $Q$.
 This condition guaranties that both  the eigenvalues $\rho_1,\dots, \rho_\ell$  and the symmetric 
polynomials $\mu_1,\dots,\mu_\ell$ can be taken as local coordinates on $U$.  
Also we know from \eqref{eq1} that $S$ is endowed with a natural complex structure $J_S$ and 
for each $\mu\in U$, the metric 
$$
g_\mu = \sum_{i=1}^\ell \mu_i g_i + g_0
$$
on $S$ is K\"ahler and $A_{\mathrm{c}}$  on $S$ is hermitian w.r.t. $(g_\mu,J_S)$.    
In addition $A_{\mathrm{c}}$ is parallel  w.r.t.  
$g_{\mathrm{c}} =g_\mu(\chi_{\mathrm{nc}}(A_{\mathrm {c}})^{-1} \cdot , \cdot)$ 
 by Proposition~\ref{prop:split1}.   This obviously implies that the metrics $g_{\mathrm{c}}$ 
and $g_\mu$  are affinely equivalent for each $\mu$, i.e.,  their Levi-Civita connections coincide.
  Hence, if we introduce 
	$\omega_{\mathrm{c}} = g_{\mathrm{c}}( J_S\cdot,\cdot) = \omega_\mu (\left(\chi_{\mathrm{nc}}(A_{\mathrm{c}})\right)^{-1}\cdot,\cdot)$, 
	then  $\omega_{\mathrm{c}}$ is parallel and therefore  $(g_{\mathrm{c}},\omega_{\mathrm{c}}, J_S)$ 
	is a K\"ahler structure on $S$ admitting a parallel hermitian endomorphism $A_{\mathrm{c}}$  
	(in other words the conclusion about the constant block in Proposition~\ref{prop:split1} now holds in the K\"ahler setting).

Summarizing, we see that the pair $(g_Q, A_Q)$  admits the following local description.

\begin{prop}
\label{prop:gQAQ}
Using the natural decomposition $Q=U(x_1,\dots,x_\ell)\times S(y_1,\dots, y_{2k})$, we can write $g_Q$ 
and $A_Q$ as follows
\begin{equation}
A_Q (x,y)= 
\begin{pmatrix} 
L (x) & 0\\0 &  A_{\mathrm{c}}(y) 
\end{pmatrix}
\quad\mbox{and}\quad
g_Q(x,y)= 
\begin{pmatrix}   
h (x) & 0 \\0 &  g_{\mathrm{c}}(y) \cdot \chi_{L}(A_{\mathrm{c}})  
\end{pmatrix}
\end{equation}
where 

\begin{itemize}

\item $(L, h)$ is a compatible pair on $U$ (in the projective sense)  such that the eigenvalues
 $\rho_1, \dots, \rho_\ell$ of $L$ are all distinct and $\d\rho_i\ne 0$. Moreover, 
$\chi_L(t)=\mathrm{det}(t\cdot\Id-L)$ denotes the characteristic polynomial of $L$;

\item $(S, g_{\mathrm{c}}, J_S)$ is a K\"ahler manifold and $A_{\mathrm{c}}$ is a parallel hermitian endomorphism on $S$.

\end{itemize}
\end{prop}

The metric $h=\sum_{i,j} B_{ij} \d x_i \d x_j$ can be 
rewritten in coordinates $\mu_1,\dots, \mu_\ell$
$$
h=\sum_{i,j=1}^\ell B_{ij} \d x_i \d x_j = \sum_{\alpha,\beta=1}^\ell G_{\alpha\beta} \d\mu_\alpha \d\mu_\beta,  
\qquad   B_{ij} = \sum_{\alpha,\beta=1}^\ell G_{\alpha\beta}\frac{\D\mu_\alpha}{\D x_i} \frac{\D\mu_\beta}{\D x_j}
$$
as in Proposition~\ref{prop:locrigid}.     
As we know from this proposition,  the components $G_{ij}$ must satisfy one additional condition,
 namely  
$\frac{\D G_{ij}}{\D \mu_k} = \frac{\D G_{kj}}{\D \mu_i}$.  
It turns out (see Proposition~\ref{prop:realising} and Lemma~\ref{lem:anotherone}  below)
 that this property follows automatically from the compatibility of $h$ and $L$. 
 This means that we have no more restrictions onto the reduced pair $g_Q$ and $A_Q$, 
and  can now summarize the above discussion as follows.

\begin{prop}
\label{prop:locform}
Let  $g, \omega$ and $A$ be as above.  Then in a neighbourhood of a regular point $p\in M^0$ 
we can introduce a local coordinate system
$$
V(t_1,\dots, t_\ell) \times U(x_1,\dots,x_\ell) \times S(y_1,\dots, y_{2k})
$$ 
in which $g$,  $\omega$ and $A$ take the following form
\begin{equation}
\label{eq:forH}
g= \sum_{\alpha,\beta=1}^\ell H_{\alpha\beta} \theta_\alpha \theta_\beta + \sum_{i,j=1}^\ell B_{ij}\d x_i \d x_j 
+  \sum_{i=0}^\ell \mu_i \cdot (-1)^i g_{\mathrm{c}}(A_{\mathrm{c}}^{\ell -i} \cdot, \cdot)
\end{equation}

\begin{equation}
\label{eq:foromega}
\omega = \sum_{\alpha=1}^\ell \d\mu_\alpha \wedge \theta_\alpha + \sum_{i=0}^\ell \mu_i \cdot (-1)^i  \omega_{\mathrm{c}}(A_{\mathrm{c}}^{\ell -i} \cdot, \cdot)
\end{equation}

\begin{equation}
\label{eq:forA}
A=\sum_{\alpha,\beta=1}^\ell  M^{\beta}_\alpha (x) \, \theta_\beta  \otimes  \frac{\partial}{\partial {t_\alpha}}  
+ \sum_{i,j=1}^\ell L^i_j (x) \,\d x_j  \otimes \frac{\partial}{\partial {x_i}} 
+  \sum_{p,q=1}^{2k}    (A_{\mathrm{c}})^q_p  \,  \d y_p  \otimes \left( \frac{\partial}{\partial {y_q}}    - \sum_{i=1}^\ell\alpha_{iq} \frac{\partial}{\partial {t_i}}   \right),
\end{equation}
 

where the ingredients in these formulas are as follows:

\begin{enumerate}

\item $(g_{\mathrm{c}},\omega_{\mathrm{c}})$ is a K\"ahler structure and 
$A_{\mathrm{c}}= \sum_{p,q}    (A_{\mathrm{c}})^q_p  \,  \d y_p\otimes  \partial_{y_q}$
 is a parallel hermitian endomorphism on $S$;

\item  $h=B_{ij}(x)\d x_i \d x_j$ is a pseudo-Riemannian metric and $L(x)$
 is an endomorphism on $U$ forming a compatible pair (in the projective sense);

\item  the eigenvalues $\rho_1,\dots, \rho_\ell$ of $L$ are pairwise distinct 
and satisfy $\d\rho_i\ne 0$ on $U$;  they are also different from the constant eigenvalues 
of $A_{\mathrm{c}}$;

\item  $\mu_i$ denote the elementary symmetric polynomials in $\rho_1,\dots,\rho_\ell$, $i=1,\dots, \ell$, 
and we set $\mu_0=1$;

\item $\theta_i=\d t_i + \alpha_i$, where $\alpha_i = \sum_{q}  \alpha_{iq} \d y_q$ is a 1-form on $S$ satisfying 
$\d\alpha_i = (-1)^i  \omega_{\mathrm{c}}(A_{\mathrm{c}}^{\ell -i} \cdot, \cdot)$;

\item and finally $H_{\alpha\beta} = \sum_{i,j} B^{ij} \frac{\D \mu_\alpha}{\D x_i} \frac{\D \mu_\beta}{\D x_j}$, 
where $B^{ij}$  is the inverse of  $B_{ij}$ and  $M^{\beta}_\alpha = \sum_{i,j} L^i_j \frac{\partial \mu_\beta}{\partial x_i} \frac{\partial x_j}{\partial \mu_\alpha}$.

\end{enumerate}
\end{prop}

\begin{proof}
The formulas \eqref{eq:forH} and \eqref{eq:foromega} follow from the discussion above. 
It remains to derive formula \eqref{eq:forA} for $A$. 
First of all we note that the basis dual to the coframe $\theta_i,\d x_j,\d y_q$ is given by 
$$
\frac{\D}{\D t_i},\quad\frac{\D}{\D x_j},\quad\frac{\D}{\D y_q}-\sum_{i=1}^\ell\alpha_{iq}\frac{\D}{\D t_i}.
$$
We see that the reduction of $A$ given by \eqref{eq:forA} is indeed given by $A_Q$ from Proposition \ref{prop:gQAQ}.
It remains to show how $A$ acts on the Killing vector fields $\D_{t_i}$.
Formula \eqref{eq:foromega} shows that $i_{\D_{t_\beta}}\omega=-\d\mu_\beta$, hence, $\D_{t_\beta}=J\gr\,\mu_\beta$.
Using that $A$ commutes with $J$ and that $L$ is $h$-selfadjoint, we obtain
$$
A\frac{\D}{\D t_\beta}=JA(\gr\,\mu_\beta)=JL(\gr_h\,\mu_\beta)=\sum_{i,j,\alpha=1}^\ell L^i_j\frac{\D\mu_\beta}{\D x_i}\frac{\D x_j}{\D\mu_\alpha}\frac{\D}{\D t_\alpha}
=\sum_{\alpha=1}^\ell M^\beta_\alpha\frac{\D}{\D t_\alpha}.
$$
which establishes formula \eqref{eq:forA}.
\end{proof}

Thus, we are lead to the situation described in Example~\ref{ex:now2} and, therefore,  the second part of Theorem~\ref{thm:invform} is proved.

The main ingredients in the above local formulas are the pair $(h, L)$
 on $U$ and the triple $(g_{\mathrm{c}},\omega_{\mathrm{c}}, A_{\mathrm{c}})$ on $S$. 
The 1-forms $\alpha_i$ on $S$ are determined by $(\omega_{\mathrm{c}}, A_{\mathrm{c}})$ only up to the transformation
$\alpha_i\longmapsto \alpha_i+\d f_i$ for arbitrary functions $f_i$ on $S$. However, such functions define
a fiber-preserving local transformation $f:M\rightarrow M$, $f(t,x,y)=(t_1+f_1(y),\dots,t_\ell+f_\ell(y),x,y)$, that fulfils $f^*\theta_i=\theta_i+\d f_i$,
$f^*\d x_j=\d x_j$ and $f^*\d y_q=\d y_q$ and pulls back the objects in Proposition \ref{prop:locform}  written down w.r.t. $\theta_i$ 
to the corresponding objects written down w.r.t. $\tilde \theta_i=\theta_i+\d f_i$. 
All the other ingredients appearing in the formulas of Proposition \ref{prop:locform} 
can be uniquely reconstructed from $(h, L)$ and $(g_{\mathrm{c}},\omega_{\mathrm{c}}, A_{\mathrm{c}})$. 
However, we do not know yet whether these ingredients can be arbitrarily chosen or should,
 perhaps, satisfy some additional restrictions which are not mentioned in Proposition~\ref{prop:locform}.  
The next section shows that there are no more restrictions and  
\eqref{eq:forH}, \eqref{eq:foromega} and \eqref{eq:forA}  can be used for the local 
description of c-compatible $g$ and $A$.  To that end, we only need to substitute into these formulas the local normal forms
 for $(h, L)$  and  $(g_{\mathrm{c}},\omega_{\mathrm{c}}, A_{\mathrm{c}})$ 
which were previously found in \cite{BM} and \cite{Boubel} respectively.

\subsection{Realisation}\label{sec:realisation}

The purpose of this section is to prove the following result which is equivalent to the first part of Theorem~\ref{thm:invform}.

\begin{prop}
\label{prop:realising}
Let $h=\sum_{i,j=1}^\ell B_{ij}(x)\d x_i \d x_j$ be a pseudo-Riemannian metric and 
$L(x)$ an endomorphism on $U$ forming a compatible pair (in the projective sense) and let
$(g_{\mathrm{c}},\omega_{\mathrm{c}})$ be a K\"ahler structure of arbitrary signature and $A_{\mathrm{c}}$ 
a parallel endomorphism on $S$.   Suppose that the eigenvalues of $L$ 
and $A_{\mathrm{c}}$ satisfy condition (3) from Proposition~\ref{prop:locform}  
 and  1-forms $\alpha_i$ on $S$ are chosen as in condition (5). Then 
 
 \begin{itemize} 
 
 \item $g$ and $\omega$ given by \eqref{eq:forH} and \eqref{eq:foromega} 
 define a K\"ahler structure on $V\times U \times S$;
 
 \item $A$  given by \eqref{eq:forA} is hermitian w.r.t. $(g,\omega)$  
and satisfies  \eqref{eq:main}, in other words $A$ and $(g,\omega)$ are c-compatible.

\end{itemize}
\end{prop}

\begin{proof}

To verify that $g$ and $\omega$ define a K\"ahler structure, we use Proposition~\ref{prop:locrigid}.   
Formulas \eqref{eq:forH} and \eqref{eq:foromega} are
 similar to \eqref{eq1}  and \eqref{eq3} but we still need to verify some conditions.
  First of all,  we can use $\mu_1,\dots, \mu_\ell$ as local coordinates on $U$ to
	rewrite the term $h=\sum_{i,j} B_{ij}\d x_i \d x_j$  as
$\sum_{\alpha,\beta} G_{\alpha\beta}\d \mu_\alpha \d \mu_\beta$, where 
$B_{ij} = \sum_{\alpha,\beta} G_{\alpha\beta}\frac{\D\mu_\alpha}{\D x_i} \frac{\D\mu_\beta}{\D x_j}$
 and then the matrices $H_{\alpha\beta}$ and $G_{\alpha\beta}$ are inverse to each other as required
 in Proposition~\ref{prop:locrigid}.   Next we need to check that 
$g_\mu=\sum_{i=0}^\ell \mu_i \cdot (-1)^i g_{\mathrm{c}}(A_{\mathrm{c}}^{\ell -i} \cdot, \cdot)$ 
and $\omega_\mu=\sum_{i=0}^\ell \mu_i \cdot (-1)^i  \omega_{\mathrm{c}}(A_{\mathrm{c}}^{\ell -i} \cdot, \cdot)$ 
define a K\"ahler structure on $S$ for any $\mu$,  but this condition immediately follows from the
 fact that $(g_{\mathrm{c}}, \omega_{\mathrm{c}})$ is K\"ahler and $A_{\mathrm{c}}$ is hermitian
 and parallel with respect to it.  Notice that the complex structure $J_S$ is, 
by construction, the same for all $(g_\mu,\omega_\mu)$.

Less trivial is the fact that $h$ is a Hessian metric in the coordinates
 $\mu_1,\dots, \mu_\ell$, i.e., that $\partial_{\mu_i} G_{jk} = \partial_{\mu_j} G_{ik}$ holds 
for all $i,j,k$ (condition (3) from Proposition~\ref{prop:locrigid}).  To prove it,  we first 
notice that this condition is equivalent to the fact that the vector fields $\gr\,\mu_1,\dots,\gr\,\mu_\ell$ commute.
 Indeed,  $\partial_{\mu_i} G_{jk} = \partial_{\mu_j} G_{ik}$ means that the 1-forms $\beta_k=\sum_{i} G_{ik}\d \mu_i$
 are all closed.  Hence,  the statement immediately follows from the observation that the forms 
$\beta_1,\dots, \beta_\ell$  are dual to the vector fields $\gr\,\mu_1,\dots,\gr\,\mu_\ell$, i.e., $\beta_k(\gr\,\mu_j)=\delta_{kj}$.

Thus, it remains to prove the following lemma (cf.~Lemma~\ref{lem:properties}~\eqref{l6} which is a c-projective analogue of this statement). 

\begin{lem}
\label{lem:anotherone}
Let $h$ be a pseudo-Riemannian metric on $U\subset \R^\ell$  and $L$ be an $h$-selfadjoint endomorphism compatible with $h$ in the projective sense. Let 
$$
\mathrm{det}(t\cdot\mathrm{Id}-L)=\sum_{i=0}^\ell(-1)^i\mu_it^{\ell-i},
$$
where the functions $\mu_{i}$, $i=1,\dots,\ell$, are the elementary symmetric functions in the eigenvalues of $L$ and $\mu_0=1$. 
Then, 
$$
[\gr\,\mu_i,\gr\,\mu_j]=0\mbox{ for all }i,j.
$$
\end{lem}

\begin{proof}
First of all, since the vector fields $\gr\,\mu_1,\dots,\gr\,\mu_\ell$ are 
constant linear combinations of the vector fields of the form $v_t=\gr\,\mathrm{det}(t\cdot\mathrm{Id}-L)$ 
for $t\in \R$ and vice versa, it suffices to prove that
\begin{align}
[v_t,v_s]=0\label{eq:KtKscommute}
\end{align}
for all $t,s\in \R$. Moreover, it suffices to prove \eqref{eq:KtKscommute} for
 $t,s$ which do not belong to the spectrum of $L$ locally in a neighbourhood of a point. For $L:TM\rightarrow TM$ 
an arbitrary non-degenerate endomorphism and $X$ an arbitrary vector field, recall the
general formula
$$
X(\mathrm{det}\,L)=(\mathrm{det}\,L)\mathrm{tr}(L^{-1}\nabla_X L).
$$
In our case, $L$ and therefore $L_s=s\cdot\Id-L$ satisfy \eqref{eq:main:proj}, i.e.
\begin{align}
\nabla_X (s\cdot\Id-L)=-\nabla_X L=-X^\flat\otimes \Lambda-\Lambda^\flat\otimes X\label{eq:Lmain}
\end{align}
holds for a certain vector field $\Lambda$. Defining $f_s=\mathrm{det}\,L_s$ and 
combining the previous two equations we obtain
\begin{align}
X(f_s)=-2f_s h(X,L_s^{-1}\Lambda).\label{eq:verifycompat1}
\end{align}
or equivalently, $v_s=\gr\,f_s=-2f_sL_s^{-1}\Lambda$.
Note that this formula is meaningful and holds true even if $s$ is in the spectrum of $L$. 
We calculate
$$
\nabla_X v_s=2f_s\left[h(X,L_s^{-1}\Lambda)L_s^{-1}\Lambda-h(\Lambda,L_s^{-1}\Lambda)L_s^{-1} X-L_s^{-1}\nabla_X\Lambda\right].
$$
It is a well-known statement in projective geometry that the endomorphisms $L$ 
and $\nabla\Lambda$ commute, see for example the discussion below Theorem 7 in \cite{Fubini}.
 Replacing $X$ by $v_t$ in the last equation and using $[L_t^{-1},\nabla\Lambda]=0$, we obtain  
$$
\nabla_{v_t} v_s=4f_s f_t\left[-h(L_t^{-1}\Lambda,L_s^{-1}\Lambda)L_s^{-1}\Lambda
+h(\Lambda,L_s^{-1}\Lambda)L_s^{-1}L_t^{-1} \Lambda+L_s^{-1}L_t^{-1}\nabla_{\Lambda}\Lambda\right].
$$
Thus,
$$
[v_t,v_s]=\nabla_{v_t} v_s-\nabla_{v_s} v_t
$$
$$
=4f_s f_t\left[-h(\Lambda,L_s^{-1}L_t^{-1}\Lambda)(L_s^{-1}-L_t^{-1})\Lambda+h(\Lambda,(L_s^{-1}-L_t^{-1})\Lambda)L_s^{-1}L_t^{-1} \Lambda\right].
$$
Inserting the identity $L_s^{-1}-L_t^{-1}=(t-s)L_s^{-1}L_t^{-1}$ into the last equation,
 we obtain \eqref{eq:KtKscommute} as we claimed.
\end{proof}

Applying this lemma to $h$ and $L$ from Proposition~\ref{prop:realising},
we get  condition (3) from Proposition~\ref{prop:locrigid}.   
Thus,  now Proposition~\ref{prop:locrigid} implies that 
$g$ and $\omega$ given by \eqref{eq:forH} and \eqref{eq:foromega}  indeed
 define a K\"ahler structure on $V\times U \times S$ which completes the proof of the first statement of Proposition~\ref{prop:realising}.

\medskip

It is easy to see that $A$ is hermitian w.r.t. $(g,\omega)$. It remains to show that $A$ satisfies  \eqref{eq:main} 
and we will proceed as follows.  Consider the hermitian metric 
\begin{align}
\label{eq:hatg}
\hat{g}=(\mathrm{det}_{\mathbb{C}}\,A)^{-1}g(A^{-1}\cdot,\cdot)
\end{align}
obtained from $g$ and $A$ by solving \eqref{eq:defA} w.r.t. $\hat{g}$. First, we show that $\hat{g}$ 
is a K\"ahler metric on $(M,J)$. Then we show that $g$ and $\hat{g}$ are c-projectively equivalent. 
This implies that $A$ satisfies \eqref{eq:main} and we are done.

\begin{prop}
\label{prop:hatgiskaehler}
The metric $\hat{g}$ is a K\"ahler metric on $(M,J)$.
\end{prop}

\begin{proof} 
We use the (local) formulas for $g$, $\omega$ and $A$ from Proposition~\ref{prop:locform} and the notation introduced there.    
Since $\hat g$ is hermitian w.r.t. $J$ by construction, we only need to check that 
$\hat \omega = \hat g(J \cdot , \cdot)=
(\mathrm{det}_{\mathbb{C}}\,A)^{-1}\omega (A^{-1}\cdot,\cdot)$ is closed. 

Notice that $\det_{\mathbb C} A =c\cdot \mu_\ell$ for some constant $c$  so that we may, 
without loss of generality, replace $\hat \omega$ by $\mu_\ell^{-1} \omega  (A^{-1}\cdot,\cdot)$.

Then, by using \eqref{eq:foromega} and \eqref{eq:forA},  we get
$$
c\cdot \hat \omega = \sum_{i=1}^\ell \mu_\ell^{-1} \bigl(\dd \mu_i \circ L^{-1}\bigr) \wedge \theta_i 
+ \sum_{k=0}^\ell \frac{\mu_k}{\mu_\ell}  \omega_k (A_{\mathrm c}^{-1}\cdot, \cdot),
$$
where $\omega_k = (-1)^k \omega_{\mathrm{c}} (A_{\mathrm{c}}^{\ell - k}\cdot,\cdot)$.

The closeness of $\hat \omega$ now follows from two facts  

\begin{itemize}

\item  $ \mu_\ell^{-1} \bigl(\dd \mu_i \circ L^{-1}\bigr) = -\dd \hat {\mu}_{\ell +1 - i}$, 
where $\hat \mu_k$ denotes the $k$th elementary symmetric polynomial in the eigenvalues of $L^{-1}$, i.e. in 
$\rho_1^{-1}, \dots, \rho_\ell^{-1}$.   This is a general property of a compatible pair $(h, L)$, see Lemma~\ref{lem:aboutmu} below.

\item  $\sum_{k=0}^\ell \frac{\mu_k}{\mu_\ell}  \omega_k (A_{\mathrm c}^{-1}\cdot, \cdot) = 
           \sum_{k=0}^\ell \hat \mu_{\ell-k}\omega_{k+1}  =
            \sum_{i=1}^\ell  \hat \mu_{\ell+1-i}\omega_{i}  +   \omega_{\ell + 1}$.   This relation is straightforward.

\end{itemize}

Hence 
$$
c\cdot \hat \omega =  \sum_{i=1}^\ell \dd \hat {\mu}_{\ell +1 - i} \wedge \theta_i +   \sum_{i=1}^\ell  \hat \mu_{\ell+1-i}\omega_{i}  +   \omega_{\ell + 1}
$$
and the property $\d\hat\omega = 0$ becomes obvious, as $\omega_k$'s are all closed and $\d\theta_i=\omega_i$ by construction.

Thus, in order to complete the proof of Proposition~\ref{prop:hatgiskaehler} it remains to prove

\begin{lem}
\label{lem:aboutmu}
Let  $h$ and $L$ be compatible in the projective sense. Then the following relation holds
$$
\frac{1}{\det L} \bigl(\dd \mu_i \circ L^{-1}\bigr) = -\dd \hat {\mu}_{\ell +1 - i}
$$ 
for $i=1,\dots,\ell$, where $\hat \mu_k$ is the $k$th symmetric polynomial in $\rho_1^{-1}, \dots , \rho_\ell^{-1}$.
\end{lem}

\begin{proof}
Recall that the compatibility condition \eqref{eq:main:proj}  implies that the Nijenhuis torsion of $L$ vanishes
(see for instance \cite[Theorem 1]{BolMatBen}).
 Lemma 10 from \cite{BMgluing}   states that for such $L$ the following formula holds:
$$
\dd\chi_L(t) \circ L - t \cdot \dd\chi_L(t) = \chi_L(t) \cdot \dd\tr L,
$$
where $\chi_L(t)=\mathrm{det}(t\cdot\Id-L)$ is the characteristic polynomial of $L$.
Let us multiply the both sides of this formula by $L^{-1}$:
$$
\dd\chi_L(t)- t \cdot \dd\chi_L(t) \circ L^{-1}= \chi_L(t) \cdot \dd\tr L \circ L^{-1}
$$
Hence,
$$
\dd\chi_L(t) \circ L^{-1} = \frac{1}{t} \left(  \dd\chi_L(t) - \chi_L(t) \cdot \dd\tr L \circ L^{-1}\right)
$$
Using another nice formula $\dd\tr L \circ L^{-1} = \dd ( \ln\det L)$, we get
$$
 \dd\chi_L(t) \circ L^{-1} = \frac{1}{t} \left(  \dd\chi_L(t) - \chi_L(t) \cdot \dd ( \ln\det L)  \right) = 
\frac{1}{t} \frac {\det L \cdot \dd\chi_L(t) - \chi_L(t) \cdot \dd \det L}{\det L} =
$$
$$
\frac{\det L}{t} \frac {\det L \cdot \dd\chi_L(t) - \chi_L(t) \cdot \dd \det L}{\det^2 L} = \frac{\det L}{t} \cdot \dd\left(   \frac{\chi_L(t)}{\det L}   \right),
$$
or equivalently
$$
\frac{1}{\det L} \dd\chi_L(t) \circ L^{-1} = t^{-1}  \dd\left(   \frac{\chi_L(t)}{\det L}   \right),
$$
which coincide with the desired relation if we take into account that $\chi_L(t) =  \sum_{i=0}^\ell (-1)^i \mu_i t^{\ell-i}$ 
and consider $t$ as a formal parameter. 
\end{proof}

\begin{rem}
We can derive the formulas in Lemma~\ref{lem:aboutmu} in an alternative way: by the same arguments as used in the proof 
of Lemma~\ref{lem:properties}~(7), one easily derives the formula
\begin{equation}
\label{eq:gradchar}
\gr_{\hat h}\chi_{L^{-1}}(1/t)=\frac{(-1)^\ell}{t^{\ell-1}}\gr_h\chi_L(t),
\end{equation}
where $\hat h$ is the metric given by \eqref{eq:defhath}. The only thing we used to derive 
\eqref{eq:gradchar} is that $\gr_h\rho_i$ is in the $\rho_i$-eigenspace of $L$ for each eigenvalue $\rho_i$ of $L$ (and, of course,
that each $\rho_i$ is smooth with $\d\rho_i\neq 0$).
The polynomial expression \eqref{eq:gradchar} in $t$ resp. $1/t$ gives rise to equivalent equations on the coefficients. These equations are given by 
$\gr_h\mu_i=-\gr_{\hat h}\hat\mu_{\ell+1-i}$ (or, what is equivalent, $\gr_{\hat h} \hat\mu_i=-\gr_h\mu_{\ell+1-i}$) for all $i$.
Taking into account formula \eqref{eq:defhath} for $\hat h$ and the fact that $L$ is $h$-selfadjoint, one sees that the latter equations are just the gradient version 
of the formulas in Lemma~\ref{lem:aboutmu}.
\end{rem}

Thus, $\hat g$ defined by  \eqref{eq:hatg} is a K\"ahler metric on $(M,J)$. \end{proof}

Consider now the K\"ahler metrics $g,\hat{g}$ on $(M,J)$ with Levi-Civita connections
$\nabla,\hat{\nabla}$ respectively. Let $T$ 
be the $(1,2)$-tensor defined by
$$T(X,Y)=\hat{\nabla}_X Y-\nabla_X Y.$$
Since  $\nabla,\hat{\nabla}$ are both torsion-free, $T$ is symmetric in $X,Y$. 
Moreover, since both $\nabla,\hat{\nabla}$ preserve $J$, we have the symmetry
\begin{align}
T(X,JY)=T(JX,Y)=JT(X,Y).\label{eq:Tsymmetry}
\end{align}

\begin{lem}
\label{lem:hatgcproequivg}
The tensor $T$ satisfies
\begin{align}
T(X,Y)=\Phi(X)Y+\Phi(Y)X-\Phi(JX)JY-\Phi(JY)JX.\label{eq:cproequiv}
\end{align}
for a certain $1$-form $\Phi$ on $M$.
\end{lem}

\begin{proof}
First we recall that $A$ preserves the vertical distribution $\mathcal{V}$ and 
the horizontal distribution $\mathcal{Q}=\mathcal{V}^\perp$ so that $\mathcal{Q}$ 
does not change if we consider $\hat{g}$ instead of $g$. We will use the same symbol 
for a vector field on $Q$ and its horizontal lift to $M$. 

Denoting by $\mathrm{pr_{\mathcal{Q}}}:TM\rightarrow \mathcal{Q}$ the projection to 
the horizontal distribution, the O'Neill formula \eqref{eq:ONeill} implies
$$\mathrm{pr_{\mathcal{Q}}}(T(X,Y))=\hat{\nabla}^Q_X Y-\nabla^Q_X Y$$
for vector fields $X,Y\in \Gamma(TQ)$, where $\nabla^Q,\hat{\nabla}^Q$ are the Levi-Civita 
connections of the horizontal parts $g_Q$ and $\hat{g}_Q$ of $g$ and $\hat{g}$ respectively.

From formula \eqref{eq:hatg} we see that $\hat g_Q=c\cdot(\mathrm{det}\,A_Q)^{-1}g_Q(A_Q^{-1}\cdot,\cdot)$ 
for some constant $c$, where $A_Q$ is the quotient of $A$ and we used that $\mathrm{det}_\C A$ equals $\mathrm{det}\,A_Q$ 
up to multiplying with a constant. Comparing this formula for $\hat g_Q$ with \eqref{eq:defhath} and noting that, by construction, 
$g_Q$ and $A_Q$ are compatible on $Q$, we see that $\hat{g}_Q$ is projectively equivalent to $g_Q$. Thus, we have that
\begin{align}
\hat{\nabla}^Q_X Y-\nabla^Q_X Y=\Phi(X)Y+\Phi(Y)X\label{eq:proequiv}
\end{align}
is satisfied for all $X,Y\in \Gamma(Q)$ for a $1$-form $\Phi$ on $Q$. Indeed, 
the fact that \eqref{eq:proequiv} is equivalent to $g_Q,\hat{g}_Q$ being projectively equivalent
is a classical statement in projective geometry, see \cite{Levi-Civita}. 

Using \eqref{eq:proequiv}, we obtain that the horizontal part of $T$ is given by 
$$
\mathrm{pr_{\mathcal{Q}}}(T(X,Y))=\Phi(X)Y+\Phi(Y)X.
$$
However, since $\mathcal{V}$ is spanned by the Killing vector fields $K_i$, any 
$g$- or $\hat{g}$-geodesic $\gamma(t)$ in $M$ being initially tangent to $\mathcal{Q}$ 
remains tangent to it for all values of $t$. It follows that 
$\nabla_X X,\hat{\nabla}_X X\in \mathcal{Q}$ whenever $X\in \mathcal{Q}$. 
Then $T(X,X)$ in $\mathcal{Q}$ for all $X\in \mathcal{Q}$ and by polarisation 
(recall that $T$ is symmetric) we have $T(X,Y)\in \mathcal{Q}$ for all $X,Y\in \mathcal{Q}$. 
Thus, we obtain
\begin{align}
T(X,Y)=\Phi(X)Y+\Phi(Y)X.\label{eq:Thorizontal}
\end{align}
for all $X,Y\in \mathcal{Q}$. Since the form $\Phi$ in \eqref{eq:proequiv} is explicitly 
given by the formula
$$
\Phi=-\frac{1}{2}\d\,\ln (\mathrm{det}\,L),
$$
(which is a classical formula that can be obtained from  \eqref{eq:proequiv} by contraction),
 we see that it vanishes upon insertion of vector fields that are contained in the
 generalised eigenspaces of $A$ corresponding to constant eigenvalues. Thus, 
\eqref{eq:Thorizontal} establishes formula \eqref{eq:cproequiv} for vector fields tangent to $\mathcal Q$. 

It remains to verify equation \eqref{eq:cproequiv} upon insertion of vertical 
vector fields $JX,JY$, where $X,Y\in J\mathcal{V}\subseteq \mathcal{Q}$. 
Using \eqref{eq:Tsymmetry}, we obtain 
$$
T(JX,JY)=-T(X,Y)\overset{\eqref{eq:Thorizontal}}{=}-\Phi(X)Y-\Phi(Y)X
$$
$$
=\underbrace{\Phi(JX)JY+\Phi(JY)JX}_{=0}-\Phi(JJX)JJY-\Phi(JJY)JJX
$$
which establishes \eqref{eq:cproequiv} evaluated on $JX,JY$. Further, for 
arbitrary $Z$ tangent to $\mathcal Q$, we obtain 
$$
T(Z,JX)=JT(Z,X)\overset{\eqref{eq:Thorizontal}}{=}\Phi(Z)JX+\Phi(X)JZ
$$
$$
=\Phi(Z)JX+\underbrace{\Phi(JX)}_{=0}Z-\underbrace{\Phi(JZ)}_{=0}JJX-\Phi(JJX)JZ
$$
establishing \eqref{eq:cproequiv} when evaluated on $Z,JX$. Thus, 
we verified \eqref{eq:cproequiv} on all possible combinations of tangent 
vectors and the claim follows.
\end{proof}

It is a classical statement in c-projective geometry, see for example \cite{Mikes,Tashiro}, and we used this fact
already in the proof of Lemma~\ref{lem:killing} that two complex torsion-free connections 
$\nabla,\hat{\nabla}$ on a complex manifold $(M,J)$ are c-projectively equivalent (i.e. their $J$-planar curves coincide) if and 
only if \eqref{eq:cproequiv} is satisfied for a certain $1$-form $\Phi$. Lemma~\ref{lem:hatgcproequivg} 
then shows that $g,\hat{g}$ are c-projectively equivalent. 
This implies that $A=A(g,\hat{g})$ is a solution of equation \eqref{eq:main} and completes 
the proof of the second part of Proposition~\ref{prop:realising}.   \end{proof}

\subsection{Explicit formulas}
\label{explicitform}

In the preceding sections, we have proved Theorem~\ref{thm:invform}  (see Propositions~\ref{prop:locform} and~\ref{prop:realising}) 
which can be understood as an invariant version of
 Theorem~\ref{thm:localclassification}.
 We are now going to derive the formulas 
from Example~\ref{ex:main}. Our starting point is Proposition~\ref{prop:locform}. 
We will derive explicit formulas for all the objects that have been introduced there and  thereby prove Theorem~\ref{thm:localclassification}.

The compatible pair $h$, $L$ can be described explicitly 
by using the results from \cite{BM}. The latter article contains explicit formulas 
for a compatible pair in the general pseudo-Riemannian case.  
In our case, there are no Jordan blocks (with non-constant eigenvalues) and the formulas become similar to the classical Levi-Civita theorem -- 
the only modification being signs $\varepsilon_i=\pm 1$ for each non-constant real eigenvalue $\rho_i$ (which allow us to ``produce'' 
an arbitrary signature) and the occurrence of complex eigenvalues. Let 
$$
E_{\mathrm{nc}}=\{\rho_1,\overline{\rho}_1,\dots,\rho_r,\overline{\rho}_r,\rho_{r+1},\dots,\rho_{r+q}\}
$$
denote the set of (non-constant) eigenvalues of $L$ ($r$ pairs of complex-conjugate eigenvalues and $q$ real eigenvalues). 
Recall that the ``gluing data'' in \cite[Theorem 1.3]{BM} 
takes the form of a $1$-dimensional block
$$
h_i=\varepsilon_i \d x_i^2,\,\,\,L_i=\rho_i(x_i)\partial_{x_i}\otimes \d x_i,
$$
for a real eigenvalue $\rho_i$ and (as follows from \cite[Theorems 5]{BM}
 or \cite[Theorem 2]{Liouville}) the form of a $2$-dimensional block 
$$
h_i=\frac{1}{4}(\overline{\rho}_i(\overline{z}_i)-\rho_i(z_i))(\d z_i^2-\d \overline{z}_i^2),\,\,\,
L_i=\rho_i(z_i)\partial_{z_i}\otimes \d z_i+\overline{\rho}_i(\overline{z}_i)\partial_{\overline{z}_i}\otimes \d\overline{z}_i,
$$
if $\rho_i,\overline{\rho}_i$ is a pair of complex conjugate eigenvalues, where $z_j$ 
is a complex coordinate w.r.t. which $\rho_j$ is a \emph{holomorphic} function.

Thus, by \cite[Theorem 1.3]{BM}, we find local coordinates 
$z_1,\dots,z_r,x_{r+1},\dots,x_{r+q}$\footnote{These coordinates should note be confused with the general 
coordinate system $x_1,\dots,x_\ell$ used, for instance, in Proposition~\ref{prop:locform}.} 
(where the $z_j$ are complex coordinates and the $x_j$ are
 real coordinates) such that
\begin{align}
\begin{array}{c}\displaystyle
h=-\frac{1}{4}\sum_{i=1}^r \left(\Delta_i \d z_i^2+c.c.\right)+\sum_{i=r+1}^{r+q}\Delta_i \varepsilon_i \d x_i^2,\vspace{1mm}\\
\displaystyle L=\sum_{i=1}^r(\rho_i\partial_{z_i}\otimes \d z_i+c.c.)+\sum_{i=r+1}^{r+q}\rho_i\partial_{x_i}\otimes \d x_i,
\end{array}
\label{eq:normalformquotient}
\end{align}
where for $1\leq i\leq r$, $\rho_i(z_i)$ is a holomorphic function of $z_i$,
 for $r+1\leq i\leq r+q$, $\rho_i(x_i)$ only depends on $x_i$, ``$c.c.$'' 
denotes the complex conjugate of the preceding term and 
$$
\Delta_i=\prod_{\rho\in E_{\mathrm{nc}}\setminus\{ \rho_i\}}(\rho_i-\rho).
$$

The parts of $g,\omega$ and $A$ in \eqref{eq:localclassificationgJw} and \eqref{eq:localclassificationA}, which correspond
to the ``constant'' block, are obtained from the expressions
$$
\sum_{i=0}^\ell \mu_i \cdot (-1)^i g_{\mathrm{c}}(A_{\mathrm{c}}^{\ell -i} \cdot, \cdot),\,\,\,
\sum_{i=0}^\ell \mu_i \cdot (-1)^i  \omega_{\mathrm{c}}(A_{\mathrm{c}}^{\ell -i} \cdot, \cdot)\mbox{ and }
\sum_{p,q=1}^{2k}    (A_{\mathrm{c}})^q_p  \,  \d y_p  \otimes \left( \frac{\partial}{\partial {y_q}}    - \sum_{i=1}^\ell\alpha_{iq} \frac{\partial}{\partial {t_i}}   \right)
$$
in the formulas \eqref{eq:forH}, \eqref{eq:foromega} and \eqref{eq:forA} by decomposing $(g_\mathrm{c},\omega_\mathrm{c})$ in the sense of 
de~Rham--Wu
\cite{deRham,Wu} according to the parallel distributions belonging to the generalised eigenspaces of $A_{\mathrm{c}}$. 
Each component $(g_\gamma,\omega_\gamma,A_\gamma)$ of this decomposition, as we have already mentioned, can be described explicitly using the results of Boubel \cite{Boubel}.

To establish the formulas \eqref{eq:localclassificationgJw} and \eqref{eq:localclassificationA} for $g,\omega$ and $A$,
it remains to find the formulas for the parts of $g,\omega$ and $A$ in the direction tangent to the Killing vector fields. 
We will specialise the general coordinate system $x_1,\dots,x_\ell$ in formula \eqref{eq:forH} by choosing 
$\mu_1,\dots,\mu_\ell$ as coordinates, compare also the formulas \eqref{eq1}--\eqref{eq4} in Proposition~\ref{prop:locrigid}.
With this choice of coordinates, the matrix $H_{ij}$ (inverse to $G_{ij}$ defined by $h=G_{ij}\d\mu_i\d\mu_j$)
is just given by 
\begin{equation}
\label{eq:Hij}
H_{ij}=g(K_i,K_j)=g(\gr_g \mu_i,\gr_g \mu_j)
=h(\gr_{h} \mu_i,\gr_{h} \mu_j).
\end{equation}
Let $\mu_{i}(\hat{\rho})$ denote the $i$th elementary symmetric polynomial in the $\ell-1$ variables 
$E_{\mathrm{nc}}\setminus\{\rho\}$. Writing $\d\mu_i=\sum_{\rho\in E_{\mathrm{nc}}}\mu_{i-1}(\hat{\rho})\d\rho$ in 
the coordinates from formula \eqref{eq:normalformquotient} we have
\begin{align}
\displaystyle \d \mu_i=\sum_{s=1}^r\left(\mu_{i-1}(\hat{\rho_s})\frac{\partial \rho_s}{\partial z_s}\d z_s
+c.c.\right)+\sum_{s=r+1}^{r+q}\mu_{i-1}(\hat{\rho}_s)\frac{\partial \rho_s}{\partial x_s}\d x_s.\label{eq:dmui}
\end{align}
Using the formula \eqref{eq:normalformquotient} for $h$ together with \eqref{eq:dmui}, we obtain
\begin{align}
\label{eq:gradmui}
\gr_{h}\,\mu_j=-4\sum_{s=1}^r\left(\frac{\mu_{j-1}(\hat{\rho_s})}{\Delta_s}\frac{\partial \rho_s}{\partial z_s}\frac{\partial}{\partial z_s}
+c.c.\right)+\sum_{s=r+1}^{r+q}\varepsilon_s\frac{\mu_{j-1}(\hat{\rho}_s)}{\Delta_s}\frac{\partial \rho_s}{\partial x_s}\frac{\partial}{\partial x_s}.
\end{align}
From \eqref{eq:Hij}, \eqref{eq:dmui} and \eqref{eq:gradmui}, we obtain
\begin{align}
\begin{array}{c}
\displaystyle H_{ij}=-4\sum_{s=1}^r\left(\frac{\mu_{i-1}(\hat{\rho_s})\mu_{j-1}(\hat{\rho_s})}{\Delta_s}\left(\frac{\partial \rho_s}{\partial z_s}\right)^2
+c.c.\right)+\sum_{s=r+1}^{r+q}\varepsilon_s \frac{\mu_{i-1}(\hat{\rho}_s)\mu_{j-1}(\hat{\rho}_s)}{\Delta_s}\left(\frac{\partial \rho_s}{\partial x_s}\right)^2.
\end{array}\label{eq:formulaforHij}
\end{align}
Inserting the formulas \eqref{eq:normalformquotient} and \eqref{eq:formulaforHij} into \eqref{eq:forH}, 
we obtain the formula \eqref{eq:localclassificationgJw} for $g$.
It remains to find the formula for $A$ and the formulas 
for $J$ acting on $\theta_i$ and $\d\rho_j$ respectively.

Taking the differential of the identity
$$
\prod_{\rho\in E_{\mathrm{nc}}}(t-\rho)=\sum_{s=0}^\ell (-1)^s \mu_s t^{\ell-s}
$$
and inserting $t=\rho_i$, we obtain 
$$
\d\rho_i=\frac{1}{\Delta_i}\sum_{s=1}^\ell (-1)^{s-1}  \rho_i^{\ell-s}\d\mu_s.
$$
By \eqref{eq4}, we have
$$
\d\rho_i\circ J=-\frac{1}{\Delta_i}\sum_{s,t=1}^\ell (-1)^{s-1}  \rho_i^{\ell-s}H_{st}\theta_t.
$$
Inserting \eqref{eq:formulaforHij} into this and applying standard Vandermonde identities (see the Appendix of \cite{ApostolovI}), we obtain
$$
\d z_i\circ J=4\frac{1}{\Delta_i}\frac{\partial \rho_i}{\partial z_i}\sum_{j=1}^\ell\mu_{j-1}(\hat{\rho_i})\theta_j\mbox{ for }1\leq i\leq r
$$
and 
$$
\d x_i\circ J=-\frac{\varepsilon_i}{\Delta_i}\frac{\partial \rho_i}{\partial x_i}\sum_{j=1}^\ell\mu_{j-1}(\hat{\rho}_i)\theta_j\mbox{ for }r+1\leq i\leq r+q.
$$
Inverting these formulas by using Vandermonde identities shows the formula for $\theta_i\circ J$ 
expressed as a linear combination of the $\d z_i$ and $\d x_i$. This establishes the formula for $J$ in \eqref{eq:localclassificationgJw}. 

Let us derive the remaining part of the formula \eqref{eq:localclassificationA} for $A$. The formula for $\omega$ 
in Proposition~\ref{prop:locform} shows that the Killing vector fields $\frac{\D}{\D t_i}$ coincide with $K_i=J\gr\,\mu_i$. We show
\begin{equation}
\label{eq:AonK}
AK_i=\mu_i K_1-K_{i+1}\mbox{ for all }i=1,\dots,\ell
\end{equation}
(where we put $K_{\ell+1}=0$) which immediately implies \eqref{eq:localclassificationA}. 
Hence, as soon as \eqref{eq:AonK} is derived, all formulas from Example~\ref{ex:main} are established and 
Theorem~\ref{thm:localclassification} is proven.

Formula \eqref{eq:AonK} is in fact the reformulation to our setting of \cite[formula (58)]{ApostolovI} and it can be proven 
in the same way: let $v_t=\gr_h\chi_{L}(t)$, where as usual 
$\chi_{L}(t)=\mathrm{det}(t\cdot\Id-L)=\sum_{i=0}^\ell(-1)^i\mu_i t^{\ell-i}$ denotes the characteristic polynomial 
of $L$ (in the terminology of Proposition~\ref{prop:locform}). For abbreviation define $v_i=\gr_h\mu_i$ such that $K_i=Jv_i$ and $2\Lambda=v_1$. 
Recall that, using the compatibility of $L$ and the metric $h$, we have derived the identity
\begin{equation}
\label{eq:LonK}
(t\cdot \Id-L) v_t=-\mathrm{det}(t\cdot \Id-L)v_1=-\sum_{i=0}^\ell(-1)^i\mu_it^{\ell-i}v_1
\end{equation}
in the proof of Lemma~\ref{lem:anotherone}. Inserting $v_t=\sum_{i=1}^\ell(-1)^i t^{\ell-i}v_i$ into the 
left-hand side of \eqref{eq:LonK} and setting $v_{\ell+1}=0$, we obtain 
$0=\sum_{i=1}^{\ell}(-1)^{i}( v_{i+1}+Lv_i-\mu_iv_1)t^{\ell-i}$, hence, $Lv_i=\mu_iv_1-v_{i+1}$. 
Of course, this holds also with $L$ replaced by $A$ and \eqref{eq:AonK} follows after 
multiplying with $J$ and using that $A$ commutes with $J$.


\section{Proof of the Yano-Obata conjecture (Theorem~\ref{thm:yano_obata}) }
\label{sec:yano_obata}

The main goal of this section is to prove Theorem~\ref{thm:yano_obata}.  
Simultaneously, we will give a proof of an important special case of  the projective Lichnerowicz  conjecture (Theorem~\ref{thm:lichnerowicz}), see Theorem~\ref{specialcase2}.  

Recall that the existence of a projective vector field $v$ for a (pseudo)-Riemannian metric $g$  implies the existence of an endomorphism $A$ compatible with $g$.   According to Proposition~\ref{prop:split1}, in a neighbourhood of a regular point, $A$ naturally splits into two blocks $A_{\mathrm{c}}$ and $L$  with constant and non-constant eigenvalues respectively.

If the non-constant block  $L$ (see Proposition~\ref{prop:split1}) is diagonalisable\footnote{Recall that in the c-projective setting this condition is fulfilled automatically  (see Lemma~\ref{lem:properties}~\eqref{l1}), whereas in the projective setting non-constant Jordan blocks are allowed.}, i.e. contains no Jordan blocks, then the proofs of Theorems~\ref{thm:yano_obata} and~\ref{thm:lichnerowicz} are almost identical. For that reason, 
  parallel to the proof of the Yano-Obata conjecture, we will prove the following version of the (pseudo-Riemannian) projective Lichnerowicz  conjecture: 
 
 \begin{thm} 
 \label{specialcase2}
 Let $M$ be a closed connected manifold   
 with an indefinite metric  $g$ on it.  Assume that $(M,g)$ admits a projective vector field $v$ and let $A$ be an endomorphism compatible with $g$ in the projective sense, $A\ne c\cdot\mathrm{Id}$. If there exists a regular point $p\in M^0$ at which the non-constant block of $A$  is diagonalisable, then $v$ is affine.
\end{thm}

In other words, this theorem says that in the absence of Jordan blocks with non-constant eigenvalues,  the (pseudo-Riemannian) projective Lichnerowicz  conjecture holds true, i.e.,  non-affine projective vector fields do not exist on compact manifolds with indefinite metrics. 
To complete the proof of this conjecture in full generality, it remains to show that Jordan blocks with non-constant eigenvalues are also ``forbidden''.
For Lorentzian metrics, this will be done in Section~\ref{sec:Jblock} which is the final step of the proof of Theorem~\ref{thm:lichnerowicz}.

The proof  of Theorem~\ref{specialcase2} is organised   as a series of remarks:  at any step  of the proof of 
 Theorem~\ref{thm:yano_obata}  we put a remark explaining how to change, if necessary, the proof in order to obtain a 
proof of an analogous step of Theorem~\ref{specialcase2}.  In particular,  we use similar notations for the projective and c-projective cases.

\subsection{Conventions and degree of mobility} \label{conv}   
Within the whole \S\ref{sec:yano_obata}  (except Remarks~\ref{rem:first}--\ref{rem:last} for Theorem~\ref{specialcase2} 
where we use similar notation in the projective setting)
we assume that  $(M,g,J)$ is  a closed connected K\"ahler manifold of (a priori) arbitrary signature and of dimension $2n\ge 4$   
(though the case $2n=4$ has been settled in \cite{BMMR}),
 and that  $v$ is a c-projective vector field which is not an affine vector field. We denote by 
 $\Phi^v_t$ its flow, by definition the  pullback  of $g$ w.r.t. $\Phi^t_v$ is a metric which is 
 c-projectively equivalent to $g$.

 We define the  \emph{degree of (c-projective) mobility} $D(g,J)$ of $(g, J)$  as 
  the dimension of the space of hermitian solutions of equation \eqref{eq:main}. 
  If $D(g,J)= 1$, the flow of  $v$ acts by  homotheties, since  otherwise $\Phi_t^* g$ is non-proportional to $g$ and
 hence $D(g,J)\geq 2$. Thus, this case is impossible since we assumed that $v$ is not an affine vector field. 
  
  Now, in the case  $D(g,J)\geq 3$, Theorem~\ref{thm:yano_obata} follows from \cite[Theorem 1]{FKMR}, where it is proved that   a closed connected K\"ahler manifold of arbitrary signature 
   with $D(g,J)\geq 3$ is either 
$(\C P(n),c\cdot g_{FS},J_{\mathrm{standard}})$ for some $c\in \mathbb{R}\setminus \{0\}$
 or any metric which is  c-projectively equivalent to $g$  is affinely  equivalent to $g$. Thus, if $D(g,J)\geq 3$, we are done. 
 In the remaining part of this section, we shall treat the case $D(g,J)=2$.

Here is our first remark related to the proof of Theorem~\ref{specialcase2}.

 \begin{remthm}[for Theorem \ref{specialcase2}]
 \label{rem:first}
 For a (pseudo)-Riemannian metric $g$, the \emph{degree of (projective) mobility}  $D(g)$ is the dimension of the vector space of $g$-selfadjoint solutions of \eqref{eq:main:proj}. If $D(g)= 1$, every projective transformation is a homothety and is an affine transformation. If $D(g)\ge 3$ and $g$ has indefinite signature, then by \cite[Corollary 5.2]{Mounoud} (which plays here the role analogous to  the role  of \cite{FKMR} in the K\"ahler setting)  each projective transformation is an affine transformation.  Therefore, in the rest of the proof of Theorem~\ref{thm:yano_obata} we may (and will) assume that $D(g)=2$.
 \end{remthm}

\subsection{Scheme of the proof}

Let us outline the steps of the proof of Theorem~\ref{thm:yano_obata} under the assumption $D(g,J)=2$. 
In \S\ref{ssec:PDE}, we will derive the PDE's that describe the evolution of the c-compatible pair $g$ and $A$ along a projective vector field $v$.

In \S\ref{sssec:nocomplex} we show that the equation for $A$ can be reduced to one of three a priori possible canonical forms.  
By using these forms we show that $A$ cannot have non-constant complex eigenvalues and moreover,  the real eigenvalues are bounded 
only for one particular canonical form, namely $\mathcal L_v A = A(\mathrm{Id}-A)$. This equation will automatically imply that the 
constant eigenvalues of $A$, if they exist, are $0$ and $1$.  

This simplifies the formulas in Theorem~\ref{thm:localclassification} considerably. 
In the local classification, $g$ and $A$ are in block-diagonal form. In \S\ref{sssec:explicit}, we will obtain a partial 
solution to the PDE system by only considering the part that corresponds to the block spanned
 by the gradients of the non-constant eigenvalues $\rho_1,\dots,\rho_\ell$ of $A$  (that is the $L$-block from Theorem~\ref{thm:invform} 
and Example~\ref{ex:now2}). The PDE system
 restricted to this block reduces to ordinary differential equations on a certain set of functions 
$F_1(\rho_1),\dots,F_\ell(\rho_\ell)$ and on the eigenvalues $\rho_i$. Thus, we obtain quite explicit formulas for $A$, $g$ and $v$ 
involving some yet unknown constants $a_1,\dots, a_\ell$ and $\CC$ as parameters.

In \S\ref{ssec:l23}, we will use these formulas  to analyse the asymptotic properties
of the scalar products $g(K_i,K_j)$ of the Killing vector fields and 
the eigenvalues of the curvature operator.  We will conclude from this analysis that there cannot 
be more than one non-constant eigenvalue $\rho$ (otherwise the eigenvalues of the curvature
 operator are unbounded for $t\to\pm \infty$ (which is impossible on a closed manifold) or $g(K_i,K_j)$
 does not tend to zero for $t\to\pm\infty$ (as it should)).

Now we are left with a very specific form of the formulas in Theorem~\ref{thm:localclassification}: 
there is only one non-constant eigenvalue $\rho$ and at most two constant eigenvalues $0$ and $1$. 
In \S\ref{sssec:l11} we complete the proof of Theorem~\ref{thm:yano_obata}. 
We do this by deriving further restrictions on the constant $\CC$ from above that appears 
as a parameter in the metric. Using results of \cite{FKMR} we are then able to conclude that the metric, up to a sign, is positive
definite. This traces back the proof of Theorem~\ref{thm:yano_obata} to the corresponding result \cite{YanoObata} for positive signature.

\subsection{C-projectively  invariant form of equation \eqref{eq:main} and special features of $D(g,J)=2$}
\label{ssec:PDE}

Consider the canonical line bundle $\mathcal{E}=\Lambda^{2n}T^* M$ over the K\"ahler manifold $(M,g,J)$. 
For a real number $w$, we define the line bundle $\mathcal{E}(w)$ whose transition functions are given
 by the transition functions of $\mathcal{E}$ to the power $w$. Note that $M$, as a complex manifold, 
has a canonical orientation and we can assume positivity of the transition functions of $\mathcal{E}$. 
For an arbitrary tensor bundle $E$ over $M$, we can then define the ``weighted version'' 
$E(w)=E\otimes \mathcal{E}(w)$. Let $S^2_J TM$ denote the bundle of hermitian contravariant $2$-tensors 
and denote by $\nabla$ the Levi-Civita connection of $g$. In the Appendix of \cite{YanoObata} it was 
shown that the PDE
\begin{align}
\nabla_X \sigma=X\odot\Lambda_\sigma+JX\odot J\Lambda_\sigma,\,\,\,X\in TM \label{eq:maininvariant}
\end{align}
on sections $\sigma$ of $S^2_J TM(\frac{1}{n+1})$ ($2n=\mathrm{dim}_{\R}\,M$), where $\Lambda_\sigma\in \Gamma(TM(\frac{1}{n+1}))$
 is the $\nabla$-divergence of $\sigma$ divided by $2n$ and $X\odot Y=X\otimes Y+Y\otimes X$ denotes the symmetric product, 
is c-projectively invariant, that is, it does not depend on the choice of connection $\nabla$ in the class 
$[\nabla]$ of c-projectively equivalent connections.
 Let us denote by $\mathcal{A}([g],J)$ the space of solutions of \eqref{eq:maininvariant}. There is an 
isomorphism between the space $\mathcal{A}(g,J)$ of hermitian solutions to \eqref{eq:main} (i.e., the space of hermitian
endomorphism c-compatible with $(g,J)$) and $\mathcal{A}([g],J)$ given by
$$
\varphi:\mathcal{A}([g],J)\longmapsto\mathcal{A}(g,J),\,\,\,\varphi(\sigma)=\sigma\sigma_g^{-1},
$$
where
$$
\sigma_g=g^{-1}\otimes \mathrm{vol}_g^{\frac{1}{n+1}}
$$
and $\mathrm{vol}_g$ is the volume form of $g$. As described in more detail in the appendix of 
\cite{YanoObata}, taking the Lie derivative $\mathcal{L}_v \sigma$ of a solution $\sigma$ to 
\eqref{eq:maininvariant} w.r.t. the c-projective vector field $v$ yields again a solution to 
\eqref{eq:maininvariant}. Thus, we obtain a linear mapping 
$$
\mathcal{L}_v:\mathcal{A}([g],J)\longmapsto\mathcal{A}([g],J).
$$
Under the assumption $D(g,J)=2$, we can chose a basis $\sigma,\hat \sigma$ of $\mathcal{A}([g],J)$ 
and find the equations
\begin{align}
\begin{array}{c}
\mathcal{L}_v\sigma=\alpha\sigma+\beta\hat\sigma,\vspace{1mm}\\
\mathcal{L}_v\hat\sigma=\gamma\sigma+\delta\hat\sigma
\end{array}\label{eq:PDEgeneral}
\end{align} 
for certain real numbers $\alpha,\beta,\gamma,\delta$.

Using \eqref{eq:PDEgeneral} we can easily derive the Lie derivatives of $A\in \mathcal{A}(g,J)$  
and $g$ along $v$.

\begin{prop}
\label{prop:PDE_gen}
Let $v$ be a $c$-projective vector field for $g$ and $A\in \mathcal{A}(g,J)$, $A\ne c \cdot \Id$.
 Then we have the equations
\begin{align}
\mathcal{L}_v A=-\beta A^2+(\delta -\alpha) A+\gamma\Id.\label{eq:PDEA_general}
\end{align}
and 
\begin{align}
\mathcal{L}_v g=\left(-\frac{\beta}{2} \mathrm{tr}\,A-(n+1)\alpha\right) g-\beta gA,
\label{eq:PDEg_general}
\end{align}
for some constants $\alpha, \beta, \gamma,\delta\in \R$.
Moreover, the restriction $\rho(t)=\rho(\Phi_t^v(p))$ of every eigenvalue $\rho$ of $A$ to an 
integral curve of $v$ satisfies the ODE
\begin{align}
\label{eq:ODEeigenvalues_gen}
\dot\rho=-\beta \rho^2+(\delta -\alpha) \rho+\gamma.
\end{align}
\end{prop}

\begin{proof}
 Take $\sigma=\sigma_g$ and $\hat \sigma=\phi^{-1} (A )= A\sigma_g$ as a basis  of $\mathcal{A}([g],J)$.  
 Then by using   \eqref{eq:PDEgeneral},   we get
$$
\mathcal{L}_v A=\mathcal{L}_v(\hat\sigma\sigma_g^{-1})=(\gamma\sigma_g+\delta\hat\sigma)\sigma_g^{-1}-
\hat\sigma \sigma_g^{-1}(\alpha\sigma_g+\beta\hat\sigma)\sigma_g^{-1}
$$
$$
=\gamma\Id+\delta A-\alpha A-\beta A^2.
$$
This yields \eqref{eq:PDEA_general}. Moreover, a straightforward calculation yields
$$
\mathcal{L}_v g=\mathcal{L}_v\left(\left(\mathrm{det\,g}\right)^{\frac{1}{2(n+1)}}\sigma_g^{-1}\right)=
\frac{1}{2(n+1)}\mathrm{tr}(g^{-1}\mathcal{L}_v g)g-g(\alpha\sigma_g+\beta\sigma)\sigma_g^{-1}
$$
$$
=\left(\frac{1}{2(n+1)}\mathrm{tr}(g^{-1}\mathcal{L}_v g)-\alpha\right) g-\beta gA.
$$
Taking the trace of this equation gives us
$$
\frac{1}{2(n+1)}\mathrm{tr}(g^{-1}\mathcal{L}_v g)=-\alpha n-\frac{\beta}{2} \mathrm{tr}\,A
$$
and inserting this back into the equation for the Lie derivative of $g$, we obtain \eqref{eq:PDEg_general}.

The remaining formula \eqref{eq:ODEeigenvalues_gen} immediately follows from \eqref{eq:PDEA_general}. \end{proof}

\begin{remthm}[for Theorem \ref{specialcase2}]
The projective (and projectively invariant) 
 analogue of \eqref{eq:maininvariant}  was obtained in \cite{EM}.  Arguing as above we obtain 
the following version of Proposition~\ref{prop:PDE_gen}:
 
 {\it Let $v$ be a projective vector field for $g$ and let $A$ be compatible to $g$ in the projective sense, $A\ne c \cdot \Id$.  
Then we have \eqref{eq:PDEA_general}  and \eqref{eq:ODEeigenvalues_gen}  and the Lie derivative of $g$ satisfies} 
 $$
 \mathcal{L}_v g=\bigl(-\beta \, \mathrm{tr}\,A-(n+1)\alpha\bigr) g-\beta gA, \quad n=\dim M,
 $$
 i.e., in the first term on the right-hand side of \eqref{eq:PDEg_general} we simply need to replace $\frac{\beta}{2}$ by $\beta$.
\end{remthm} 


\subsection{Properties of the eigenvalues of $A$}
\label{sssec:nocomplex}

The equation \eqref{eq:ODEeigenvalues_gen} allows us to make several important conclusions about 
the eigenvalues of $A\in \mathcal{A}(g,J)$.

First of all we notice that the coefficient $\beta$ in equation \eqref{eq:ODEeigenvalues_gen} 
does not vanish.  Indeed, otherwise this equation takes 
the form $\dot \rho = (\delta - \alpha) \rho + \gamma$ and its non-constant solutions are unbounded  
which is impossible due to compactness of $M$.  Hence the eigenvalues of $A$ are all 
constant which contradicts the assumption that $v$  is not affine. Thus, $\beta\ne 0$.

Next,  we  see that the constant eigenvalues of $A$ are solutions to the quadratic equation 
$0=-\beta x^2+(\delta -\alpha) x+\gamma$.  Hence, there are at most two constant eigenvalues.  

To simplify the further discussion,  
without loss of generality we may assume that  the evolution of a non-constant eigenvalue 
$\rho$ of $A$ along $v$ is given by one of the ODEs
\begin{align}
\dot \rho=\rho^2+1,\,\,\,\dot \rho=\rho(1-\rho)\mbox{ or }\dot \rho=\rho^2. \label{eq:complexODEs}
\end{align}
Indeed,  the equations \eqref{eq:ODEeigenvalues_gen} can be reduced to one of these canonical
 forms by rescaling  $v$ and replacing $A$ by $c_1 A+c_2\Id$ for an appropriate choice of 
constants $c_1\neq 0,c_2\in\R$.

We now show that the eigenvalues of $A$  cannot be complex.

\begin{prop}
\label{thm:nocomplex}
Let $(M,g,J)$ be a closed connected K\"ahler manifold of degree of mobility $D(g,J)=2$ 
and let $v$ be a $c$-projective vector field which is not affine. 
Then all non-constant eigenvalues of $A\in \mathcal{A}(g,J)$ are real-valued.
\end{prop}

\begin{proof}
If we allow $\rho$ to be complex-valued,  each of the above ODEs  \eqref{eq:complexODEs} should
 be considered as a system of two ODEs on the real and imaginary part of $\rho$. The phase 
portraits corresponding to these systems  are shown in Figure~\ref{fig:phase}. It can be seen 
from the pictures or shown directly using the ODEs \eqref{eq:complexODEs}, that a non-constant 
solution to one of these equations with imaginary part not identically zero is given by a circle 
in the complex plane. 

\begin{figure}[ht]
  \includegraphics[width=.30\textwidth]{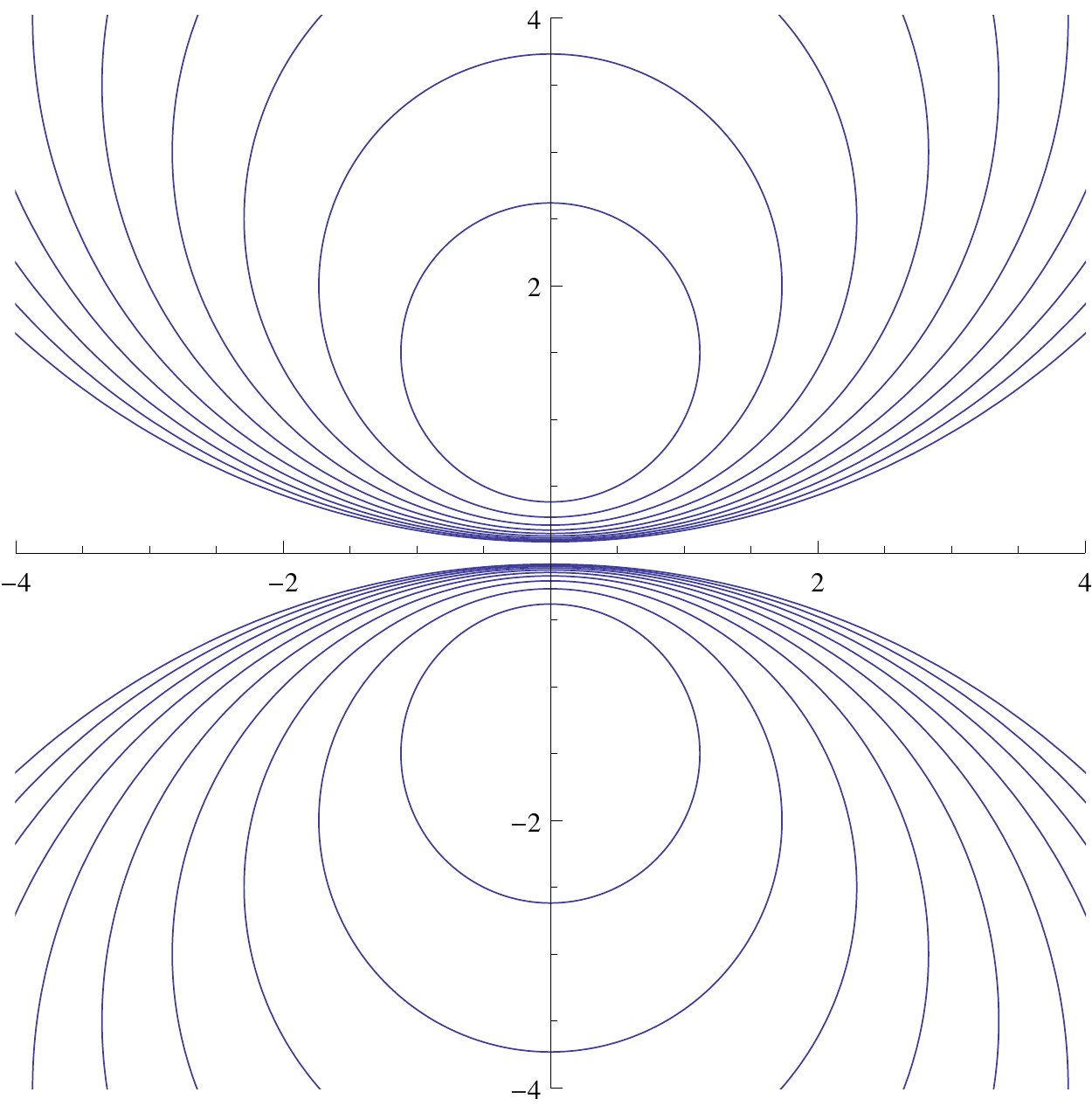}\hspace{4mm}
	\includegraphics[width=.30\textwidth]{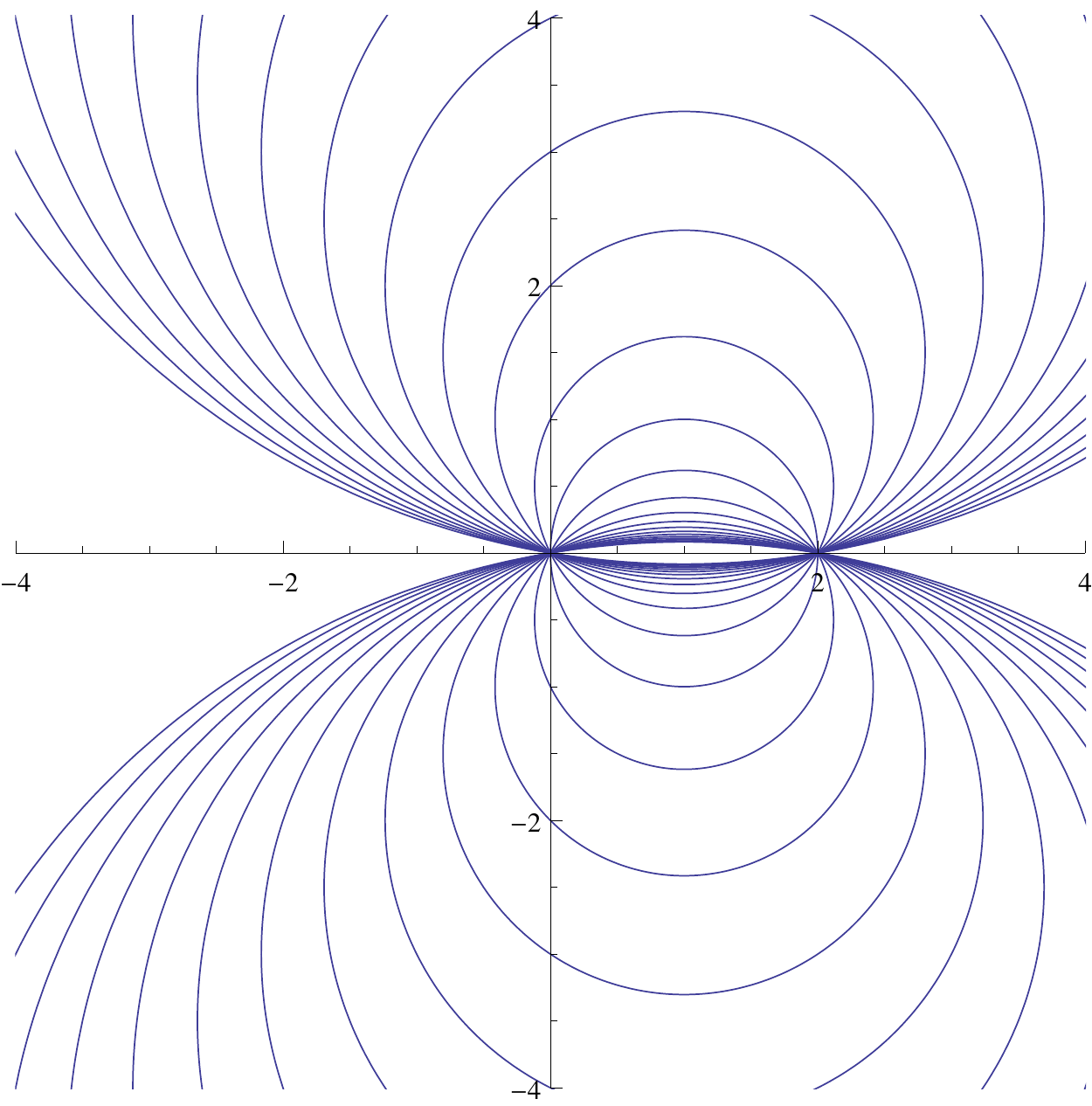}\hspace{4mm}
	\includegraphics[width=.30\textwidth]{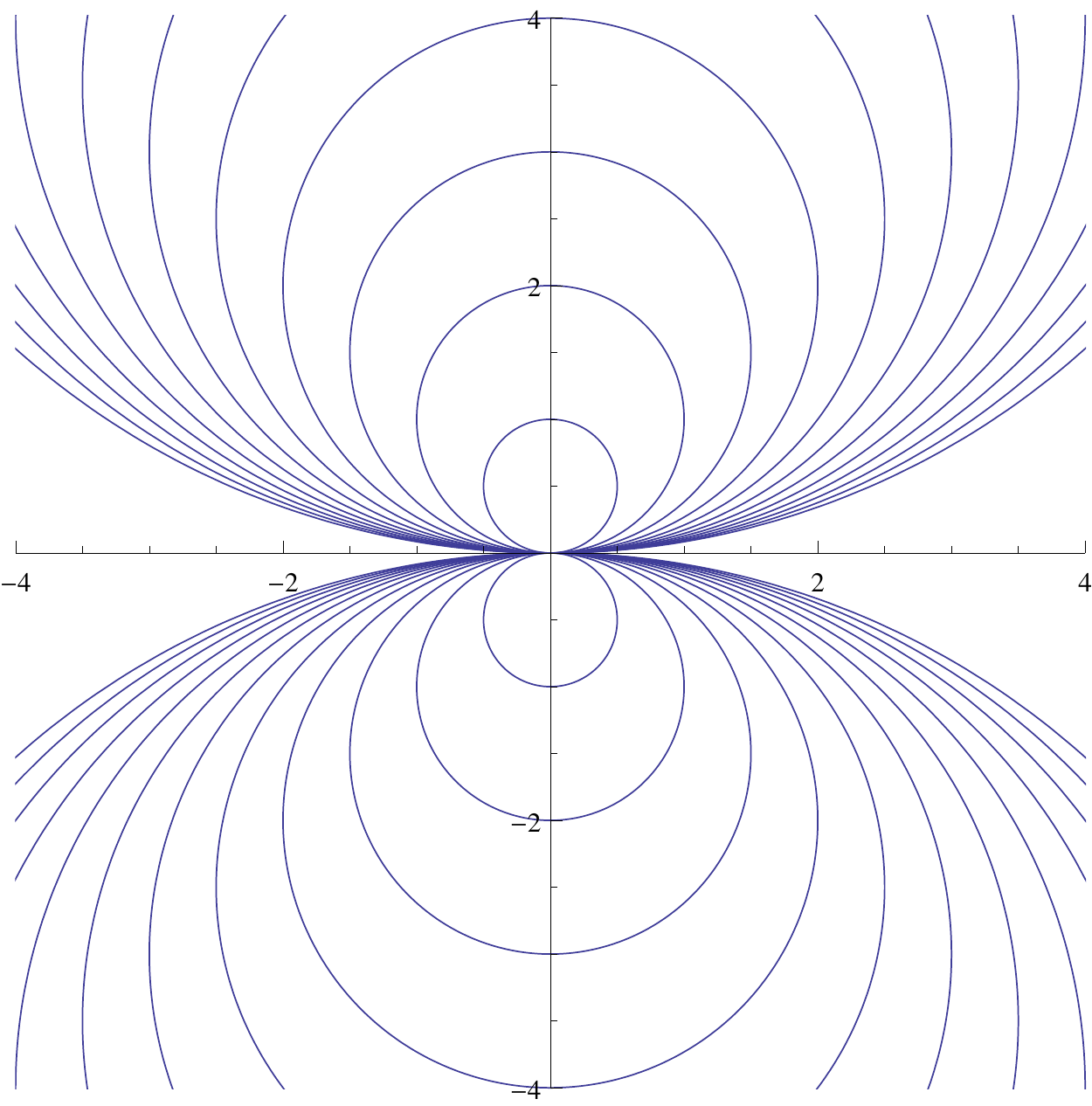}
  \caption{The phase portraits for the ODEs $\dot \rho=\rho^2+1$, $\dot \rho=\rho(1-\rho)$, $\dot \rho=\rho^2$ 
	(from left to right).}\label{fig:phase}
\end{figure}

As we see from Figure~\ref{fig:phase}, the maximal value of an imaginary part of a complex eigenvalue,
 taken over all points of the manifold and all eigenvalues, is not equal to the imaginary part of a
 constant eigenvalue (for each case in \eqref{eq:complexODEs}, the constant eigenvalues are $\pm i$,  
 $0$ and $1$ or $0$ resp.). In particular, the derivative of this eigenvalue in the direction of the 
c-projective vector field is well-defined and non-zero.  More precisely, the derivative of the imaginary 
part of $\rho$ is zero whereas the derivative of the real part is not.  We now show that for 
complex eigenvalues such a situation is impossible.

Basically this fact follows from our local description of c-projectively equivalent metrics 
(Theorem~\ref{thm:localclassification})  which states that the complex eigenvalues are holomorphic 
in an appropriate local coordinate system.  However, such a coordinate system exists only at 
generic points, so we need to modify this  idea and in particular to take into account the fact that
the eigenvalues cannot be considered as smooth functions at ``collision'' points where the 
multiplicities of the eigenvalues change.

Assume that at a point $p\in M$,  a complex eigenvalue $\rho$ of $A$, $\mathrm{Im}\,\rho(p)\ne 0$,  has multiplicity $k$.  
It follows from standard facts that for small 
neighbourhoods $U(p)\subseteq M$ of $p$ and $U(\rho(p))\subseteq \C$ of $\rho(p)$, $A$ 
has precisely $k$ eigenvalues $\rho_1,\dots,\rho_k$ at each point of $U(p)$ contained 
in $U(\rho(p))$ and that the elementary symmetric functions in the variables $\rho_1,\dots,\rho_k$
 are smooth complex-valued functions on $U(p)$. In particular, the function $\rho_1+\dots+\rho_k$ 
is smooth in 
$U(p)$.

\begin{lem}
\label{lem:complex}  
Let the differential of the imaginary part of $\rho_1+\dots+\rho_k$ vanishes at $p$.  
Then    $$\d(\rho_1+\dots+\rho_k) = 0.$$
\end{lem}

\begin{proof}
By contradiction,  assume that $\d(\rho_1+\dots+\rho_k)\ne 0$.
Consider the smooth endomorphism 
$$
\tilde A = (A-\rho_1\cdot\Id)\dots (A-\rho_k\cdot\Id):T_\C^* U(p)\rightarrow T_\C^* U(p),
$$
where $T_\C^* U(p)$ denotes the complexified cotangent bundle of $U(p)$. The kernel of 
$\tilde A$ defines a smooth complex $k$-dimensional distribution $D$ in $T_\C^* U(p)$ whose
 value $D(p)$ at the point $p$ coincides with the kernel of $(A-\rho(p)\cdot\Id)^k$. The subspace
 $D(p)$ is therefore the generalised $\rho(p)$-eigenspace of $A$. 

The differential of the smooth function $\rho_1+\dots+\rho_k$ at a generic point of $U(p)$ 
is $\d \rho_1+\dots+\d\rho_k$ and since $\d\rho_i \circ A=\rho_i\d\rho_i$  (because, by Lemma~\ref{lem:properties}, 
$\gr\,\rho$ is an eigenvector with eigenvalue $\rho$ of the 
$g$-selfadjoint endomorphism $A$), we obtain that $\d(\rho_1+\dots+\rho_k)$ is contained in $D$ 
at every point of $U(p)$. In particular we have that at $p$, $\d(\rho_1+\dots+\rho_k)$ is 
contained in $D(p)$, i.e., it is a generalised eigenvector of $A$ corresponding to $\rho(p)$. 
 On the other hand, since $\d \bigl( \mathrm{Im} (\rho_1+\dots + \rho_k)\bigr) (p)=0$,   
$\d(\rho_1+\dots+\rho_k)$ is a real 1-form at $p$.  
This contradicts to the following simple fact from Linear Algebra:   if $\rho$ is a complex 
eigenvalue of a real linear endomorphism $A$, then the generalised $\rho(p)$-eigenspace of $A$ 
contains no real vectors.\end{proof}

Now take a point $p\in M$ such that the imaginary part of an eigenvalue $\rho$ takes its maximum 
in the sense that it is greater or equal than the imaginary parts of all eigenvalues at all 
points of $M$ and apply Lemma~\ref{lem:complex}.  The imaginary part of $\rho_1+\dots+\rho_k$
  obviously satisfies $\d \left(\mathrm{Im} \sum \rho_i(p)\right)=0$.   On the other hand, the 
	Lie derivative of $\sum\rho_i$   along $v$  at the point $p$ is  $k\cdot  \dot \rho\ne 0$, 
	where $\dot \rho$ is defined by one of the equations \eqref{eq:complexODEs}. Thus, $\d (\rho_1+\dots+\rho_k) \ne 0$ 
	in contradiction with Lemma~\ref{lem:complex}.  This completes the proof. \end{proof}

In fact, only the second equation of \eqref{eq:complexODEs} is allowed.  Indeed, since $v$ is not affine,  
$A$ must have at least one  non-constant eigenvalue and, moreover, this eigenvalue has to be 
real according to Proposition~\ref{thm:nocomplex}.   
But it is easy to see that the equations $\dot \rho=\rho^2+1$ and $\dot \rho=\rho^2$  
have no bounded  real-valued non-constant solutions  whereas $\rho(t)$ must be bounded due to compactness of $M$.   

Thus, we are left with the case when the eigenvalues of $A$ satisfy
\begin{align}
\label{eq:ODEeigenvalues1}
v(\rho)=\rho(1-\rho).
\end{align}
and we may summarize the above discussion in the following

\begin{prop}
\label{thm:Bols}
Let $(M,g,J)$ be a closed connected K\"ahler manifold of real dimension $2n\geq 4$ such that $D(g,J)=2$ 
and let $v$ be a c-projective vector field that is not affine.  Then, after an appropriate rescaling of $v$, we can find $A\in \mathcal{A}(g,J)$ such that 
 
\begin{enumerate}

\item the eigenvalues of $A$ satisfy \eqref{eq:ODEeigenvalues1};

\item the eigenvalues of $A$ are all real;

\item $A$ has at most two constant eigenvalues $0$ and $1$.

\end{enumerate}
\end{prop}


\begin{remthm}[for Theorem \ref{specialcase2}]
The reduction of \eqref{eq:ODEeigenvalues_gen} to one of the canonical forms \eqref{eq:complexODEs} 
is a simple general fact from the theory of ODE's.  The statement of Proposition~\ref{thm:nocomplex} 
was proved in \cite[Theorem 1.11]{BM}  even under more general assumptions:  on a compact manifold $M$,  
each non-real eigenvalue of an endomorphism $A$ compatible with $g$ is necessarily constant.  

The first and third canonical forms from \eqref{eq:complexODEs} are impossible for the same reason:  
in these two cases non-constant real eigenvalues are not bounded (along $v$).  Thus, the eigenvalues of $A$ 
satisfy \eqref{eq:ODEeigenvalues1} and the statement of Proposition~\ref{thm:Bols}  remains true without any changes.
\end{remthm} 


\subsection{Explicit formulas for the non-constant block} 
\label{sssec:explicit}

In what follows, we will work with  $A\in \mathcal{A}(g,J)$ and a c-projective vector field $v$ as in Proposition~\ref{thm:Bols}.  
In particular, we assume that the non-constant eigenvalues $\rho_1,\dots.,\rho_\ell$ of $A$ are 
real-valued and the constant eigenvalues, $0$ and $1$, have multiplicities $m_0\ge 0$ and $m_1\ge 0$ respectively.

Under these assumptions,  the PDE systems \eqref{eq:PDEA_general}  and  \eqref{eq:PDEg_general}   take the form
\begin{align}
\begin{array}{c}
\mathcal{L}_v A=A(\Id- A),\vspace{1mm}\\
\mathcal{L}_v g=-gA-\left(\sum_{i=1}^\ell \rho_i + \CC  \right)g.
\end{array}\label{eq:PDE}
\end{align}
where $\CC$ is some constant (yet unknown and playing the role of an additional parameter).

We are going to evaluate the PDE system \eqref{eq:PDE} component-wise.
To make this more transparent, we notice that on the set of regular points $M^0$ there is a 
natural structure of two mutually orthogonal foliations.  Recall that $\mathcal F$ denotes the integrable and totally geodesic
distribution spanned by the commuting vector fields given by the Killing vector fields $K_1,\dots,K_\ell$ and the 
vector fields $JK_1,\dots,JK_\ell$.
Now let $\mathcal U$ be the distribution generated  by $\gr\, \rho_1,\dots,\gr\, \rho_\ell$ 
(or, equivalently, by $JK_1,\dots,JK_\ell$) so that $\mathcal F=\mathcal U\oplus \mathcal V$, where
$\mathcal V$ is the distribution generated by $K_1,\dots, K_\ell$ and defined in the first part of the article.  
The leaves of the other distribution $\mathcal U^{\bot}$,  orthogonal to $\mathcal U$,   are just common level surfaces of 
the eigenvalues $\rho_1,\dots, \rho_\ell$.  Note that both distributions are integrable: for $\mathcal U^{\bot}$ this 
holds by definition and for $\mathcal U$ it follows from the fact that  $\mathcal U$ is generated by the commuting vector fields $JK_1,\dots,JK_\ell$.

The next  statement summarizes some general properties of $\mathcal U$  
which hold true for any metric $g$  from Theorem~\ref{thm:localclassification}.
\begin{prop}
\label{prop:distributions}
Let $M^0\subset M$ be the set of regular points  (in particular, the non-constant eigenvalues 
$\rho_1,\dots, \rho_\ell$ of $A$ are all distinct and independent on $M^0$).  Then on $M^0$ there is a structure 
of the integrable distribution  $\mathcal U$ generated by  $\gr\,\rho_i$ $(i=1,\dots,\ell)$ with the following properties:

\begin{enumerate}
\item The leaves of $\mathcal U$ are totally geodesic.

\item The leaves of $\mathcal U^{\bot}$ are common level surfaces of $\rho_1,\dots,\rho_\ell$. 

\item  Let  $\mathcal L\subset M^0$ be a leaf of $\mathcal U$, then $g|_{\mathcal L}$ and $A|_{\mathcal L}$ are 
compatible in the projective sense.

\item The non-constant eigenvalues $\rho_1,\dots, \rho_\ell$  can be considered as local coordinates on $\mathcal L$.

\item Locally the metric $g$ can be written as 
$$
g=\sum_{i=1}^\ell  g_{i}(\vec\rho) \d\rho_i^2   + \sum_{\alpha,\beta} b_{\alpha\beta}(\rho, y) \d y_\alpha \d y_\beta,
$$
where $\sum_{i=1}^\ell  g_{i}(\vec\rho) \d\rho_i^2$  is $g|_{\mathcal L}$.

\item  The vector fields $\gr\,\rho_i$ on $M$ and on $\mathcal L$ coincide (no matter which
metric, $g$ or $g|_{\mathcal L}$, is used to take the gradient)  and  
therefore the quantities  $g(\gr\, f_1(\vec\rho), \gr\, f_2(\vec\rho))$  
do not depend on how we compute them  (on $M$ or on $\mathcal L$).

\item The vector field $\Lambda = \frac{1}{4}\gr\,(\tr A)$ is the same on $\mathcal L$ and $M$. 
The tangent space of $\mathcal L$ is invariant with respect to the endomorphism $\nabla\Lambda$, moreover, 
the restriction  of $\nabla \Lambda$ on $\mathcal L$ coincides with $\nabla|_{\mathcal L} \Lambda$ computed on $\mathcal L$.  
This implies in particular that the eigenvalues of $\nabla|_{\mathcal L}\Lambda$ are some of the eigenvalues of $\nabla\Lambda$.

\item The leaves of $\mathcal U$ are (locally) isometric.   

\end{enumerate}
\end{prop}

\begin{proof}
In view of $\nabla J=0$ and the fact that $K_i,JK_i$ are holomorphic, (1) follows immediately from Lemma~\ref{lem:properties}~(8).

(2) is clear from the definition.
 
(3) follows immediately from (1) combined with equation \eqref{eq:main} (compare also Lemma~\ref{lem:redtorealproj}).

(4) follows from the assumption that  $\rho_1,\dots,\rho_\ell$ are independent.

(5) is immediate from the local classification Theorem~\ref{thm:localclassification} and (6) follows directly from (5). 
The first statement of (7) follows from (6) and the formula $\Lambda=\frac{1}{2}\sum_{i=1}^\ell\gr\,\rho_i$.
The remaining statements of (7) then follow directly from (1).

The last statement (8) can be seen from the formula for $g$ in (5).
\end{proof}


\begin{remthm}[for Theorem~\ref{specialcase2}]
\label{rem:rem9}
In the projective setting, the definition of the set $M^0$ of regular points remains essentially the same. We say that $p\in M$ is regular if: 1)   the algebraic type of $A$ does not change in a (sufficiently small) neighbourhood $U$ of $p$,  2) each eigenvalue $\rho$ is either constant on $U$ (and then $\rho$ equals either 0 or 1)  or, if $\rho$ is not constant on $U$, then  $\d \rho(p)\ne 0$. Recall that in view of \eqref{eq:ODEeigenvalues1}:
$$
\rho (p) \ne 0 \mbox{ or } 1   \quad \Rightarrow \quad  \d\rho(p) \ne 0.
$$
Clearly,  $M^0$ is open and everywhere dense on $M$.  However,  a priori  $M^0$ might contain several components related to different algebraic types.  In particular,  the number $\ell$ of non-constant eigenvalues may be different for different components of $M^0$.  We continue to work with one of the components of $M^0$  (for simplicity, we will still denote this component by $M^0$). 
Since we want to prove Theorem~\ref{specialcase2} by contradiction, we assume from now on that there is a regular point  (or, which is the same, one of the connected components of $M^0$) where the non-constant part of $A$ is diagonalisable (over $\R$), i.e. with no Jordan blocks. Recall  from \cite{BM} that in this case each non-constant eigenvalue has multiplicity one.

Summarizing we have an open subset $M^0 \subset M$ with the following properties:

\begin{enumerate}

\item  at each point $p\in M^0$,  the endomorphism $A$ has $\ell$  non-constant real eigenvalues $\rho_1<\dots<\rho_\ell$, each of multiplicity one;

\item  $\d \rho_i \ne 0$ everywhere on $M^0$;  

\item  the ``constant part'' of $A$ has some fixed algebraic type with at most two eigenvalues $0$ and $1$ of multiplicity $m_0$ and $m_1$ respectively.

\end{enumerate}

In particular, according to \cite{BM}, locally in a neighbourhood of every point $p\in M^0$, we can choose a coordinate system  $\rho_1, \dots, \rho_\ell, y_1, \dots, y_s$  (cf. Proposition 4.1) in which both $g$ and $A$ split:
$$
g = 
\begin{pmatrix}  
h(\rho) & 0 \\ 0 & g_{\mathrm{c}}(y)\cdot \chi_L\bigl(A_{\mathrm{c}}(y)\bigr) 
\end{pmatrix} 
\quad\mbox{and}\quad 
A = 
\begin{pmatrix} 
L(\rho) & 0 \\ 0 & A_{\mathrm{c}}(y)
\end{pmatrix} 
$$
Here $L=\mathrm{diag}(\rho_1,\dots,\rho_\ell)$,  the metric $h$ and the endomorphism $L$ are compatible in the projective sense,   $A_{\mathrm{c}}$ is parallel w.r.t. $g_{\mathrm{c}}$ and $\chi_L(t)=(t-\rho_1)(t-\rho_2)\dots (t-\rho_\ell)$ is the characteristic polynomial of $L$.
This decomposition into ``constant'' and ``non-constant'' blocks can naturally be reformulated in terms of two orthogonal foliations $\mathcal U$ and $\mathcal U^\bot$.  Proposition~\ref{prop:distributions}  remains true without any changes.

Also notice that \eqref{eq:PDE} holds without any change in spite of the fact that in \eqref{eq:PDEg_general}  we should replace $\frac{\beta}{2}$ by $\beta$.  This happens because $\frac 1 2 $ is compensated by the factor $2$ in the formula for the trace in the c-projective setting where $\tr A =  2 \sum\rho_i + const$  (whereas in the real case $\tr A = \sum\rho_i + const$). 

The further analysis deals with  the restrictions of $g$ and $v$ to leaves of $\mathcal U$.  These restrictions are exactly the same for the projective and c-projective cases and, until \S\ref{sssec:l11}, the proof for the both cases will be the same. The difference which appears in \S\ref{sssec:l11} will be clearly described. 
\end{remthm}


We now want to derive explicit formulas for the restrictions  of  
$g$ and $v$ to leaves of $\mathcal U$. This is sufficient for many discussions since 
some of the globally defined objects derived from $g$ and $A$ only depend on the level sets 
of the non-constant eigenvalues, see \S\ref{ssec:l23} below.

We first notice that according to Theorem~\ref{thm:localclassification}, at regular points 
$p\in M^0$ we have $\rho_i = \rho_i(x_i)$ and moreover $\frac{\partial \rho_i}{\partial x_i}\ne 0$.  
This means, in particular, that we can use $\rho_i$ as local coordinates instead of $x_i$.   
This change of variables will simplify further computations.

Let $v=v_1+v_2$ be the decomposition of the c-projective vector field $v$ w.r.t. $TM=\mathcal U\oplus \mathcal U^\bot$.
According to Proposition~\ref{prop:distributions}, there are certain functions $v^i_{1}$ 
such that $v$ can be written as
\begin{align}
v=\sum_{i=1}^\ell v^i_{1}\frac{\D}{\D \rho_i}+v_2.\label{eq:Ansatz_v}
\end{align}

Since each eigenvalue $\rho_i$  satisfies \eqref{eq:ODEeigenvalues1} and is constant in the direction of $\mathcal U^\bot$,  
we immediately see that
$$
v(\rho_i) = v^i_{1} \, \frac{\D \rho_i}{\D \rho_i} = \rho_i (1-\rho_i),
$$
that is,  $v^i_{1} = \rho_i (1-\rho_i)$.

From Theorem~\ref{thm:localclassification}, we know that in our simplified situation (no complex non-constant eigenvalues), 
$g$ and $gA$ can be written in the form (after the above change of variables $x_i \leftrightarrow \rho_i$)
\begin{align}
\begin{array}{c}
\displaystyle g=\sum_{i=1}^\ell \frac{\Delta_i}{F_i} \d \rho_i^2+\sum_{i,j=1}^\ell H_{ij}\theta_i\theta_j+g_{\mathrm {c}}(\chi_{\mathrm{nc}}(A_{\mathrm {c}})\cdot,\cdot),\vspace{1mm}\\
\displaystyle gA=\sum_{i=1}^\ell \rho_i \frac{\Delta_i}{F_i} \d \rho_i^2+\sum_{i,j=1}^\ell \tilde H_{ij}\theta_i\theta_j+g_{\mathrm {c}}(\chi_{\mathrm{nc}}(A_{\mathrm {c}})A_{\mathrm {c}}\cdot,\cdot),
\end{array}\label{eq:localclass_simple1}
\end{align}
where $\Delta_i=\prod_{j\neq i}(\rho_i-\rho_j)$,   $F_i = \varepsilon_i \left(\frac{\partial \rho_i}{\partial x_i} \right)^2$, 
the functions $H_{ij},\tilde H_{ij}$ only depend on $\rho_1,\dots,\rho_\ell$ and $\chi_{\mathrm{nc}}(t)=\prod_{i=1}^\ell(t-\rho_i)$.

\begin{prop}
\label{prop:localclassgvFi}
Locally, the metric $g$ and the c-projective vector field $v$ in the coordinates $\rho_1,\dots,\rho_\ell$ are given by the formulas
\begin{align}
\label{eq:localclassgvFi}
g=\sum_{i=1}^\ell\frac{\Delta_i}{F_i}\d\rho_i^2+\dots,\qquad v=\sum_{i=1}^\ell \rho_i(1-\rho_i)\frac{\D}{\D \rho_i}+\dots,
\end{align}
where $\Delta_i=\prod_{j\neq i}(\rho_i-\rho_j)$ and 
\begin{align}
F_i(t)=a_i(1-t)^{-\CC}t^{1+\ell+\CC}.\label{eq:Fi}
\end{align}
for some real constants $a_i$ and $\CC$.
\end{prop}

Note that 
the term $g-\sum_{i=1}^\ell\frac{\Delta_i}{F_i}\d\rho_i^2$ in formula \eqref{eq:localclassgvFi}, which is not written down explicitly,
coincides with $\sum_{\alpha,\beta} b_{ij}(\rho, y) dy_\alpha dy_\beta$ from part $(5)$ of Proposition~\ref{prop:distributions}. The expression 
$v-\sum_{i=1}^\ell \rho_i(1-\rho_i)\frac{\D}{\D \rho_i}$ is just the projection of $v$ onto $\mathcal U^\bot$.

\begin{proof}
It easily follows from the explicit form of $g$ given by \eqref{eq:localclass_simple1}  
that the second equation of  \eqref{eq:PDE} implies
\begin{equation}
\label{Lvrho}
\mathcal L_{v_1} \left( \sum_{i=1}^\ell\frac{\Delta_i}{F_i}\d\rho_i^2 \right) =  -\sum_{i=1}^\ell \left(\rho_i  
+  \sum_{j=1}^\ell \rho_j +\CC\right) \frac{\Delta_i}{F_i} \d \rho_i^2
\end{equation}
where $v_1 = \sum \rho_j(1-\rho_j) \frac{\D}{\D \rho_j}$. In other words, the first part of $g$ can be 
differentiated along $v$ independently of the remaining terms.

From this equation we can easily derive the explicit formulas for $F_i$. To that end, we compute 
the left hand side of \eqref{Lvrho}:

$$
\begin{aligned}
&\mathcal L_{v_1} \left( \sum_{i=1}^\ell\frac{\Delta_i}{F_i}\d\rho_i^2 \right) = \\
& \sum_{i=1}^\ell  \sum_{j\ne i} \rho_j(1-\rho_j) \frac{\D}{\D \rho_j} \left( \frac{\Delta_i}{F_i}\right) \d\rho_i^2  +
\sum_{i=1}^\ell  \rho_i(1-\rho_i) \frac{\D}{\D \rho_i} \left( \frac{\Delta_i}{F_i}\right) \d\rho_i^2  +   \sum_{i=1}^\ell\frac{\Delta_i}{F_i} L_{v_1}( \d\rho_i^2) =\\
&
\sum_{i=1}^\ell  \sum_{j\ne i} \frac{\rho_j(1-\rho_j)}{\rho_j-\rho_i}  \cdot \frac{\Delta_i}{F_i} \d\rho_i^2  +
\sum_{i=1}^\ell  \sum_{j\ne i}  \frac {\rho_i(1-\rho_i)}{\rho_i-\rho_j} \cdot \frac{\Delta_i}{F_i} \d\rho_i^2 
-  \sum_{i=1}^\ell  \rho_i(1-\rho_i) \frac{\D F_i}{\D\rho_i} \frac{\Delta_i}{F_i^2} \d\rho_i^2 + \\  
&+ 2 \sum_{i=1}^\ell\frac{\Delta_i}{F_i} \frac{\D}{\D\rho_i} \bigl(\rho_i (1-\rho_i)\bigr)\d\rho_i^2 =\\
&\sum_{i=1}^\ell  \sum_{j\ne i} \left( \frac{\rho_j(1-\rho_j)}{\rho_j-\rho_i}  + \frac {\rho_i(1-\rho_i)}{\rho_i-\rho_j}\right)  \cdot \frac{\Delta_i}{F_i}  \d\rho_i^2 
 -  \sum_{i=1}^\ell  \rho_i(1-\rho_i) \frac{\D F_i}{\D\rho_i} \frac{\Delta_i}{F_i^2} \d\rho_i^2 + 2  \sum_{i=1}^\ell\frac{\Delta_i}{F_i} (1-2\rho_i)\d\rho_i^2 =\\
& \sum_{i=1}^\ell  \left(   \sum_{j\ne i}( 1-\rho_i - \rho_j  )
-  \rho_i(1-\rho_i) \frac{\D \ln F_i}{\D\rho_i}  + 2 (1-2\rho_i) \right)  \frac{\Delta_i}{F_i} \d\rho_i^2 = \\
&\sum_{i=1}^\ell  \left(  \ell - 1 -(\ell-2)\rho_i - \sum_{j=1}^\ell \rho_j 
-  \rho_i(1-\rho_i) \frac{\D \ln F_i}{\D\rho_i}  + 2 (1-2\rho_i) \right)  \frac{\Delta_i}{F_i} \d\rho_i^2 = \\
&\sum_{i=1}^\ell  \left(  \ell + 1 -(\ell+2)\rho_i - \sum_{j=1}^\ell \rho_j 
-  \rho_i(1-\rho_i) \frac{\D \ln F_i}{\D\rho_i}   \right)  \frac{\Delta_i}{F_i} \d\rho_i^2
\end{aligned}
$$

Comparing this expression with the right hand side of \eqref{Lvrho}, we obtain a simple ODE on $F_i$:
$$
 \ell + 1 -(\ell+2)\rho_i - \sum_{j=1}^\ell \rho_j 
-  \rho_i(1-\rho_i) \frac{\D \ln F_i}{\D\rho_i}  = - \left(\rho_i  +  \sum_{j=1}^\ell \rho_j +\CC\right)
$$
which after simplification becomes
$$
 \frac{\D \ln F_i}{\D\rho_i}  = \frac{ \ell + 1 + \CC -(\ell+1)\rho_i}{\rho_i(1-\rho_i)} = \frac{\CC}{1-\rho_i} + \frac{1+\ell + \CC}{\rho_i}
$$
and we get
$$
F_i = a_i (1-\rho_i)^{-\CC} \rho_i^{1+\ell+\CC},
$$
as required. \end{proof}

In the proof of Proposition~\ref{prop:localclassgvFi}, we have actually shown that for $D(g,J)=2$, not only
$g$ and $A$ restricted to integral leaves $\mathcal L$ of $\mathcal U$ are compatible in the projective sense, but also
the projection $v_1$ of $v$ to $\mathcal L$ is a projective vector field for $h=g|_{\mathcal L}$. Indeed,  the second equation of 
\eqref{eq:PDE} we used in the construction implies  the second equation of \eqref{eq:PDE1} which is equivalent to the property of 
$v_1$ to be a projective vector field for  $h$. Moreover, we already know that $v_1=v_1(\rho)$ depends only on $\rho$ which is equivalent
to the fact that $v$ preserves the distribution $\mathcal{U}^\perp$.
Similarly, we have that $v_2=v_2(y)$ (in the notation of Proposition~\ref{prop:distributions}) for the component of $v$ tangent to $\mathcal{U}^\perp$ 
or, equivalently, that $v$ preserves $\mathcal{U}$.
To see this, note that we have $[v,K_i](\rho_j)=0$ for all $i,j=1,\dots,\ell$, hence, $[v,K_i]\in \mathcal{U}^\perp$. 
Since $v$ is holomorphic and, by the first equation in \eqref{eq:PDE}, preserves the generalised eigenspaces of $A$, we obtain that 
$[v,JK_i]\in \mathcal U$ for all $i=1,\dots,\ell$, hence, $v$ preserves $\mathcal U$. We summarize this discussion in the following 

\begin{cor}
\label{cor:reductionvf}
Let $v$ be a c-projective vector field and $D(g,J)=2$. Consider the natural decomposition of $v$  
associated with the distributions $\mathcal U$ and $\mathcal U^\bot$:
$$
v = v_1 + v_2,   \quad   v_1 \in \mathcal U,  v_2 \in \mathcal U^\bot.
$$  
Let $\mathcal L$ denote an integral leaf of $\mathcal U$ and denote by $h=g|_{\mathcal L}$ and $L=A|_{\mathcal L}$
the restrictions of $g$ resp. $A$ to $\mathcal L$. Then 

\begin{enumerate}

\item The vector field $v$ preserves the both distributions $\mathcal U$ and $\mathcal U^\bot$, that is,  
$v_1=v_1(\rho)$ and $v_2=v_2(y)$ (in the notation of Proposition~\ref{prop:distributions}).

\item The projection $v_1$ of $v$ onto $\mathcal L$ is a projective vector field for $h$ and the 
equations \eqref{eq:PDE} can be naturally restricted onto 
$\mathcal L$,  namely we have 
\begin{equation}
\label{eq:PDE1}
\begin{aligned}
&\mathcal{L}_{v_1} L=L(\Id- L),\\
&\mathcal{L}_{v_1} h=-h L -\left(\sum_{i=1}^\ell \rho_i + \CC  \right)h.
\end{aligned}
\end{equation}

\end{enumerate}
\end{cor}

\subsection{There is only one non-constant eigenvalue}
\label{ssec:l23}

The goal of this subsection is to prove

\begin{prop}
\label{thm:l23}
Let $A\in \mathcal{A}(g,J)$ and the assumptions be as in Proposition~\ref{thm:Bols}. Then $A$ cannot 
have more than one non-constant eigenvalue. 
\end{prop}

The idea of the proof is based on the analysis of geometric properties of the metric $g$ given by  
\eqref{eq:localclassgvFi} or, more precisely, of the restriction $h=g|_{\mathcal L}$.  
By construction these explicit formulas for $g$ are local, but we will show that they 
make sense for all admissible values of $\rho_i$.

\begin{prop}
\label{prop:global}
Consider the domain $U=\{ 0< \rho_1 < \dots < \rho_\ell <1\}$ with the metric 
$$
\sum_{i=1}^\ell\frac{\Delta_i}{F_i}\d\rho_i^2
$$
(compare \eqref{eq:localclassgvFi}). There is a natural isometric embedding of $\phi: U \to M$  
(as a maximal leaf of the totally geodesic foliation $\mathcal U$).
\end{prop}

\begin{proof}
Locally,  $\phi$ is defined in a very natural way. We choose a particular leaf $\mathcal L$ with $\rho_i$ 
as local coordinates and say that  $\phi(\rho_1,\dots, \rho_\ell)$ is the point on $\mathcal L$ with the same coordinates  
$\rho_1,\dots, \rho_\ell$.  We start with a certain point $a_0\in U$ and then extend this map as long as we can.  
The map $\phi$ so obtained is obviously an isometry. We need to show that such a prolongation can be made to any point of $U$.  

The argument is standard.  Consider a smooth curve $a(t)$ with $a(0)=a_0\in U$ and $a(1)=a_1\in U$, 
and choose $T_0\in [0,1]$ to be the supremum of those $T\in [0,1]$  for which the extension along the curve $a(t)$, 
$t\in [0, T]$, is well defined.   Take the image  $\phi(a(t))$, $t\in [0,T_0)$. Since $M$ is compact,  
we can find a limit point $p$ of $\phi(a(t))$ as $t\to T_0$.  By continuity,  the eigenvalues of $A(p)$ 
coincide with the coordinates of $a(T_0)$ in $U$, i.e. $0< \rho_1 < \dots < \rho_\ell <1$.  
But this condition guarantees that $p\in M^0$, 
i.e. in a neighbourhood of $p$, the foliation $\mathcal U$ is defined.  This obviously implies that $p\in \mathcal L$, 
moreover $p=\phi(a(T_0))$ and we can extend $\phi$ to some neighbourhood of $a(T_0)\in U$.   
In particular,  $T_0$ cannot be an interior point of $[0,1]$. 
In other words,  the prolongation of $\phi$  along $a(t)$ is well defined for all $t\in[0,1]$.  \end{proof}

Consider the function $f:U \to \R$:
\begin{equation}
\label{frho}
f (\vec\rho)=h\bigl(\gr\sum_{i=1}^\ell \rho_i,\gr\sum_{i=1}^\ell \rho_i\bigr),
\end{equation}
where $h$ is the metric on $U$ defined explicitly by (the first term of $g$ in) \eqref{eq:localclassgvFi}.

According to Proposition~\ref{prop:global}, we can naturally identify the domain $U=\{ 0< \rho_1 < \dots < \rho_\ell <1\}$  
with a leaf $\mathcal L$ of $\mathcal U$.  Moreover, the function $f$ can be considered (up to a constant multiple) as the restriction to
$\mathcal L$ of the function $g(\gr\, \tr A, \gr\, \tr A)$ which is globally defined and smooth on $M$  
(see part $(6)$ of Proposition~\ref{prop:distributions}).  
This  immediately implies certain conditions on $f(\vec\rho)$.

\begin{prop}
\label{prop:limit}  
The function $f(\vec\rho)$ must be bounded on $U$.  Moreover, 
we have $f(\vec\rho)\to 0$ as $(\rho_1,\dots,\rho_\ell)\to (0,\dots,0)$ or $(\rho_1,\dots,\rho_\ell)\to (1,\dots,1)$. 
\end{prop}

\begin{proof} 
The first claim follows from the compactness of $M$. Since $\mathrm{tr}\,A$ takes 
its minimum resp. maximum value at the limit points $(0,\dots,0)$ resp. $(1,\dots,1)$, 
we obtain that $\gr\,\mathrm{tr}\,A$ tends to zero if $(\rho_1,\dots,\rho_\ell)$ tends 
to $(0,\dots,0)$ or $(1,\dots,1)$. The proves the proposition.
\end{proof}

These conditions will give us some further restrictions on $g$. It is straightforward to compute the
 function $f(\vec\rho)$ explicitly using \eqref{eq:localclassgvFi}:
\begin{align}
f (\vec\rho)=\sum_{i=1}^\ell \frac{F_i}{\Delta_i}
=\sum_{i=1}^\ell \frac{a_i(1-\rho_i)^{-\CC}\rho_i^{1+\ell+\CC}}{\prod_{j\neq i}(\rho_i-\rho_j)}
\label{eq:fPrho}
\end{align}

To study the limiting behaviour of such functions we use the following

\begin{lem}
\label{lem:determinant}
\begin{enumerate}
\item Let $k_1,\dots,k_\ell$ be functions of one variable. Then the function
$$
f=\sum_{i=1}^\ell\frac{k_i(\rho_i)}{\prod_{j\neq i}(\rho_i-\rho_j)},
$$
defined on the domain $\rho_1<\dots<\rho_\ell$, is equal to the quotient of determinants
$$
\mathrm{det}\left(\begin{array}{cccc}1&1&\dots&1\\\rho_1&\rho_2&\dots&\rho_\ell\\\vdots&\vdots&\ddots&\vdots\\\rho_1^{\ell-2}&\rho_2^{\ell-2}&\dots&\rho_\ell^{\ell-2}\\k_1(\rho_1)&k_2(\rho_2)&\dots&k_\ell(\rho_\ell)\end{array}\right)\mathrm{det}\left(\begin{array}{cccc}1&1&\dots&1\\\rho_1&\rho_2&\dots&\rho_\ell\\\vdots&\vdots&\ddots&\vdots\\\rho_1^{\ell-2}&\rho_2^{\ell-2}&\dots&\rho_\ell^{\ell-2}\\\rho_1^{\ell-1}&\rho_2^{\ell-1}&\dots&\rho_\ell^{\ell-1}\end{array}\right)^{-1}
$$ 
\item If  $f$ is bounded on the domain $\rho_1<\dots<\rho_\ell$, then $k_1=\dots=k_\ell$.
\item We have $f(\rho_1,\dots,\rho_\ell)\equiv 0$ if and only if $k_i(\rho_i)=p(\rho_i)$, 
where $p(\rho_i)$ is a polynomial in $\rho_i$ (independent of $i$) of degree $\leq \ell-2$.
\end{enumerate}
\end{lem}

\begin{proof}
The proof of (1) follows from standard manipulations for calculating determinants and by applying 
the formula for the determinant of the Vandermonde matrix. 

To prove (2), consider for instance the limit $\rho_1\to\rho_2$ under which the Vandermonde determinant in the 
denominator of the expression for $f$ tends to zero. Since $f$ is bounded, the determinant in the numerator must 
also tend to zero which implies that $k_1$ and $k_2$ are equal.  

Part (3) follows immediately from the formula in part $(1)$ and the vanishing of the determinant in the nominator.
\end{proof}

We apply part $(2)$ of Lemma~\ref{lem:determinant} to the function $f(\vec\rho)$ from  \eqref{frho}   
written in the form \eqref{eq:fPrho} to obtain

\begin{cor}
\label{cor:equalconstants}
The parameters $a_i$'s  in  \eqref{eq:localclassgvFi}, \eqref{eq:Fi} are all equal. In other words, 
the functions $F_i$ from \eqref{eq:localclassgvFi} coincide and take the form
$$
F_i(t)=F(t)=a(1-t)^{-\CC}t^{1+\ell+\CC}
$$
\end{cor}

To find restrictions on the constant $\CC$ and the number $\ell$ of non-constant eigenvalues we use 
another interesting property of functions of the form \eqref{eq:fPrho}.

\begin{lem}
\label{lem:zeros}
Let 
$$
f(\vec{\rho})=\sum_{i=1}^\ell\frac{k(\rho_i)}{\prod_{j\neq i}(\rho_i-\rho_j)},
$$
where $\vec{\rho}=(\rho_1,\dots,\rho_\ell)$. If $k$ is smooth in a neighbourhood of $x\in\R$, then 
$$
\lim_{\vec{\rho}\to \vec{x}}f(\vec{\rho})=\frac{1}{(\ell-1)!}k^{(\ell-1)}(x),
$$
where we define $\vec{x}=(x,\dots,x)$.
\end{lem}

\begin{proof}
In a neighbourhood of $x$, we can write $k$ in the form
$$
k(t)=\sum_{j=0}^{\ell-1}\frac{1}{j!}k^{(j)}(x)(t-x)^j+\mathcal{O}((t-x)^\ell).
$$
Inserting this into the formula for $f$, we obtain 
$$
f(\vec{\rho})=\sum_{j=0}^{\ell-1}\frac{1}{j!}k^{(j)}(x)\sum_{i=1}^\ell\frac{(\rho_i-x)^j}{\prod_{r\neq i}(\rho_i-\rho_r)}
+\sum_{i=1}^\ell\frac{\mathcal{O}((\rho_i-x)^\ell)}{\prod_{j\neq i}(\rho_i-\rho_j)}.
$$
Applying Lemma~\ref{lem:determinant} (3) to the functions $k_i(\rho_i)$ of the form $(\rho_i-x)^j$, we obtain
$$
f(\vec{\rho})=\frac{1}{(\ell-1)!}k^{(\ell-1)}(x)+\sum_{i=1}^\ell\frac{\mathcal{O}((\rho_i-x)^\ell)}{\prod_{j\neq i}(\rho_i-\rho_j)}.
$$
The claim now follows by taking the limit $\vec{\rho}\to\vec{x}$.
\end{proof}

The limiting behaviour $f (\vec\rho) \to 0$ for $\vec\rho\to (1,\dots,1)$ and $\vec\rho\to (0,\dots,0)$ 
(see Proposition~\ref{prop:limit}) 
combined with Lemma~\ref{lem:zeros}  with $k(t)= F(t)=a (1-t)^{-\CC}t^{1+\ell+\CC}$
implies that the function $F(t)$ must have a zero of order $\ell-1$ at $t=0$ and $t=1$. Hence,
$$
-\CC>\ell-1\mbox{ and }1+\ell+\CC>\ell-1
$$
or equivalently
$$
-2<\CC<1-\ell.
$$
This inequality is not fulfilled for $\ell\geq 3$. Thus,  we have

\begin{cor}
\label{cor:aboutell}
The number $\ell$ of non-constant eigenvalues is either $1$ or $2$. Moreover, we have 
$$
-2<\CC<-1 \ \ \mbox{for }\ell=2 \qquad\mbox{and} \qquad -2<\CC<0 \ \ \mbox{for }\ell=1.
$$
\end{cor}

Now let us show that also $\ell=2$ contradicts our assumptions.  
If $\ell=2$, then the metric $h=g|_{\mathcal L}$  (see Proposition~\ref{prop:localclassgvFi}) takes the form
\begin{equation}
\label{eq:gindim2}
(\rho_1 - \rho_2) \left(  \frac{\d\rho^2_1}{F(\rho_1)}  - \frac{\d\rho^2_2}{F(\rho_2)}  \right) 
\end{equation}
where $F(t)= a (1-t)^{-\CC}t^{3+\CC}$ and  $-2<\CC<-1$. Without loss of generality we may assume that $a=1$.

Let us compute one special eigenvalue of the curvature operator of the metric $g$ on $M$ by using Proposition~\ref{prop:eigvalcurv}, see Appendix.   
We know that each eigenvector of $A$ (e.g., $\gr\,\rho_i$) is at the same time an eigenvector of $\nabla\Lambda$ (Lemma~\ref{lem:RicciId} and Remark~\ref{rem:commute}).  In particular, we must have 
$$
\nabla_{\gr\,\rho_i} \Lambda=m_i\gr\,\rho_i \quad \mbox{for }i=1,2,
$$
where the functions $m_1,m_2$ are the eigenvalues of the endomorphism $\nabla\Lambda$ 
(of course, this observation is not limited to the case $\ell=2$).  
Then according to Proposition~\ref{prop:eigvalcurv}  and Remark \ref{rem:eigvalcurv}, the function 
$$
\lambda=\frac{m_1-m_2}{\rho_1-\rho_2}
$$
is an eigenvalue of the curvature operator.

\begin{prop}
\label{prop:lambda}
The eigenvalue $\lambda$ of the curvature operator is given (up to multiplication with a non-zero constant) by the formula
\begin{align}
\lambda=\frac{(\rho_1-\rho_2)(F'(\rho_1)+F'(\rho_2))+2(F(\rho_2)-F(\rho_1))}{4(\rho_1-\rho_2)^3},\label{eq:lambda}
\end{align}
where $F(t)=(1-t)^{-\CC}t^{3+\CC}$.
\end{prop}

\begin{proof}
We have $\Lambda=\frac{1}{2}(\gr\,\rho_1+\gr\,\rho_2)$, hence,
$$
\nabla_{\Lambda} \Lambda=\frac{1}{2}\nabla_{\gr\,\rho_1}\Lambda+\frac{1}{2}\nabla_{\gr\,\rho_2}\Lambda=\frac{1}{2}(m_1\gr\,\rho_1+m_2\gr\,\rho_2).
$$
Dualizing this (and using that $\nabla\Lambda$ is $g$-selfadjoint), we obtain
$$
\d g(\Lambda,\Lambda)=m_1\d\rho_1+m_2\d\rho_2.
$$
Thus, to find the formulas for $m_1$ and $m_2$ it remains to calculate the differential of $g(\Lambda,\Lambda)$.

By \eqref{eq:gindim2} we have  
$$
g(\Lambda,\Lambda)=\frac{1}{4}(|\gr\rho_1|^2+|\gr\rho_2|^2)=\frac{1}{4}\frac{F_1(\rho_1)-F_2(\rho_2)}{\rho_1-\rho_2}.
$$
The functions $F_i(t)$ are given by formula \eqref{eq:Fi} (with $\ell=2$) and by Corollary~\ref{cor:equalconstants}, 
they are all equal. Then,
$$
\d g(\Lambda,\Lambda)=\frac{1}{4(\rho_1-\rho_2)^2}\Big[((\rho_1-\rho_2)F'(\rho_1)+F(\rho_2)-F(\rho_1))\d\rho_1
$$
$$
-((\rho_1-\rho_2)F'(\rho_2)+F(\rho_2)-F(\rho_1))\d\rho_2\Big]
$$
and we obtain
$$
m_1=\frac{(\rho_1-\rho_2)F'(\rho_1)+F(\rho_2)-F(\rho_1)}{4(\rho_1-\rho_2)^2}
$$
and 
$$
m_2=-\frac{(\rho_1-\rho_2)F'(\rho_2)+F(\rho_2)-F(\rho_1)}{4(\rho_1-\rho_2)^2}.
$$
Thus, $\lambda=(m_1-m_2)/(\rho_1-\rho_2)$ is given by formula \eqref{eq:lambda} as we claimed.
\end{proof}

From Proposition~\ref{prop:distributions}, we also know that this number $\lambda$ can be computed for 
the restriction $h$ of $g$ 
to the totally geodesic leaves of the distribution $\mathcal U$, i.e. for the metric 
\eqref{eq:gindim2}. Notice that  \eqref{eq:lambda} coincides with the usual formula for the scalar curvature of
$h$.
Being an eigenvalue of the curvature operator of $g$,  the function  $\lambda$ must be bounded 
on $U=\{0<\rho_1<\rho_2<1\}$.

\begin{lem}
\label{lem:aboutF}
If $F$ is smooth at $x\in\R$, then the function $\lambda(\rho_1, \rho_2)$ given by \eqref{eq:lambda} 
is bounded in a neighbourhood of the point $(x,x)$. Moreover,  
$$
\lim_{(\rho_1,\rho_2) \to (x,x)} \lambda(\rho_1, \rho_2) = \frac{1}{24} F'''(x).
$$
Conversely,   if $\lim_{t\to x} F'''(t) = \infty$, then $\lambda$ is not bounded as $(\rho_1,\rho_2) \to (x,x)$.
\end{lem}

\begin{proof}
Setting $y=\rho_1$ and $x=\rho_2$ in formula \eqref{eq:lambda} for $\lambda$ and inserting 
the Taylor expansion at the point $x$ of $F(y)$ considered as a function of $y$, we obtain
$$
\lim_{y\to x} \lambda(y,x)=\lim_{y\to x} \frac{(y-x)(F'(y) + F'(x)) + 2 (F(x) - F(y))}{4(y-x)^3} 
$$
$$
=\lim_{y\to x} \frac{1}{4(y-x)^3}\Big((y-x)(2F'(x) + F''(x)(y-x) + \frac{1}{2}F'''(x)(y-x)^2 )
$$
$$
-2(F'(x) (y-x) + \frac{1}{2}F''(x)(y-x)^2 + \frac{1}{6}F'''(x)(y-x)^3)+\mathcal{O}((y-x)^4)\Big) 
$$
$$
=\lim_{y\to x} \frac{\frac{1}{6}F'''(x) (y-x)^3 + \mathcal{O}((y-x)^4)}{4(y-x)^3}=\frac{1}{24} F'''(x).
$$
\end{proof}

Now it is easy to see that the condition that $F'''(t)$ is bounded as $t\to 0$ and $t\to 1$ for 
$F(t)=(1-t)^{-\CC}t^{3+\CC}$ can only be fulfilled for $\CC=0,-1,-2, -3$  
(by the way, in this case $\lambda$ is constant).  But in our case, $-2<\CC<-1$ so that  $\lambda$ 
goes to infinity  either for $(\rho_1,\rho_2) \to (1,1)$ or $(\rho_1,\rho_2) \to (0,0)$. 

Thus, we conclude that $\ell=2$ is forbidden and the only remaining case is $\ell=1$.
 This completes the proof of Proposition~\ref{thm:l23}.


\begin{remthm}[for Theorem~\ref{specialcase2}] 
In the proof of Proposition~\ref{prop:lambda}, we used Proposition~\ref{prop:eigvalcurv} to derive a formula for 
one of the eigenvalues of the curvature operator.   An analogue of 
Proposition~\ref{prop:eigvalcurv} holds in the  non-K\"ahler  case too  (see \cite{BFom70}) and the proof remains 
essentially the same.  Thus, Proposition~\ref{prop:lambda} and Lemma~\ref{lem:aboutF}  remain unchanged and we obtain 
that the number of non-constant eigenvalues $\ell$ is $1$ (in a neighbourhood of a regular point).  
\end{remthm}



\subsection{Proof of Theorem~\ref{thm:yano_obata} when there is only one non-constant eigenvalue}
\label{sssec:l11}

We deal with the PDE system \eqref{eq:PDE} which, in the case one single non-constant eigenvalue $\rho$ of $A$, takes the form 
\begin{align}
\begin{array}{c}
\mathcal{L}_v A=A(\Id- A),\vspace{1mm}\\
\mathcal{L}_v g=-gA-\left(\rho+\CC \right)g.
\end{array}\label{eq:PDEl1}
\end{align}

Recall from Proposition~\ref{prop:distributions} that
we have two mutually orthogonal integrable distributions $\mathcal U$ and $\mathcal U^\bot$ on $M$,   
the first one $\mathcal U$ being $1$-dimensional and totally geodesic 
and the metric takes the following matrix form
$$
g =
\begin{pmatrix}   
g_1(\rho) & 0 \\ 0 & g_2(\rho, y)
\end{pmatrix}
$$
w.r.t. the orthogonal decomposition $TM=\mathcal U\oplus \mathcal U^\bot$.
We also have a $c$-projective  vector field $v$ preserving both $\mathcal U$ and $\mathcal U^\bot$, so that 
$v = v_1(\rho) + v_2(y)$, where $v_1$ and $v_2$ are the components of 
	$v$ w.r.t. $\mathcal U$ resp. $\mathcal U^\bot$ (see also Corollary~\ref{cor:reductionvf}). Hence
  \begin{equation}
  \label{eq:lvg}
  \mathcal L_v g =
  \begin{pmatrix}
  \mathcal L_{v_1} g_1 & 0 \\
0  & \mathcal L_{v_2} g_2 +  \mathcal D_{v_1} g_2
   \end{pmatrix},
  \end{equation}
  where $\mathcal D_{v_1} $ means that we differentiate each term of the matrix $g_2$ along $v_1$.

  Our goal is to analyse how the volume form of the metric $g_2$, defined on the leaves of $\mathcal U^\bot$,  
	is changing under the flow generated by $v_2$.  In other words, we want to compute the coefficient 
	$f(\rho, y)$ in the formula $\mathcal L_{v_2} \mathrm{vol}_{g_2} = f(\rho, y)\cdot \mathrm{vol}_{g_2}$.  
	We will show that this coefficient is constant. Namely,
  
  \begin{prop}
  \label{prop:volume}
  We have $\mathcal L_{v_2} \mathrm{vol}_{g_2} = (-\CC-1)(m_1+m_0+1) \mathrm{vol}_{g_2}$.
  \end{prop}

 \begin{proof}

 By \eqref{eq:PDEl1}, we have $\mathcal L_v g = -g \cdot  \bigl(A + (\rho +\CC)\Id\bigr)$.  
Since $A$  admits a natural splitting $\begin{pmatrix}  \rho  &0 \\ 0 & A_2 \end{pmatrix}$ 
w.r.t. $\mathcal U$ and $\mathcal U^\bot$, we get  
 $$
\mathcal L_v g =  
\begin{pmatrix}
 -\frac{1}{F}  (2\rho + \CC)& 0 \\
0  &  -g_2 \bigl(A_2 +(\rho+\CC)\mathrm{Id}\bigr)  
   \end{pmatrix}
  $$  
  Hence, comparing with \eqref{eq:lvg}, we obtain  $\mathcal L_{v_2} g_2  =- g_2 \bigl(A_2 +(\rho+\CC)\mathrm{Id}\bigr) -   \mathcal D_{v_1} g_2$. 
  
  We now use the following general formula that explains the relation between  $\mathcal L_{v_2} g_2$ 
	and $\mathcal L_{v_2} \mathrm{vol}_{g_2}$:
 $$
\mathcal L_{v_2} \mathrm{vol}_{g_2} = f \cdot   \mathrm{vol}_{g_2}, \quad \mbox{where }  
f = \frac{1}{2}\tr (g_2^{-1} \mathcal L_{v_2} g_2).
 $$

  Hence in our case
   $$
 f = -\frac{1}{2}\, \tr \bigl( A_2 + (\rho + \CC)\mathrm{Id} +  g_2^{-1}  \mathcal D_{v_1} g_2\bigr).
  $$

Let us compute $\tr (g_2^{-1}  \mathcal D_{v_1} g_2)$.  In our case $v_1=\rho(1-\rho)\frac{\partial}{\partial\rho}$ and the metric $g_2$ 
has the form
 $$
 F(\rho) \theta^2 + g_{\mathrm{c}}\bigl( (A_{\mathrm{c}}-\rho \cdot\Id) \cdot , \cdot\bigr),
 $$
 hence, $ \mathcal D_{v_1} g_2 = \rho (1-\rho)  F'(\rho) \theta^2  -  \rho(1-\rho) g_{\mathrm{c}}(\cdot, \cdot)$.
We may think of $g_2$ as a block diagonal form,  then  $\mathcal D_{v_1} g_2$ is block-diagonal too so that we can write it as 
$D_{v_1} g_2 = g_2  C$, where $C$ is an endomorphism with the matrix
 $$
 C= 
 \begin{pmatrix}     
 \rho (1-\rho) \frac{ F'(\rho)}{F(\rho)} & 0 \\ 0 & - \rho(1-\rho) (A_{\mathrm{c}}-\rho)^{-1}            
  \end{pmatrix}
 $$
 Therefore we have the following formula:
 $$
 \tr  g_2^{-1}  \mathcal D_{v_1} g_2 = \tr  C =  \rho (1-\rho)  \frac{ \D \ln F}{ \D \rho} -   \rho (1-\rho) \tr (A_{\mathrm{c}}-\rho)^{-1}.
 $$
  The first term can be easily computed from the explicit formula for $F$.   
	The second term is easy to find as we know that the eigenvalues of $A_{\mathrm{c}}$ 
	are $0$ and $1$ with multiplicities $2m_0$ and $2m_1$ respectively.  We get  
 $$
  \tr  g_2^{-1}  \mathcal D_{v_1} g_2 = (2+ \CC - 2\rho)  -   
	\rho (1-\rho) \bigl(   -2m_0 \rho^{-1} + 2m_1 (1-\rho)^{-1}      \bigr) = 2+\CC +2m_0 - \rho(2m_0 + 2m_1+2)
 $$
  
The matrix of $A_2$ is known, namely
$A_2 =  
\begin{pmatrix}
\rho & 0 \\
0 &  A_{\mathrm{c}} 
\end{pmatrix} $.
So $\tr A_2 = \rho + \tr A_{\mathrm{c}} = \rho + 2m_1$.  Thus, finally
$$
\begin{aligned}
f = &-\frac{1}{2} \bigl(  \rho + 2m_1 + (\rho + \CC) (2m_1+2m_0 + 1)  +  2 +\CC +2m_0 - \rho(2m_0 + 2m_1+2)
\bigr)=\\
&(-\CC - 1) (m_0 + m_1 + 1),
\end{aligned}
$$
as claimed.    \end{proof}

 This proposition immediately implies that $\CC=-1$.  Indeed,  consider $M_\rho = \{ \rho= c\}$ where $c\in (0,1)$.  
Then $M_\rho$  is a compact smooth submanifold of $M$ which entirely belongs to the set of regular points $M^0$ and, 
therefore,  the formula from Proposition~\ref{prop:volume} holds for the volume form on $M_\rho$ as a whole. However, 
due to compactness of $M_\rho$, this is impossible unless $-\CC - 1=0$. Thus,  we have completely reconstructed the ``non-constant'' 
block of the metric $g$  and now we can rewrite $g$   as follows:
\begin{equation}
\label{eq:gfinal}
g =  F(\rho)^{-1} \d\rho^2 + F(\rho) \theta^2 + g_{\mathrm{c}}\bigl( (A_{\mathrm{c}} -\rho\,\mathrm{Id})\cdot , \cdot\bigr), 
\end{equation}
where $F(\rho)=-4B (1-\rho) \rho$ and $B\ne 0$ is some constant (the notation $-4B$ for the constant factor is chosen to 
emphasize the relationship with some formulas from \cite{FKMR} that are used at the final stage of our proof).

\begin{remthm}[for Theorem~\ref{specialcase2}]
 The proof of and the formula in Proposition~\ref{prop:volume} slightly changes because now the terms with $\theta$ do not appear.   
Here is the modified version in the projective setting:

  \begin{prop}
  \label{prop:volumeLich}
  We have $\mathcal L_{v_2} \mathrm{vol}_{g_2} = \frac{1}{2}(-\CC-1)(m_1+m_0)  \mathrm{vol}_{g_2}$.
  \end{prop}

  \begin{proof}   
By \eqref{eq:PDEl1}, we have $\mathcal L_v g = -g \cdot  \bigl(A + (\rho +\CC)\Id\bigr)$.  
Since $A$  admits a natural splitting $\begin{pmatrix}  \rho  &0 \\ 0 & A_2 \end{pmatrix}$ 
w.r.t. $\mathcal U$ and $\mathcal U^\bot$, we get  
 $$
\mathcal L_v g =  
\begin{pmatrix}
 -\frac{1}{F}  (2\rho + \CC)& 0 \\
0  &  -g_2 \bigl(A_2 +(\rho+\CC)\mathrm{Id}\bigr)  
   \end{pmatrix}
  $$  
  Hence, comparing with \eqref{eq:lvg}, we obtain  $\mathcal L_{v_2} g_2  =- g_2 \bigl(A_2 +(\rho+\CC)\mathrm{Id}\bigr) -   \mathcal D_{v_1} g_2$. 
  Using
 $$
\mathcal L_{v_2} \mathrm{vol}_{g_2} = f \cdot   \mathrm{vol}_{g_2}, \quad \mbox{where }  
f = \frac{1}{2}\tr (g_2^{-1} \mathcal L_{v_2} g_2),
 $$
we obtain in our case
   $$
 f = -\frac{1}{2}\, \tr ( A_2 + (\rho + \CC)\mathrm{Id} +  g_2^{-1}  \mathcal D_{v_1} g_2).
  $$
 Let us compute $\tr(g_2^{-1}  \mathcal D_{v_1} g_2)$.  In our case $v_1=\rho(1-\rho)\frac{\partial}{\partial\rho}$ and the metric $g_2$ 
has the form $g_2(\cdot,\cdot) = g_{\mathrm {c}}\bigl( (A_{\mathrm {c}}-\rho \cdot\Id) \cdot , \cdot\bigr)$, hence,
$$
 \mathcal D_{v_1} g_2 =  -  \rho(1-\rho) g_{\mathrm{c}}(\cdot, \cdot) 
 $$
Thus we can write  $\mathcal D_{v_1} g_2$  as $D_{v_1} g_2 = g_2  C$, where $C$ is an endomorphism with the matrix
 $$
 C=  - \rho(1-\rho) (A_{\mathrm{c}}-\rho)^{-1}            
 $$
 Therefore we have the following formula:
 $$
 \tr \, g_2^{-1}  \mathcal D_{v_1} g_2 = \tr \, C =  -   \rho (1-\rho) \tr(A_{\mathrm{c}}-\rho)^{-1}.
 $$
 This quantity is easy to find as we know the eigenvalues of $A_{\mathrm{c}}$ are $0$ and $1$ 
with multiplicities $m_0$ and $m_1$ respectively.  We get  

 $$
  \tr \, g_2^{-1}  \mathcal D_{v_1} g_2 =    -   
	\rho (1-\rho) \bigl(   -m_0 \rho^{-1} + m_1 (1-\rho)^{-1}      \bigr) = m_0 (1-\rho) - m_1\rho
 $$
  
 The matrix of $A_2$  coincides with $A_{\mathrm{c}}$ so that $\tr A_2 =  \tr A_{\mathrm{c}} = m_1$.  Thus, finally
$$
\begin{aligned}
f = -\frac{1}{2} \bigl(  m_1 + (\rho + \CC) (m_1+m_0)  +  m_0(1-\rho) - m_1\rho 
\bigr)= \frac{1}{2} (-\CC - 1) (m_0 + m_1),
\end{aligned}
$$
as claimed.    \end{proof}

The conclusion from this proposition remains unchanged: $\CC = -1$ as required. 
We use the fact that the leaves of $\mathcal U^\bot$ are smooth closed submanifolds of the form $\{\rho = c\}$ which are entirely located in the set of regular points.  Since the leaves of $\mathcal U^\bot$ are common levels of non-constant eigenvalues, the latter follows from the fact (used already several times) that  $0<\rho <1  $ implies $  \d\rho(p) \ne 0$. Hence we have the following formula for the metric  (this is \eqref{eq:gfinal} with the term with $\theta$ removed):
\begin{equation}
\label{eq:gfinalLich}
g =  F(\rho)^{-1} \d\rho^2 +  g_{\mathrm {c}}\bigl( (A_{\mathrm{c}} -\rho\,\mathrm{Id})\cdot , \cdot\bigr), 
\end{equation}
where $F(\rho)=-4B (1-\rho) \rho$, $B\ne 0$ is some constant and $A_{\mathrm{c}}$ is parallel w.r.t. $g_{\mathrm{c}}$.  
\end{remthm}
\medskip

Denote by $E_0$ resp. $E_1$ the generalised 
eigenspaces of $A$ corresponding to the eigenvalues $0$ resp. $1$. Let $L_0=A|_{E_0}$ 
and $L_1=A|_{E_1}$ denote the restrictions of $A$ to these subspaces. We start with describing $\nabla\Lambda$ 
restricted to $E_0$ resp. $E_1$  explicitly as a matrix function of $L_0$ resp. $L_1$. 
In fact, it is a general statement (see Remark~\ref{rem:commute} in the appendix), that $\nabla\Lambda$ and $A$ commute and, moreover,
at each point $\nabla \Lambda$ can be written as a function (or even polynomial) of $A$.
 
\begin{lem}
\label{lem:polynomialsB} 
At each   regular point, we have 
\begin{equation}
\label{eq:nablaL}
\nabla \Lambda |_{E_0\oplus E_1}= -g(\Lambda,\Lambda) (A - \rho\,\mathrm{Id})^{-1}|_{E_0\oplus E_1}=  B (1-\rho) \rho (A - \rho\,\mathrm{Id})^{-1}|_{E_0\oplus E_1}.
\end{equation}    
\end{lem}

\begin{proof}
In the special case of only one non-constant eigenvalue $\rho$, $\Lambda=\frac{1}{2}\gr\,\rho$ 
itself is an eigenvector field of $A$ corresponding to the eigenvalue $\rho$. Hence, we can apply \eqref{eq:covderiv} 
with $X$ replaced by $\Lambda$.  Then \eqref{eq:covderiv} becomes
$$
(A-\rho\, \Id)\nabla_Y\Lambda=-g(\Lambda,\Lambda)Y,
$$
for any tangent vector $Y\in E_0\oplus E_1$, or equivalently,  if we take into account that $E_0\oplus E_1$ is invariant under $A-\rho\cdot\mathrm{Id}$ and  $(A-\rho\cdot\mathrm{Id})|_{E_0\oplus E_1}$ is invertible:
$$
\nabla \Lambda|_{E_0\oplus E_1} = -g(\Lambda,\Lambda) (A - \rho\,\mathrm{Id})^{-1}|_{E_0\oplus E_1}.  
$$
It remains to notice that \eqref{eq:gfinal} implies $g(\Lambda,\Lambda)=\frac{1}{4}g(\gr\,\rho, \gr\,\rho)=B (1-\rho) \rho$, as stated.  
\end{proof}

\begin{lem}
\label{lem:noJordanB} 
At each regular point, 
we have $A|_{E_0}=L_0=0$ and $A|_{E_1}=L_1=\Id$.
\end{lem}

\begin{proof}
Our statement is equivalent to the absence of non-trivial Jordan blocks corresponding to the constant eigenvalues $0$ and $1$.   By contradiction, assume that non-trivial Jordan blocks exist and apply  Proposition~\ref{prop:eigvalcurv} to compute the eigenvalues of the curvature operator of $g$ related to these blocks.

Formula \eqref{eq:nablaL}  expresses $\nabla \Lambda|_{E_0\oplus E_1} $ (pointwise) as a matrix function $f(A|_{E_0\oplus E_1} )$, where $f(t)=B(1-\rho)\rho (t -\rho)^{-1}$.   Then according to Proposition~\ref{prop:eigvalcurv} and Remark~\ref{rem:eigvalcurv} from Appendix,  
we conclude that if $L_0$ (resp. $L_1$) has a non-trivial Jordan block, then
$$
4f'(0)=-4B \frac{1-\rho}{\rho}\quad \left(\mbox{resp. } 4f'(1)=-4B\frac{\rho}{1-\rho}\right)
$$
is an eigenvalue of the curvature operator $R$ of $(M,g,J)$. When restricted to a non-constant integral curve of $v$, 
this eigenvalue goes to infinity  for $t\to-\infty$ resp. $t\to+\infty$ which contradicts to the boundedness of the eigenvalues of $R$. Thus, we conclude $L_0=0$ and $L_1=\Id$.
\end{proof}

\begin{remthm}[for Theorem~\ref{specialcase2}]
Lemma~\ref{lem:polynomialsB} uses formula \eqref{eq:covderiv}  which, in the pseudo-Riemannian case takes the form:
$$
(A - \rho \, \Id) \nabla_Y X =  \d \rho (Y)X - g(X,Y)\Lambda - g(X,\Lambda) Y  
$$
where $\Lambda = \frac{1}{2} \gr\, \tr A= \frac{1}{2} \gr\,\rho$, $X$ is a vector field satisfying $(A - \rho \, \Id) X=0$, i.e., a $\rho$-eigenvector field and $Y$ is an arbitrary vector field.

Notice that in this case $\Lambda$ satisfies the condition for $X$, so in this formula we may set $X=\Lambda$. Then we get
$$
(A - \rho \, \Id) \nabla_Y \Lambda = 2 g(\Lambda,Y) \Lambda - g(\Lambda,\Lambda) Y - g(Y,\Lambda) \Lambda =
g(\Lambda,Y) \Lambda - g(\Lambda,\Lambda) Y,  
$$
and if we assume that $Y\in E_0 \oplus E_1$ and take into account that $\Lambda$ is orthogonal to $E_0 \oplus E_1$, we obtain the same formula as in the K\"ahler case:
$$
(A-\rho\, \Id) \nabla_Y\Lambda = - g(\Lambda,\Lambda) Y.
$$
The rest of the proof does not change and therefore, in the projective setting, Lemmas~\ref{lem:polynomialsB} and~\ref{lem:noJordanB} remain unchanged.
\end{remthm}

\begin{lem}
\label{lem:VnBB}  At every point, 
we have 
\begin{align}
\nabla\Lambda=\mu \, \Id +BA.\label{eq:VnBB}
\end{align}
for the constant $B$ and a function $\mu=B(\rho-1)$. Moreover,
\begin{align}
\nabla\mu=2B \Lambda^\flat.\label{eq:VnB2B}
\end{align}
\end{lem}

\begin{proof}  
We will first prove the lemma near a regular point.
We know that the eigenspaces of $A$ are invariant under $\nabla \Lambda$  and from Lemma~\ref{lem:polynomialsB}, 
we know the formula for the restriction of $\nabla\Lambda$ onto $E_0\oplus E_1$. Taking into account Lemma \ref{lem:noJordanB} 
we obtain
$$
\nabla\Lambda|_{E_0} = B(1-\rho)\rho ( -\rho\,\mathrm{Id})^{-1}= B(\rho-1)\mathrm{Id} =  (\mu \, \Id +BA)|_{E_0},
$$  
and
$$
\nabla\Lambda|_{E_1} = B(1-\rho)\rho ( \mathrm{Id} -\rho\,\mathrm{Id})^{-1}= B \rho\,\mathrm{Id} =  (\mu\, \Id +BA)|_{E_1},
$$  

Thus, it remains to verify the formula for the restriction of $\nabla\Lambda$ onto the two-dimensional $\rho$-eigenspace $E_\rho$. 

Notice that $\Lambda\in E_\rho$ and moreover $\Lambda$ is the tangent vector to the one-dimensional 
totally geodesic distribution $\mathcal U$.  In other words, we can consider $\Lambda$ as a tangent vector to a geodesic $\gamma(t)$.  Hence,
$(\nabla\Lambda)\Lambda=\nabla_\Lambda\Lambda=f\Lambda$ where
$f=\frac{1}{2}\frac{\Lambda(g(\Lambda,\Lambda))}{g(\Lambda,\Lambda)}$.

Using $\Lambda=\frac{1}{2}\gr\,\rho$ and the explicit formula \eqref{eq:gfinal} for $g$, we obtain
$$
f=B(2\rho-1).
$$
The other $\rho$-eigenvector is $J\Lambda$. Since $\Lambda$ and $J\Lambda$ commute, we get
$$
(\nabla\Lambda)J\Lambda = \nabla_{J\Lambda}\Lambda=\nabla_\Lambda J\Lambda= J\nabla_\Lambda \Lambda = 
J \bigl(B(2\rho-1) \Lambda\bigr) = B(2\rho-1) J\Lambda.
$$

Thus,
$$
\nabla\Lambda|_{E_\rho}=B(2\rho-1) \Id = (\mu \Id +BA)|_{E_\rho}.
$$

Since at regular points $TM$ decomposes as $TM=E_\rho\oplus E_0\oplus E_1$, we have verified equation \eqref{eq:VnBB} in a neighbourhood 
of every point of a dense and open subset of $M$ for a local constant $B$ and a locally defined function 
$\mu=B(\rho-1)$. Taking the derivative of this function (and using $\d\rho=2\Lambda^\flat$) shows that it 
satisfies \eqref{eq:VnB2B}. It was proven in \cite[\S 2.5]{FKMR} that having the equations \eqref{eq:VnBB} and \eqref{eq:VnB2B} 
satisfied in a neighbourhood of almost every point for locally defined constant $B$ and function $\mu$, the constant $B$ 
is the same for each such neighbourhood, hence, is globally defined, and therefore also $\mu$ is globally defined. 
This completes the proof of the lemma.  
\end{proof}

\begin{remthm}[for Theorem~\ref{specialcase2}]
The statement of Lemma~\ref{lem:VnBB} remains unchanged. 
The proof of Lemma~\ref{lem:VnBB} changes slightly, as the $\rho$-eigenspace is one-dimensional  (not two-dimensional as it was in the c-projective setting).  This makes the proof shorter as we do not need to consider the second eigenvector $J\Lambda$.  The projective analogue of \cite[\S 2.5]{FKMR} is  \cite[\S 2.3.4]{KioMat2010}.  
\end{remthm}

\medskip

Now  we  can proof Theorem~\ref{thm:yano_obata}.  We have shown that the existence of a non-affine c-projective vector field on
  a closed connected K\"ahler manifold $(M,g,J)$  of arbitrary signature implies  the existence of  a solution $(A,\Lambda,\mu)$ of the system
\begin{align}
\begin{array}{c}
\nabla_X A=X^\flat \otimes \Lambda+\Lambda^\flat\otimes X+(JX)^\flat \otimes J\Lambda+(J\Lambda)^\flat\otimes JX,\vspace{1mm}\\
\nabla\Lambda=\mu \, \Id +BA,\vspace{1mm}\\
\nabla \mu=2B \Lambda^\flat
\end{array}\nonumber
\end{align}
for a certain constant $B$. This solution is non-trivial in the sense that $\Lambda$ is not identically zero. 
By  \cite[Theorem 7]{FKMR},  $g$ is positive definite, up to multiplying the metric with a constant. 
On the other hand, Theorem~\ref{thm:yano_obata} was proven for positive signature in 
\cite{YanoObata}. This completes the proof of Theorem~\ref{thm:yano_obata}.

\begin{remthm}[for Theorem~\ref{specialcase2}]
\label{rem:last} 
The existence of $(L, \Lambda, \mu)$ satisfying \eqref{eq:main:proj} and   \eqref{eq:VnBB} immediately implies that the (non-constant)  function 
$\alpha=\operatorname{tr}(L)$ satisfies the equation 
\begin{equation}
\label{one}
\qquad \nabla^3 \alpha (X,Y,Z)- B\cdot \bigl( 2 (\nabla\alpha  \otimes g)(X,Y,Z)+(\nabla\alpha \otimes g)(Y,X,Z)+(\nabla\alpha \otimes g)(Z,X,Y)\bigr)=0, 
\end{equation}
see e.g. \cite[Corollary 4]{KioMat2010}. By \cite[Theorem 1]{Mounoud}, the metric  $-B  g$ is positive definite. This contradiction proves  Theorem~\ref{specialcase2}.
\end{remthm}

\section{Final step of the proof of Theorem~\ref{thm:lichnerowicz}:  the case of a non-constant Jordan block }
\label{sec:Jblock}

In \S\ref{sec:yano_obata}, the Lichnerowicz conjecture was proved under the additional assumption that the non-constant part of $A$ is diagonalisable.

In the case when the endomorphism $A$ has a non-trivial Jordan block with a non-constant eigenvalue,  the scheme of the proof remains essentially the same but we need to modify some steps accordingly. In fact,  the proof becomes much easier because the presence of a Jordan block, as we shall see below, immediately leads to unboundedness  of one special eigenvalue of the curvature operator.

The first steps of the proof do not use the algebraic type of $A$ and, therefore, they remain the same as in Sections \ref{ssec:PDE} and \ref{sssec:nocomplex}.   Namely, we may assume that the degree of mobility of $g$ equals 2 and $g$, $A$ and $v$ satisfy the equations 
\begin{equation}
\label{eq:LiederAg}
\begin{aligned}
 \mathcal L_v A &= -A^2 + A, \\
\mathcal L_v g &= - g A - (\tr A + \CC') g.
\end{aligned}
\end{equation}
These equations coincide with \eqref{eq:PDE} though we need to replace $\sum \rho_i$ by $\tr A$ as some non-constant eigenvalues may now have multiplicity $>1$,  see also our comment on \eqref{eq:PDE}  in Remark~\ref{rem:rem9} for Theorem~\ref{specialcase2}.

All the eigenvalues of $A$ are real  and satisfy the equation $\mathcal L_v \rho = -\rho^2 + \rho$,  in particular,  there are at most two constant eigenvalues,  0 and 1,  and, due to boundedness of the non-constant eigenvalues, they satisfy $0 \le \rho_i \le 1$.

The definition of the set  $M^0$ of regular points also remains the same but the algebraic type of $A$ changes. Recall that in general $M^0$ may consist of several connected components with different algebraic types, so we continue working with one of them (we still denote it by $M^0$ and refer to it as the set of regular points).
In the Lorentzian case, two sizes of Jordan blocks for $g$-selfadjoint endomorphisms are allowed, $2\times 2$ and $3\times 3$.  Thus, to complete the proof of Theorem~\ref{thm:lichnerowicz} we need to consider  two additional types of regular points $p\in M^0$.

Namely, below we assume that the endomorphism $A$ has $\ell$ non-constant distinct real eigenvalues $\rho_1,\dots, \rho_\ell$  and, possibly,  two constant eigenvalues $0$ and $1$ of multiplicity $m_0$ and $m_1$ respectively. The first eigenvalue $\rho_1$ has multiplicity 2 or 3,  and the endomorphism $A$ ``contains'' a single $2\times 2$ or resp. $3\times 3$ Jordan $\rho_1$-block.   On $M^0$, the eigenvalues $\rho_i$'s are ordered:
\begin{equation}
\label{eq:ineqforrho}
0 < \rho_2 <\dots <\rho_\ell < 1 \quad \mbox{and $\rho_1$ belongs to one of the intervals $(0,\rho_2)$,  $(\rho_2, \rho_3), \dots, (\rho_\ell, 1).$}
\end{equation}
Due to the existence of the projective vector field $v$,  the above conditions automatically imply that $\d\rho_i \ne 0$ on $M^0$.

In a neighbourhood of any point $p\in M^0$ we can now, following \cite{BM},  find a canonical coordinate system and reduce $g$ and $A$ to a normal form.  In general, this normal form contains arbitrary functions $F_i(\rho_i)$. The existence of a projective vector field $v$ on $M^0$,  satisfying \eqref{eq:LiederAg} allows us to reconstruct these functions (as well as the components of $v$) almost uniquely, i.e. up to a finite number of arbitrary constants of integration (cf. Proposition~\ref{prop:localclassgvFi}).  This ``reconstruction'' can be done by a straightforward computation as in Proposition~\ref{prop:localclassgvFi},  but since now we deal with a more complicated situation involving Jordan blocks,  we prefer to use the following general statement which explains how to split \eqref{eq:LiederAg}  in a block-wise manner.  This statement is a direct corollary of the splitting construction from \cite{BMgluing}.

\begin{thm}
\label{thm:PDEsplitting}
Let $(h,L)$ be a compatible pair on $M$ which splits into two blocks in the sense of \cite{BMgluing, BM}, i.e.,  
there exist a local coordinate system $x,y$ with $x=(x_1,\dots, x_{n_1})$ and $y=(y_1,\dots, y_{n_2})$ and compatible pairs
$\bigl(h_1(x),L_1(x)\bigr)$ and $\bigl(h_2(y),L_2(y)\bigr)$  such that
$$
h(x,y) = 
\begin{pmatrix} 
h_1(x) \chi_{L_2} (L_1(x)) &  0 \\ 0 & h_2(y) \chi_{L_1} (L_2(y))  
\end{pmatrix}, 
\quad
L(x,y) = 
\begin{pmatrix} 
L_1(x) & 0 \\ 0  & L_2(y) 
\end{pmatrix}
$$
where $\chi_{L_1}(\cdot)$ and $\chi_{L_2}(\cdot)$ denote the characteristic polynomials of the blocks $L_1$ and $L_2$ respectively.
Let  $v$ be a projective vector field for $g$  satisfying the equations
\begin{equation}
\label{eq:LvgL}
\begin{aligned}
\mathcal L_{v} L &= -L^2 + L , \\
\mathcal L_{v} h &= -hL - (\tr L + \CC) h.
\end{aligned}
\end{equation}
Then the vector field $v$ and  equations \eqref{eq:LvgL} also split as follows:
$$
v(x,y) = v_1(x) + v_2(y),
$$ 
where $v_i$ denote the natural projections of $v$ on the $x$-- and $y$--subspaces  and \eqref{eq:LvgL} is equivalent to 
$$
\begin{aligned}
\mathcal L_{v_1} L_1 = -L_1^2 + L_1 ,  \quad &\mathcal L_{v_1} h_1 = (n_2-1)h_1L_1 - (\tr L_1 + \CC + n_2) h_1, \\
\mathcal L_{v_2} L_2 = -L_2^2 + L_2 ,  \quad &\mathcal L_{v_2} h_2 = (n_1-1)h_2L_2 - (\tr L_2 + \CC + n_1) h_2. 
\end{aligned}
$$
\end{thm}

\begin{proof}
The first part $v(x,y) = v_1(x) + v_2(y)$ follows from the fact that $v$ preserves the invariant $x$- and $y$-subspaces of $L$ (see also \cite[Lemma 3]{BMgluing}).  
The latter can be seen from the equation $\mathcal L_v L=-L^2+L$. The formulas for $\mathcal L_{v_1} L_1$ and $\mathcal L_{v_2} L_2$ are straightforward.
After this we can differentiate $h$ as follows:
$$
\mathcal L_{v_1 + v_2} 
\begin{pmatrix} 
h_1 \chi_{L_2} (L_1) &  0 \\ 0 & \!\!\!\!\!\!\!\! h_2 \chi_{L_1} (L_2)  \end{pmatrix} =
\begin{pmatrix} \mathcal L_{v_1} \bigl(h_1 \chi_{L_2} (L_1)\bigr) {+}  h_1 (\mathcal D_{v_2} \chi_{L_2}) (L_1) &  \!\!\!\!\!\!\!\!\!\!\!\!\!\!\!\!\!\!\!\!\!\!\!\! 0 \\ 
0 &  \!\!\!\!\!\!\!\!\!\!\!\!\!\!\!\!\!\!\!\!\!\!\!\!\! \mathcal L_{v_2} \bigl(h_2 \chi_{L_1} (L_2)\bigr) {+}  h_2 (\mathcal D_{v_1} \chi_{L_1}) (L_2)   
\end{pmatrix}
$$
where $\mathcal D_{v_i} \chi_{L_i} $  means that we differentiate each coefficient of the characteristic polynomial $\chi_{L_i}$ along $v_i$ in the usual sense. 
We can rewrite the right-hand side as\footnote{To make some expressions shorter, we are using the formula $\mathcal L_v (\ln B) = B^{-1} \mathcal L_v B$  which, in the matrix case, holds true if $B$ and $\mathcal L_v B$ commute. In our case this condition is fulfilled for $B=L_1$ and, consequently, for $B=\chi_{L_2} (L_1)$ since
$\mathcal L_{v_1}L_1 = -L_1^2 + L_1$.}
$$
\mathcal L_{v_1} \bigl(h_1 \chi_{L_2} (L_1)\bigr) +  h_1 (\mathcal D_{v_2} \chi_{L_2}) (L_1)=
$$
$$
(\mathcal L_{v_1} h_1 ) \chi_{L_2} (L_1)+ h_1 \chi_{L_2} (L_1) \mathcal L_{v_1} \bigl(\ln (\chi_{L_2} (L_1))\bigr)  +  h_1 \chi_{L_2}(L_1) \mathcal D_{v_2}\bigl( \ln( \chi_{L_2} (L_1))\bigr).
$$
On the other hand, from \eqref{eq:LvgL} we know that this expression equals $h_1 \chi_{L_2} (L_1) (-L_1-\tr L{\cdot}\Id - \CC\cdot\Id)$.  Multiplying by $\chi_{L_2} (L_1)^{-1}$ we get
$$
\mathcal L_{v_1} h_1 = - h_1 \Big(  \mathcal L_{v_1} \bigl(\ln (\chi_{L_2} (L_1))\bigr) +  \mathcal D_{v_2}\bigl( \ln( \chi_{L_2} (L_1))\bigr) + L_1 +\tr L{\cdot}\Id +\CC{\cdot}\Id\Big).
$$
To evaluate this further, let $\lambda_1,\dots, \lambda_{n_2}$ denote the eigenvalues of $L_2$, i.e. the roots of $\chi_{L_2}$ (some of them may coincide). Then,  
$$
 \mathcal L_{v_1} \bigl(\ln (\chi_{L_2} (L_1))\bigr) =  \mathcal L_{v_1} \left( \sum_{i=1}^{n_2} \ln (L_1 - \lambda_i{\cdot}\Id)\right)=\sum_{i=1}^{n_2} (L_1 - \lambda_i{\cdot}\Id)^{-1}\mathcal L_{v_1} L_1 = \sum_{i=1}^{n_2} (L_1 - \lambda_i{\cdot}\Id)^{-1} (-L_1^2 + L_1)
$$
$$
\mathcal D_{v_2}\bigl( \ln( \chi_{L_2} (L_1))\bigr) = \mathcal D_{v_2} \left( \sum_{i=1}^{n_2} \ln (L_1 - \lambda_i{\cdot}\Id)\right)=
-\sum_{i=1}^{n_2} (L_1 - \lambda_i\cdot\Id)^{-1} \mathcal D_{v_2} \lambda_i = -\sum_{i=1}^{n_2} (L_1 - \lambda_i{\cdot}\Id)^{-1} (-\lambda_i^2 +\lambda_i)
$$

Hence,
$$
 \mathcal L_{v_1} \bigl(\ln (\chi_{L_2} (L_1))\bigr)  + \mathcal D_{v_2}\bigl( \ln( \chi_{L_2} (L_1))\bigr) =
 \sum_{i=1}^{n_2} (L_1 - \lambda_i\cdot\Id)^{-1} \bigl(-L_1^2 + L_1 +\lambda_i^2 \cdot\Id-\lambda_i\cdot\Id)\bigr)= 
 $$
$$
=\sum_{i=1}^{n_2} (\Id-\lambda_i\cdot\Id - L_1)= n_2\cdot\Id -\tr L_2\cdot\Id -n_2 L_1.
$$
Finally, we obtain 
$$
\mathcal L_{v_1} h_1 = -h_1 \left(  n_2{\cdot}\Id -\tr L_2{\cdot}\Id -n_2 L_1 + L_1 +\tr L{\cdot}\Id +\CC {\cdot}\Id  \right)=
-h_1 \bigl( (1-n_2)L_1  + (\tr L_1  + \CC +n_2)\Id \bigr)
$$
as we claimed.
\end{proof}

Let us apply the theorem by specifying the blocks into which we want to split:

\begin{cor}
\label{cor:PDEsplitting}
Let $(h,L)$ be a compatible pair that splits into two blocks of compatible pairs $(h_1,L_1)$, $(h_2,L_2)$
of dimensions $n_1$ and $n_2$ respectively as in Theorem~\ref{thm:PDEsplitting}.  
Suppose $v$ is a projective vector field such that $h$, $L$ and $v$
satisfy  \eqref{eq:LvgL}.

\begin{enumerate}

\item {\rm Case of a trivial $1\times 1$ Jordan $\rho$-block.} 
Let $(h_1,L_1)$ be given by
$$
h_1=\frac{1}{F(\rho)}\d\rho^2,\,\,\,L_1=\rho\,\frac{\D}{\D\rho} \otimes \d \rho,
$$
w.r.t. a coordinate $\rho$. Then $v_1=\rho(1-\rho)\frac{\D}{\D\rho}$ and 
\begin{equation}\label{eq:Sol1}
F=a(1-\rho)^{-\CC}\rho^{n+1+\CC}.
\end{equation}

\item {\rm Case of a  $2\times 2$ Jordan $\rho$-block.} 
Let $(h_1,L_1)$ be given by
\begin{equation}
\label{eq:hLJblock2}
h_1=\left(\begin{array}{cc}0&F(\rho)+x\\F(\rho)+x&0\end{array}\right)\mbox{ and }
L_1=\left(\begin{array}{cc}\rho&F(\rho)+x\\0&\rho\end{array}\right)
\end{equation}
w.r.t. coordinates $x,\rho$. Then $v_1=G(x,\rho)\frac{\D}{\D x}+\rho(1-\rho)\frac{\D}{\D\rho}$, where
$$
G(x,\rho)=\frac{1}{2}((n_2-1)\rho-1-\CC-n_2)x+G_1(\rho)
$$ 
and $F(\rho),G_1(\rho)$ satisfy the ODE system
\begin{equation}
\label{eq:ODE2}
\begin{array}{rl}
F'&=\frac{1}{\rho(1-\rho)}\left(\frac{1}{2}((n_2-1)\rho-1-\CC-n_2)F-G_1\right),\vspace{1mm}\\
G_1'&=\frac{1}{2}(n_2-1)F.
\end{array}
\end{equation}

\item {\rm Case of a $3\times 3$ Jordan $\rho$-block.} 
Let $(h_1,L_1)$ be given by
\begin{equation}
\label{eq:hLJblock3}
h_1=\left(\begin{array}{ccc}0&0&F(\rho)+2x_2\\0&1&x_1\\F(\rho)+2x_2&x_1&x_1^2\end{array}\right)\mbox{ and }
L_1=\left(\begin{array}{ccc}\rho&1&x_1\\0&\rho&F(\rho)+2x_2\\0&0&\rho\end{array}\right)
\end{equation}
in coordinates $x_1,x_2,\rho$. Then, 
$$
v_1=G(x_1,x_2,\rho)\frac{\D}{\D x_1}+H(x_2,\rho)\frac{\D}{\D x_2}+\rho(1-\rho)\frac{\D}{\D \rho},
$$
where
$$
G(x_1,x_2,\rho)=-\frac{1}{2}(\CC+n_2+2-n_2\rho)x_1+\frac{1}{2}n_2x_2+G_1(\rho)
$$
and
$$
H(x_2,\rho)=-\frac{1}{2}(\CC+n_2+(4-n_2)\rho)x_2+H_1(\rho)
$$
and $F(\rho),H_1(\rho),G_1(\rho)$ satisfy the ODE system 
\begin{equation}\label{eq:ODE3}
\begin{array}{rl}
F'&=-\frac{1}{\rho(1-\rho)}\left(\frac{1}{2}(\CC+n_2+(4-n_2)\rho)F+2H_1\right),\vspace{1mm}
\\
H_1'&=\frac{1}{2}(n_2-2)F-G_1,\vspace{1mm}
\\
G_1'&=0.
\end{array}
\end{equation}

\end{enumerate}
\end{cor}

\begin{rem}
\label{rem:rem15}
   Formulas \eqref{eq:hLJblock2} and \eqref{eq:hLJblock3}  are exactly the normal forms obtained in \cite{BM} for 
	compatible pairs $h$ and $L$ in the case when $L$ is conjugate to a $2\times 2$ resp.  $3\times 3$ Jordan block with a non-constant eigenvalue. 
Notice that these formulas are only meaningful at those points where $F+x\neq 0$ (resp. $F+2x_2\neq 0$).
\end{rem}

\begin{rem}
In part (1) of Corollary~\ref{cor:PDEsplitting}, we actually re-derived the formulas \eqref{eq:Fi} for the components of $v$  and the functions
$F_i$ parametrizing the metric 
$$
h=g|_{\mathcal L}=\sum_{i=1}^\ell \frac{\Delta_i}{F_i} \d \rho_i^2
$$
obtained from the metric $g$ by restricting it to leaves $\mathcal L$ of the distribution $\mathcal U$. 
\end{rem}

\begin{proof}[Proof of Corollary~\ref{cor:PDEsplitting}]
Theorem~\ref{thm:PDEsplitting} shows that $g_1$, $L_1$ and $v_1$ have to satisfy the equations
\begin{equation}
\label{eq:PDE11}
\mathcal L_{v_1} L_1 = -L_1^2 + L_1\mbox{ and }\mathcal L_{v_1} g_1 = (n_2-1)h_1L_1 - (\tr L_1 + \CC + n_2) g_1,
\end{equation}
where $n_2=n-n_1$ is the dimension of the block $(g_2,L_2)$ of $(g,L)$ complementary to $(g_1,L_1)$. 

\medskip
(1) From the first equation in \eqref{eq:PDE11}, it follows immediately that $v=\rho(1-\rho)\D_\rho$. This solves the
first equation identically. It is straightforward to check that the second equation in \eqref{eq:PDE11} with $n_2$ replaced by 
$n-1$ is equivalent to the ODE
$$
\frac{\d F}{F}=\frac{n+1+\CC-(n+1)\rho}{\rho(1-\rho)}\d\rho.
$$
The solution to this ODE is \eqref{eq:Sol1} as we claimed.

\medskip

(2) Since the first equation in \eqref{eq:PDE11} implies that $v_1$ preserves invariant subspaces of $L_1$, we can suppose that
$$
v_1=G(x,\rho)\D_x+\rho(1-\rho)\D_\rho
$$
for a certain function $G(x,\rho)$. Using this, together with the explicit formulas for $h_1$ and $L_1$, we see that the 
first equation in \eqref{eq:PDE11} is equivalent to 
\begin{equation}
\label{eq:L1}
\rho(1-\rho)F'+G-(F+x)\D_x G=0
\end{equation}
whilst the second equation in \eqref{eq:PDE11} is equivalent to the equations
\begin{equation}
\label{eq:g1}
\rho(1-\rho)F'+G+(F+x)(\D_x G+1+\CC+n_2+(1-n_2)\rho)=0
\end{equation}
and 
\begin{equation}
\label{eq:g2}
2\D_\rho G=(n_2-1)(F+x).
\end{equation}
Substracting \eqref{eq:L1} from \eqref{eq:g1} and dividing by $F+x$ yields
\begin{equation}
\label{eq:g3}
2\D_x G+1+\CC+n_2+(1-n_2)\rho=0.
\end{equation}
This shows that $G$ must be of the form
$$
G(x,\rho)=\frac{1}{2}((n_2-1)\rho-1-\CC-n_2)x+G_1(\rho)
$$
as we claimed. Inserting this into \eqref{eq:L1} (now equivalent to \eqref{eq:g1}) and \eqref{eq:g2}, 
we obtain \eqref{eq:ODE2} after rearranging terms.

\medskip
(3) Again, since $v_1$ preserves invariant subspaces of $L_1$, we have
$$
v_1=G(x_1,x_2,\rho)\frac{\D}{\D x_1}+H(x_2,\rho)\frac{\D}{\D x_2}+\rho(1-\rho)\frac{\D}{\D \rho}
$$
for certain functions $G(x_1,x_2,\rho)$ and $H(x_2,\rho)$. A straightforward calculation gives that
the first equation in \eqref{eq:PDE11} is equivalent to the equations
\begin{equation}
\label{eq:L3D1}
-1+2\rho +\D_{x_2}H-\D_{x_1}G=0,
\end{equation}
\begin{equation}
\label{eq:L3D2}
2x_2+F+G+\D_\rho H-(F+2x_2)\D_{x_2}G-x_1\D_{x_1}G=0
\end{equation}
and
\begin{equation}
\label{eq:L3D3}
2H+\rho(1-\rho)F'-(F+2x_2)\D_{x_2}H=0
\end{equation}
whilst the second equation in \eqref{eq:PDE11} is equivalent to the equations
\begin{equation}
\label{eq:g3D1}
2H+\rho(1-\rho)F'+(F+2x_2)(1+\CC+n_2+(2-n_2)\rho+\D_{x_1}G)=0,
\end{equation}
\begin{equation}
\label{eq:g3D2}
\CC+n_2+(4-n_2)\rho+2\D_{x_2}H=0,
\end{equation}
\begin{equation}
\label{eq:g3D3}
(1+\CC+n_2+(2-n_2)\rho+\D_{x_2}H) x_1+G+\D_\rho H-(F+2x_2)(n_2-1-\D_{x_2}G)=0
\end{equation}
and
\begin{equation}
\label{eq:g3D4}
(2+\CC+n_2-n_2\rho)x_1^2+2(G+\D_\rho H)x_1-2(F+2x_2)((n_2-1)x_1-\D_\rho G)=0.
\end{equation}
Subtracting \eqref{eq:L3D3} from \eqref{eq:g3D1} and dividing by $F+2x_2$, we obtain
$$
1+\CC+n_2+(2-n_2)\rho+\D_{x_1}G+\D_{x_2}H=0.
$$
Hence, using \eqref{eq:L3D1},
$$
\D_{x_1}G=\frac{1}{2}(n_2\rho-2-\CC-n_2)\mbox{ and }\D_{x_2}H=\frac{1}{2}((n_2-4)\rho-\CC-n_2)
$$
or, in other words,
$$
G(x_1,x_2,\rho)=\frac{1}{2}(n_2\rho-2-\CC-n_2)x_1+\tilde G(x_2,\rho)\mbox{ and }
H(x_2,\rho)=\frac{1}{2}((n_2-4)\rho-\CC-n_2)x_2+H_1(\rho)
$$
for certain functions $\tilde G(x_2,\rho),H_1(\rho)$. Inserting this back into our PDE system \eqref{eq:L3D1}--\eqref{eq:g3D4}
we obtain that \eqref{eq:PDE11} is equivalent to the equations
\begin{equation}
\label{eq:L3D4}
\frac{1}{2}n_2 x_2+F+\tilde G+H_1'-(F+2x_2)\D_{x_2}\tilde G=0,
\end{equation}
\begin{equation}
\label{eq:L3D5}
(\CC+n_2+(4-n_2)\rho)F+4H_1+2\rho(1-\rho)F'=0,
\end{equation}
\begin{equation}
\label{eq:g3D5}
-\frac{3}{2}n_2x_2+\tilde G+H_1'+(F+2x_2)\D_{x_2}\tilde G+(1-n_2)F=0
\end{equation}
and
\begin{equation}
\label{eq:g3D6}
(-n_2x_2-(n_2-2)F+2\tilde G+2H_1')x_1+2(F+2x_2)\D_{\rho}\tilde G=0.
\end{equation}
Substracting \eqref{eq:L3D4} from \eqref{eq:g3D5} and dividing by $F+2x_2$ gives
$-n_2+2\D_{x_2}\tilde G=0$ and we see that
$$
\tilde G(x_2,\rho)=\frac{1}{2}n_2x_2+G_1(\rho).
$$
Inserting this formula for $\tilde G$ into \eqref{eq:L3D4}--\eqref{eq:g3D6}, we obtain that \eqref{eq:PDE11} 
is equivalent to the equations
$$
\frac{1}{2}(2-n_2)F+G_1+H_1'=0,
$$
$$
(\CC+n_2+(4-n_2)\rho)F+4H_1+2\rho(1-\rho)F'=0,
$$
$$
4x_2 G_1'+((2-n_2)x_1+2G_1')F+2x_1(G_1+H_1')=0.
$$
The first two of these equations give the first two equations in \eqref{eq:ODE3}. Multiplying
the first equation by $2x_1$ and subtracting it from the third gives
$2(F+2x_2)G_1'=0$, hence, $G_1'=0$ as we claimed.
\end{proof}

Now we are ready to describe the local structure of $g$,  $A$ and $v$ in the case when $A$  ``contains'' a $2\times 2$ or $3\times 3$ Jordan block.

\begin{prop}
\label{prop:canformJ2} 

\begin{enumerate}

\item Let $p\in  M^0$ be a regular point and $A$ contain a $2\times 2$ Jordan block with a non-constant eigenvalue $\rho=\rho_1$. Then 
in a neighborhood of $p$ there exists a local coordinate system $x, \ \rho_1,\dots,\rho_\ell, \ y_1,\dots, y_N$,  $N=m_0+m_1$,  in which $g$ and $A$  take the following form: 
\begin{equation}
\label{eq:gAJblock}
A = 
\begin{pmatrix}
L(x,\vec\rho) & 0\\0 & A_{\mathrm{c}}(y) \end{pmatrix} \quad \mbox{and} \quad g = \begin{pmatrix} h(x,\vec \rho) & 0 \\0 & g_{\mathrm{c}} (y)\cdot\chi_L(A_{\mathrm{c}}(y))
\end{pmatrix}
\end{equation}
where
\begin{equation}
\label{eq:LJord2}
L \!=\!  
\begin{pmatrix}
L_1(x,\rho_1)& 0   & \dots  &0      \\
            0&\!\!\!\! \rho_2    & \dots  &0    \\
     \vdots       & \vdots   & \ddots  &    \vdots\\
          0 &  0  & \dots  & \rho_\ell   
\end{pmatrix}, 
\quad 
h\!=\!  
\begin{pmatrix} 
h_1 (x,\rho_1) \cdot \Delta_1 (L_1) &0 &\dots &0 \\
0& \!\!\!\!\! F_2^{-1}(\rho_2) \cdot \Delta_2  & \dots &0 \\
\vdots&\vdots &  \!\!\!\!\! \ddots & \vdots\\
0& 0&\dots &  \!\!\!\!\! F_\ell^{-1}(\rho_\ell) \cdot \Delta_\ell  
\end{pmatrix}, 
\end{equation}
and the ingredients in these matrices are as follows:

\begin{itemize}

\item $L_1$ and $h_1$ are defined by \eqref{eq:hLJblock2} with $\rho{=}\rho_1$, 

\item $\Delta_1(\cdot)$ is the polynomial of the form $\Delta_1(t)=\Pi_{j=2}^\ell (t-\rho_j)$, 

\item $F_i(\rho_i) {=}    a_i (1-\rho_i)^{-\CC} \rho_i^{\ell+2+\CC}$,  

\item $\Delta_i {=} (\rho_i-\rho_1)^2 \Pi_{j=2, j\ne i}^\ell (\rho_i - \rho_j)$, $i=2,\dots,\ell$,  

\item $A_{\mathrm c}(y)$ is selfadjoint and parallel w.r.t. $g_{\mathrm c}(y)$.

\end{itemize}

 Furthermore,
$$
v = G(x,\rho_1) \frac{\partial}{\partial x} + \sum_{i=1}^\ell  \rho_i(1-\rho_i)\frac{\partial}{\partial \rho_i}  + \dots
$$
with $G(x,\rho_1)$ as in Corollary~\ref{cor:PDEsplitting} (2), with $n_2=\ell-1$, $\rho_1=\rho$.

\item Similarly, let $p\in  M^0$ be a regular point and $A$ contain a $3\times 3$ Jordan block with a non-constant eigenvalue $\rho=\rho_1$. 
Then in a neighbourhood of  $p\in M^0$ we can choose local coordinates $x_1, x_2, \ \rho_1, \dots, \rho_\ell, \ y_1\dots,y_N$ 
such that $A$ and $g$ are given by \eqref{eq:gAJblock}, \eqref{eq:LJord2}  where $x=(x_1, x_2)$ and the other ingredients  are as follows:

\begin{itemize}

\item the $3\times 3$ blocks 
$L_1(x_1,x_2,\rho_1)$  and $h_1(x_1,x_2,\rho_1)$ are defined by \eqref{eq:hLJblock3} with $\rho_1=\rho$,

\item $\Delta_1(\cdot)$ is the polynomial of the form $\Delta_1(t)=\Pi_{j=2}^\ell (t-\rho_j)$, 

\item $F_i(\rho_i) {=}    a_i (1-\rho_i)^{-\CC} \rho_i^{\ell+3+\CC}$,  $i=2,\dots,\ell$, 

\item $\Delta_i {=} (\rho_i-\rho_1)^3 \Pi_{j=2, j\ne i}^\ell (\rho_i - \rho_j)$, $i=2,\dots,\ell$,

\item $A_{\mathrm c}(y)$ is selfadjoint and parallel w.r.t. $g_{\mathrm c}(y)$.

\end{itemize}

 Furthermore, 
$$
v=G(x_1, x_2,\rho_1)\frac{\partial}{\partial x_1} + H(x_2,\rho_1) \frac{\partial}{\partial x_2}  + \sum_{i=1}^\ell  \rho_i(1-\rho_i)\frac{\partial}{\partial \rho_i}  + \dots
$$
with $G(x_1, x_2,\rho_1)$ and $H(x_2,\rho)$ as in Corollary~\ref{cor:PDEsplitting} (3), with $n_2=\ell-1$, $\rho_1=\rho$.

\end{enumerate}
\end{prop}


\begin{proof}    
The formulas \eqref{eq:gAJblock}  and \eqref{eq:LJord2}  (with $F(\rho), F_2(\rho_2),\dots, F_\ell(\rho_\ell)$ 
being arbitrary functional parameters) are  just a reformulation of the main result of \cite{BM} in the case when  $A$ has the algebraic type described above.

In our situation we have, in addition,  a projective vector field $v$ satisfying \eqref{eq:LiederAg}.   Consider the natural 
decomposition of $v$ that corresponds to the splitting \eqref{eq:gAJblock}  of $g,A$ into ``constant'' and ``non-constant'' blocks:  $v=v_{\mathrm{nc}} (x,\vec\rho) + v_{\mathrm{c}}(y)$.

It is easy to see  (cf. \eqref{Lvrho}, \eqref{eq:PDE1}) that  \eqref{eq:LiederAg}  can be rewritten for the non-constant
 block without any change, i.e.,
 $$
  \mathcal L_{v_{\mathrm{nc}}} L  = -L^2 + L,  \quad 
\mathcal L_{v_{\mathrm{nc}}} h = - h L   - (\tr L + \CC) h.
$$   
Here  $\tr A + \CC' = \tr L  + \CC$ and $\CC=\CC'+m_1$ where $m_1$ is the multiplicity of the constant eigenvalue $1$ or,
 which is the same, $m_1=\tr A_{\mathrm{c}}$.

After this remark, Proposition \ref{prop:canformJ2} follows immediately by applying Theorem~\ref{thm:PDEsplitting} and Corollary~\ref{cor:PDEsplitting}  (for $h$, $L$ and 
$v_{\mathrm{nc}}$ but not $v$\,!) to reconstruct the functions $F(\rho), F_2(\rho_2),\dots, F_\ell(\rho_\ell)$ as well 
as the components of $v_{\mathrm{nc}}$   (the components of $v_{\mathrm{c}}(y)$ are not important for our purposes and we ignore them,  
in Proposition~\ref{prop:canformJ2} they are denoted by $\dots$). \end{proof}

Partitioning local coordinates into two groups $x,\rho_1,\dots,\rho_\ell$ and $y_1,\dots, y_N$  determines 
two natural integrable distributions $\mathcal U$ and $\mathcal U^\bot$ on $M^0$ similar to those from \S\ref{sssec:explicit}.  
All geometric properties of the corresponding foliations listed in Proposition~\ref{prop:distributions}  still hold  with one little 
amendment that $\dim \mathcal U=\ell+1$  or $\ell + 2$ (but not $\ell$ as before) so that now we should think of $x$ as 
an additional coordinate to $\rho_i$'s.

The next statement is an analogue of Proposition~\ref{prop:global}.
Consider the domain $U\subset \R^{\ell+1}(x,\rho_1, \dots, \rho_\ell)$ in the case of a $2\times 2$ Jordan block (resp.  $U\subset \R^{\ell+2} (x_1,x_2,\rho_1,\dots,\rho_\ell)$ in the case of a $3\times 3$ Jordan block) on which the above local formulas  \eqref{eq:LJord2}  for $h$ are naturally defined.  More precisely, $U$ is defined by the inequalities \eqref{eq:ineqforrho} for the $\rho_i$'s  and the additional coordinates 
 $x$ and $x_2$  satisfy $F(\rho_1) +x \ne 0$ for a $2\times 2$-block and resp. $F_1(\rho_1) + 2x_2 \ne 0$ for a $3\times 3$ block, see Remark~\ref{rem:rem15}.

\begin{prop}
There is a natural isometric immersion  $\phi : U \to M$ (as a leaf of the totally geodesic foliation determined by $\mathcal U$). In other words, the above formulas \eqref{eq:gAJblock}  and \eqref{eq:LJord2} have global meaning on $M$ for all admissible values of coordinates.
\end{prop}

\begin{proof} 
The proof is almost identical to that of Proposition~\ref{prop:global}.    We start with a certain point $p\in M^0$ and  locally identify the leaf of $\mathcal U$ through $p$ with $U$ by using a canonical coordinate system in its neighbourhood constructed in Proposition~\ref{prop:canformJ2}.

After this we use prolongation along a path as in Proposition~\ref{prop:global}.  The only thing we need to explain is why such a prolongation is always possible. More specifically,  we need to show that  the limit point of the curve $\phi( a(t))$ as $t\to T_0$ (we use the same notation as in Proposition~\ref{prop:global}) cannot leave $M^0$, the set of regular points.

Since,  $\phi$ preserves the eigenvalues of $L$,  the multiplicities of the eigenvalues remain unchanged and the inequalities \eqref{eq:ineqforrho} hold at the limit point.  The condition $\d\rho_i =0$ is fulfilled automatically and we only need to check that the Jordan block ``survives'' at the limit point.  A priori under continuous deformations the Jordan block may split into smaller blocks and we need to show that this event may not happen  under our assumptions.

To prove this fact we use the following algebraic lemma.

\begin{lem}
\label{lem:Jreg}
Let $h$ be a  non-degenerate bilinear form of Lorentzian signature and $L$ be an $h$-selfadjoint endomorphism.  Assume that $L$ has a single real eigenvalue $\rho$ and $e_1$ is a $\rho$-eigenvector of $L$.   Then in dimension $2$ and $3$ we have respectively:

\begin{enumerate}

\item For any canonical basis $e_1, e_2$  {\rm(}i.e., such that $h=\begin{pmatrix}   0 & 1\\  1& 0  \end{pmatrix}$ or, equivalently, $h(e_i, e_j)=\delta_{i, 3-j}${\rm)}, the matrix of $L$ has the form
$$
L=\begin{pmatrix}  
\rho & \alpha \\ 0 & \rho
\end{pmatrix}.
$$ 
Moreover, $\alpha$  does not depend on the choice of $e_2$,  and can be computed from the following formula
$\alpha = \mathrm{vol}_h \bigl(e_2,  (L-\rho\,\mathrm{Id}) e_2\bigr)$.

\item For any canonical basis $e_1, e_2, e_3$  {\rm(}i.e.,  such that $h=\begin{pmatrix}  0 & 0 & 1\\  0 & 1 & 0 \\ 1& 0 & 0 \end{pmatrix}$ or, equivalently, $h(e_i, e_j)=\delta_{i, 4-j}${\rm)}, the matrix of $L$ has the form
$$
L=\begin{pmatrix}  
\rho & \alpha & \beta \\ 0 & \rho & \alpha \\ 0 & 0 & \rho
\end{pmatrix}.
$$ 
Moreover, $\alpha$  does not depend on the choice of $e_2$ and $e_3$,  and can be computed from the following formula
$\alpha = \mathrm{vol}_h \bigl(e_3,  (L-\rho\,\mathrm{Id}) e_3, (L-\rho\,\mathrm{Id}) ^2 e_3\bigr)^{\frac{1}{3}}$.

\end{enumerate}
\end{lem}

\begin{proof}
The proof is straightforward and we only give some comments for $\dim = 3$.
The first statement follows immediately from two facts: 

\begin{enumerate}

\item $L$ is $h$-selfadjoint and therefore the matrices of $L$ and $h$ satisfy $L^\top h = h L$;  

\item the first column of $L$ is $(\rho, 0, 0)^\top$. 

\end{enumerate} 

The formula for $\alpha$ is obvious in the basis $e_1, e_2, e_3$. We now check that this formula is independent on the choice of $e_2$ and $e_3$.  Let $e_1, e'_2, e'_3$ be another canonical basis. Then $e_3'=a_1 e_1 + a_2 e_2 + a_3 e_3$,  but  $h(e_1, e_3)= h(e_1, e_3')=1 $ immediately implies that $a_3=1$.  Hence, using the explicit formula for $L$, we can easily conclude that the additional terms $a_1 e_1 + a_2 e_2$ do not contribute. Indeed,
$$
\begin{aligned}
\mathrm{vol}_h(e'_3,    &(L-\rho\, \mathrm{Id}) e'_3, (L-\rho\, \mathrm{Id}) ^2 e'_3)= \\
&\mathrm{vol}_h(e_3 + a_1 e_1 + a_2 e_2,    (L-\rho\, \mathrm{Id}) e_3 + a_2 (L-\rho\, \mathrm{Id}) e_2 , (L-\rho\, \mathrm{Id}) ^2 e_3)=\\
& \mathrm{vol}_h(e_3 + a_1 e_1 + a_2 e_2, \beta e_1 + \alpha e_2,  \alpha^2 e_1)=\alpha^3  \mathrm{vol}_g(e_3,e_2,e_1)=\alpha^3.
\end{aligned}
$$

\end{proof}

 This lemma gives us a simple method to recognise if $L$ has a Jordan block of maximal size or not.

To verify that the limit point $p$ of the curve $\phi(a(t))$, as $t\to T_0$, is regular, i.e., the Jordan block ``survives'',  we will use the fact that the eigenvalue $\rho_1$ of the Jordan block is a smooth function on $M^0$.  Moreover, the vector field $e_1=\gr\, \rho_1$ does not vanish and is an eigenvector of the $\rho_1$-block.  Notice that these conditions hold not only on $M^0$,  but also on a slightly bigger set $\widetilde M^0$   ($M^0\subset \widetilde M^0$) which can be characterised by the property that the multiplicities of eigenvalues are fixed but the algebraic type of $L$ is allowed to change, namely, the Jordan $\rho_1$-block may split into smaller $\rho_1$-blocks.  Notice that the natural splitting into blocks corresponding to the eigenvalues of $g$ makes sense on $\widetilde M^0$,  so we can work with each block separately.    An important additional fact, we are going to use, is that $\phi$, by construction, preserves $\gr\, \rho_1$.     

We use Lemma~\ref{lem:Jreg}  to verify that the parameter $\alpha$ in the matrix of  $L(p)$ does not vanish.   Since both $L$ and $e_1$ are smooth, we have by continuity $\alpha=\lim_{t\to T_0} \alpha(t) $, where $\alpha(t)$ is computed at the point $\phi(a(t))$  (w.r.t. to $e_1=\gr\, \rho$).  But since $\phi$ is an isometry whenever it  is well defined, then the limit can be computed  on $U$. Since all the points of $U$ are regular by construction, we have $\lim_{t\to T_0} \alpha(t)\ne 0$ as required.

Thus the limit point of $\phi (a(t))$, $t\to T_0$, is regular  and the further continuation of $\phi$  is possible which completes the proof. \end{proof}

Now to prove that Jordan blocks with non-constant eigenvalues cannot appear on compact manifolds,  we compute one special eigenvalue of the curvature operator of the metric $g$  given by the formulas 
from Proposition~\ref{prop:canformJ2}.  For this computation we use the following real analogue of Proposition~\ref{prop:eigvalcurv}, which can be proved in a similar way.

\begin{prop} [\cite{BFom70}]
Let $g$ and $L$ be compatible in the projective sense and $\Lambda$ be as in \eqref{eq:main:proj}.  Let $\nabla\Lambda= f (L)$ at some point $p\in M$, where $f(\cdot)$ is a polynomial (or, more generally, an analytic function) and suppose $L(p)$ has a non-trivial Jordan $\rho$-block.  Then one of the eigenvalues of the curvature operator of $g$ at the point $p$ takes the form 
$$
f'(\rho).
$$ 
\end{prop}

This number can be computed for our metric $g$ (equivalently, for $h$ given by \eqref{eq:LJord2})
A straightforward calculation shows the following:

\begin{prop}
\begin{enumerate}

\item For a $2\times 2$ Jordan $\rho_1$-block, we have
$$
f'(\rho_1)=-\frac{F_1'(\rho_1)}{(F_1(\rho_1)+x)^3\prod_{i\geq 2}(\rho_1-\rho_i)}
+\sum_{i\geq 2}\frac{F_i(\rho_i)}{4(\rho_i-\rho_1)^4\prod_{j\notin \{1,i\}}(\rho_i-\rho_j)}.
$$

\item For a $3\times 3$ Jordan $\rho_1$-block, we have
$$
f'(\rho_1)=-\frac{3}{4(F_1(\rho_1)+2x_2)^2\prod_{i\geq 2}(\rho_1-\rho_i)}
+\sum_{i\geq 2}\frac{F_i(\rho_i)}{4(\rho_i-\rho_1)^5 \prod_{j\notin \{1,i\}}(\rho_i-\rho_j)}.
$$

\end{enumerate}
\end{prop}

These formulas immediately imply that the quantity $f'(\rho_1)$ (which is some special eigenvalue of the curvature operator of $g$) is unbounded on $M^0$. Indeed,  $x$ and $x_2$ may vary independently of the other coordinates and, in particular, we may fix the values of all $\rho_i$'s  and then vary $x$ (resp.  $x_2$) so that 
$F_1(\rho_1)+x$  (resp. $F_1(\rho_1)+2x_2$) tends to $0$  and therefore $f'(\rho_1) \to \infty$,  which is impossible due to compactness of $M$.  Thus,   Jordan blocks with non-constant eigenvalues may not occur in our situation and this conclusion completes the proof of Theorem~\ref{thm:lichnerowicz}.

\appendix
\section{Eigenvalues of the curvature operator}
\label{app}

In what follows, we consider a real vector space $V$ with a complex structure $J$ and an inner product $g$ 
(not necessarily positive definite) such that $g(Ju,Jv)=g(u,v)$. Such a triple $(V,g,J)$ will be 
referred to as a \emph{pseudo-hermitian vector space}. We use the symbol 
$$
\mathfrak{u}(g,J)=\{X\in \mathfrak{gl}(V):[X,J]=0\mbox{ and }g(Xu,v)=-g(u,Xv)\}
$$
to denote the space (Lie algebra) of skew-hermitian endomorphisms on $V$.

Let us first reformulate the integrability condition for equation \eqref{eq:main} in a way adapted to the Lie theory. 
Recall that  the Riemann curvature operator (at a point $x\in M$) can be understood as a map $R:T_x M\otimes T_x M \to \mathfrak{so} (g)$, $R(u,v) =  \nabla_u\nabla_v -\nabla_v\nabla_u -\nabla_{[u,v]}$.  Taking into account the fact that we are dealing with a K\"ahler manifold and using the symmetries of the curvature tensor of a K\"ahler metric, we can also think of $R$ as an operator defined on the unitary Lie algebra (we still use the same notation)
$$
R: \mathfrak{u}(g,J) \to  \mathfrak{u}(g,J)
$$
by setting $R(u,v) = \frac{1}{4} R(u \wedge_J v)$, where
\begin{align}
\label{eq:basic0}
u\wedge_J v = u^\flat\otimes v- v^\flat\otimes u+(Ju)^\flat\otimes Jv-(Jv)^\flat\otimes Ju \in \mathfrak{u}(g,J)
\end{align}
and $u^\flat=g(u,\cdot)$ denotes the metric dual of $u$.

\begin{lem}
\label{lem:RicciId}
Let $(M,g,J)$ be a  K\"ahler manifold of arbitrary signature, $A\in \mathcal{A}(g,J)$ be a hermitian solution of \eqref{eq:main} and $\Lambda=\frac{1}{4}\gr(\tr A)$. 
Then the curvature operator 
$R:\mathfrak{u}(g,J)\rightarrow \mathfrak{u}(g,J)$ satisfies the relation
\begin{align}
[R(X),A]=4[X,\nabla\Lambda]\mbox{ for all }X\in \mathfrak{u}(g,J).\label{eq:RicciId}
\end{align}
\end{lem}

\begin{proof}
The Ricci identity applied to an arbitrary field of endomorphisms $A$ reads
$$
\nabla_u\nabla_v A-\nabla_v\nabla_u A-\nabla_{[u,v]}A=[R(u,v),A]
$$
for any vector fields $u,v$. Let now $A\in \mathcal{A}(g,J)$ be a hermitian solution of \eqref{eq:main}. 
Since $\nabla\Lambda$ is $g$-selfadjoint and $J$-linear, i.e., hermitian too,  we have
$$
\nabla_u\nabla_v A-\nabla_v\nabla_u A-\nabla_{[u,v]}A
$$
$$
=v^\flat\otimes \nabla_u\Lambda+(\nabla_u\Lambda)^\flat\otimes v-u^\flat\otimes \nabla_v\Lambda-(\nabla_v\Lambda)^\flat\otimes u
$$
$$
+(Jv)^\flat\otimes \nabla_{Ju}\Lambda+(\nabla_{Ju}\Lambda)^\flat\otimes Jv-(Ju)^\flat\otimes \nabla_{Jv}\Lambda-(\nabla_{Jv}\Lambda)^\flat\otimes Ju
$$
$$
=[X,\nabla\Lambda],
$$
where  $X=u\wedge_J v $. This proves formula \eqref{eq:RicciId} for elements $X\in\mathfrak{u}(g,J)$ of the form  $u\wedge_J v$
and the claim follows from the fact that all skew-hermitian endomorphisms are sums of such elements.
\end{proof}

\begin{rem}
\label{rem:commute}
If formula \eqref{eq:RicciId} holds for an operator $R:\mathfrak{u}(g,J)\rightarrow \mathfrak{u}(g,J)$ 
and hermitian endomorphisms $A$ and $\nabla \Lambda$ then, in fact, $\nabla \Lambda$ can be presented in the form $\nabla \Lambda = p(A)$ for some polynomial $p(\cdot)$  with real coefficients.   To show this,  take an arbitrary  $J$-complex matrix $Y$  and consider the following algebraic relations:
$$
\tr \bigl( X \cdot [\nabla \Lambda, Y] \bigr) = \tr \bigl( Y \cdot [X, \nabla \Lambda]  \bigr) = \frac{1}{4}\tr \bigl( Y \cdot [R(X), A]  \bigr) = \frac{1}{4} \tr \bigl( R(X)\cdot [A,Y]  \bigr),
$$
 where $X\in \mathfrak u(g,J)$ and  $\tr$ denotes the complex trace. Since $\mathfrak u(g,J)$ spans $\mathfrak {gl} (T_xM,J)$ in the complex sense, we conclude that $[\nabla \Lambda, Y]=0$ for any $Y$ commuting with $A$. It is a well-known algebraic fact that in this case $\nabla \Lambda$ can be written as a polynomial of $A$. Moreover, as both $A$ and $\nabla \Lambda$ are hermitian,  this polynomial must be real, i.e. with real coefficients.
\end{rem}

Proposition~\ref{prop:eigvalcurv} below together with formula \eqref{eq:RicciId} allows us to calculate 
eigenvalues of the curvature operator in terms of the eigenvalues of $A$ and $\nabla\Lambda$. 
For the main concepts of the proof of this proposition and for the relation to sectional operators 
in the theory of integrable systems compare also with \cite[\S 3]{Fubini} and \cite{BFom70, BolHol}.

\begin{prop}
\label{prop:eigvalcurv}
Let $(V,g,J)$ be a pseudo-hermitian vector space and let $A:V\rightarrow V$ be a hermitian 
 endomorphism. Suppose an operator $R:\mathfrak{u}(g,J)\rightarrow \mathfrak{u}(g,J)$ satisfies 
\begin{align}
[R(X),A]=[X,B]\mbox{ for all }X\in \mathfrak{u}(g,J),\label{eq:LinAlg}
\end{align}
where $B=p(A)$ and $p(\cdot )$ is a polynomial with real coefficients. Then we have the following:

\begin{enumerate}

\item For all real eigenvalues $\lambda_i\neq \lambda_j$ of $A$, 
$$ 
\frac{p(\lambda_i)-p(\lambda_j)}{\lambda_i-\lambda_j}
$$
is an eigenvalue of $R$. 

\item If $A$ has a non-trivial $\lambda_i$-Jordan block, $\lambda_i\in \R$, then $p'(\lambda_i)$ is an eigenvalue of $R$.

\end{enumerate}
\end{prop}

\begin{rem}
\label{rem:eigvalcurv}  
The first item of Proposition \ref{prop:eigvalcurv} can be understood in a slightly different way. Notice that  $B=p(A)$ implies that each eigenvector of $A$  with an eigenvalue $\lambda_i$ is, at the same time, an eigenvector of $B$ with the eigenvalue $m_i= p(\lambda_i)$. Hence the formula for the eigenvalue of $R$ from item (1) can be rewritten as $\frac{m_i - m_j}{\lambda_i - \lambda_j}$  so that we do not actually need to find $p(\cdot)$ explicitly; it is sufficient to know the eigenvalues  $m_i$ of $B$ corresponding to $\lambda_i$.

The second item of Proposition \ref{prop:eigvalcurv} can also be modified by using the following simple fact from Linear Algebra.   Let $\lambda$ be an eigenvalue of an endomorphism $A$ having a non-trivial     
$\lambda$-Jordan block. Let $p(\cdot)$ and $q(\cdot)$ be two polynomials (or even more generally,  analytic functions) such that $p(A)=q(A)$, then
$p(\lambda)=q(\lambda)$ and $p'(\lambda)=q'(\lambda)$.  It follows from this statement that the polynomial $p$  in the second item of Proposition \ref{prop:eigvalcurv} can be replaced by any other function $q$ satisfying  $p(A|_{V_{\lambda_i}})= q(A|_{V_{\lambda_i}})$, where $V_{\lambda_i}$ denotes the generalised $\lambda_i$-eigenspace of $A$.
\end{rem}

\begin{proof}[Proof of Proposition~\ref{prop:eigvalcurv}]
We start with some general considerations regarding formula \eqref{eq:LinAlg}. We view this formula as an 
equation on $R$ for fixed $A$ and $B=p(A)$. Suppose $R_1,R_2:\mathfrak{u}(g,J)\rightarrow \mathfrak{u}(g,J)$ 
are two solutions of \eqref{eq:LinAlg}. Then, 
$$
[R_1(X)-R_2(X),A]=0\mbox{ for all }X\in \mathfrak{u}(g,J),
$$
that is, $R_1-R_2$ takes values in the Lie algebra
$$
\mathfrak{g}_A=\{X\in \mathfrak{u}(g,J):[X,A]=0\},
$$
the centraliser of $A$ in $\mathfrak{u}(g,J)$. Thus, any solution $R$ of \eqref{eq:LinAlg} is unique up to 
adding an operator $ \mathfrak{u}(g,J)\rightarrow \mathfrak{g}_A$. Moreover, an operator $R$ satisfying 
\eqref{eq:LinAlg} preserves the centraliser $\mathfrak{g}_A$. Indeed, since $B=p(A)$ is a polynomial in $A$, 
we have $[X,B]=0$ for all $X\in  \mathfrak{g}_A$. Then \eqref{eq:LinAlg} implies $[R(X),A]=0$ for 
$X\in  \mathfrak{g}_A,$ showing 
$$
R(\mathfrak{g}_A)\subseteq \mathfrak{g}_A
$$
as we claimed. For any solution $R$ of \eqref{eq:LinAlg}, we may therefore consider the induced operator
$$
\tilde R:\mathfrak{u}(g,J)/\mathfrak{g}_A\rightarrow \mathfrak{u}(g,J)/\mathfrak{g}_A
$$
on the quotient space. It is a general fact that eigenvalues of the quotient operator $\tilde R$ 
are eigenvalues of the original operator $R$. On the other hand, we have just seen  that the quotient 
map $\tilde R$ is the same for all solutions $R$ of \eqref{eq:LinAlg}. We will use these facts by working 
with the quotient map $\tilde R_0$ coming from a special solution $R_0$ of \eqref{eq:LinAlg} defined by 
$$
R_0=\frac{\d}{\d t}p(A+tX)|_{t=0}.
$$
This is indeed a solution of \eqref{eq:LinAlg} as follows immediately from differentiating the identity 
$[p(A+tX),A+tX]=0$ at $t=0$. By definition, if $p(t)=\sum_{k=0}^m a_k t^k$, then
$$
R_0(X)=\sum_{k=1}^m a_k\sum_{p+q=k-1}A^p X A^q.
$$ 

Hence for a generating element $u\wedge_J v$, we obtain
\begin{align}
R_0(u\wedge_J v)=\sum_{k=1}^m a_k\sum_{p+q=k-1} A^p u\wedge_J A^q v. \label{eq:R0ongenerators}
\end{align}
We are now in the position to prove Proposition~\ref{prop:eigvalcurv}. First, we show that $R_0$ 
has eigenvalues as given in part $(1)$ and $(2)$ of the proposition:

$(1)$ Suppose $u$ and $v$ are eigenvectors of $A$ for real eigenvalues $\lambda_i$ and $\lambda_j$ respectively, $\lambda_i\neq \lambda_j$. 
Then \eqref{eq:R0ongenerators} becomes equal to
$$
R_0(u\wedge_J v)=\left(\sum_{k=1}^m a_k\sum_{r=0}^{k-1}\lambda_i^r\lambda_j^{k-1-r} \right)u\wedge_J  v
=\left(\sum_{k=1}^m a_k\frac{\lambda_i^k-\lambda_j^k}{\lambda_i-\lambda_j}\right)
u\wedge_J  v=\frac{p(\lambda_i)-p(\lambda_j)}{\lambda_i-\lambda_j}u\wedge_J  v.
$$
Hence, 
$$
\frac{p(\lambda_i)-p(\lambda_j)}{\lambda_i-\lambda_j}
$$
is an eigenvalue of $R_0$ with eigenvector $u\wedge_J  v$.

$(2)$ Let us first argue, that without loss of generality we can suppose that a fixed real eigenvalue $\lambda_i$ of $A$ 
is equal to zero. Indeed, using $\tilde A=A-\lambda_i{\cdot}\Id$ instead of $A$ in \eqref{eq:LinAlg}, the equation \eqref{eq:LinAlg} 
holds for $R_0$, $\tilde A$ and the same $B=\tilde p(\tilde A)$ for another polynomial 
$\tilde p(t)=\sum_{k=0}^{m}\tilde a_k t^k$ that is equal to $p(t+\lambda_i)$. Since $\tilde p'(0)=p'(\lambda_i)$, 
we may assume that the fixed eigenvalue $\lambda_i$ under consideration is equal to zero. Suppose $A$ 
has a non-trivial $\lambda_i$-Jordan block and denote by $V_{\lambda_i}$ the generalised $\lambda_i$-eigenspace. 
Let $u\in V_{\lambda_i}$ be an eigenvector of $A$, i.e. $Au=0$, and $v\in V_{\lambda_i}$ satisfy $A v=u$. 
Then \eqref{eq:R0ongenerators} becomes
$$
R_0(u\wedge_J v)=\sum_{k=1}^m  a_k(u\wedge_J A^{k-1} v)= a_1 u\wedge_J  v
$$
Thus, $a_1$ is an eigenvalue of $R_0$ with eigenvector $u\wedge_J  v$. Since $ a_1= p'(0)$, 
the eigenvalue is as in part $(2)$ of Proposition~\ref{prop:eigvalcurv}.

To summarize, we have shown that $R_0$ has eigenvalues as given in  part $(1)$ and $(2)$ of the proposition. 
It remains to show that these eigenvalues are also eigenvalues for the quotient map $\tilde R_0$. 
Since for an arbitrary operator $R$ solving \eqref{eq:LinAlg} we have $\tilde R=\tilde R_0$, 
we then obtain that $\tilde R$ and hence $R$, has eigenvalues as in part $(1)$ and $(2)$ of the proposition.

It is straightforward to show that for any operator $\varphi:V\rightarrow V$  with a $\varphi$-invariant 
subspace $U\subseteq V$, an eigenvalue $\lambda$ of $\varphi$ is also an eigenvalue of the quotient map
$$
\tilde\varphi:V/U\rightarrow V/U
$$
if and only if the generalised $\lambda$-eigenspace of $\varphi$ is not contained in $U$ 
(although, it may have a non-trivial intersection with $U$). 

To complete the proof of Proposition~\ref{prop:eigvalcurv}, it therefore suffices to show the following statements, 
each of which proves one of the parts of the proposition:

$(1)$ $u\wedge_J v\notin \mathfrak{g}_A$ for eigenvectors $u$ and $v$ of $A$ 
corresponding to real eigenvalues $\lambda_i$ and $\lambda_j$ respectively, $\lambda_i\neq \lambda_j$.

$(2)$ $u\wedge_J v\notin \mathfrak{g}_A$ for an eigenvector $u$ of $A$ corresponding to a 
real eigenvalue $\lambda_i$ and a vector $v\in V_{\lambda_i}$ such that $Av=u+\lambda_i v$.

Introducing the notation
$$
u\odot_J v=u^\flat\otimes v+v^\flat\otimes u+(Ju)^\flat\otimes Jv+(Jv)^\flat\otimes Ju,
$$
we have $[u\wedge_J v,A]=Au\odot_J v-u\odot_J Av$. Thus, for case $(1)$ we obtain
$$
[u\wedge_J v,A]=(\lambda_i-\lambda_j) u\odot_J v
$$
which is non-zero, hence, $u\wedge_J v\notin \mathfrak{g}_A$. For case $(2)$ we obtain
$$
[u\wedge_J v,A]=\lambda_i u\odot_J v-u\odot_J (u+\lambda_i v)=-u\odot_J u
$$
which is non-zero, hence, $u\wedge_J v\notin \mathfrak{g}_A$. This finishes the proof of the proposition.
\end{proof}

\end{document}